\documentclass[12pt]{article}

\usepackage{amssymb,amsfonts,mathrsfs,amsmath,amsthm,mathtools,bm,dsfont, thmtools,bbm,mathabx}\allowdisplaybreaks
\usepackage[dvipsnames]{xcolor}
\usepackage{array, graphicx, graphics,float,multirow,tabularx,tikz,pgfplots, ctable, longtable,threeparttable}
\usepackage[top=1.25in, bottom=1.25in, left=1.25in, right=1.25in]{geometry}
\usepackage[colorlinks,citecolor=blue,urlcolor=blue]{hyperref}
\usepackage{bibentry, natbib} 
\usepackage{palatino} 
\usepackage{paralist}  
\usepackage{appendix} 
\usepackage{setspace}
\usepackage[hang,flushmargin]{footmisc} 
\usepackage{mathrsfs}
\usepackage{appendix}
\usepackage{enumitem}
\usepackage{caption}
\usepackage{titlesec}
\usepackage{textcase,relsize}
\usepackage{xr}
\usepackage{datetime}

\newtheoremstyle{named-normal}{}{}{\normalfont}{}{\bfseries}{.}{.5em}{#1 (\thmnote{#3})}
\newtheoremstyle{named-italic}{}{}{\itshape}{}{\bfseries}{.}{.5em}{#1 (\thmnote{#3})}
\declaretheorem[style=named-normal,numbered=yes,name=Assumption]{named-asmp}
\declaretheorem[style=named-normal,numbered=yes,name=Condition]{named-cond}
\declaretheorem[style=named-normal,numbered=yes,name=Example]{named-exmp}

\newdateformat{mydate}{\monthname[\THEMONTH] \THEYEAR}

	\def \calB {\mathcal{B}}		
			
\def \bbD {\mathbb{D}}			
\def \bbE {\mathbb{E}}

	\def \calJ {\mathcal{J}}		
	\def \calK {\mathcal{K}}		
\def \bbL {\mathbb{L}}		\def \scrL {\mathscr{L}}	
			
	\def \calN {\mathcal{N}}		
			
\def \bbP {\mathbb{P}}			
\def \bbQ {\mathbb{Q}}			
\def \bbR {\mathbb{R}}

	\def \calX {\mathcal{X}}

\def \bbOne{\mathbbm{1}}

			\def \whatG {\widehat{G}}

\def \tilK {\tilde{K}}

		\def \hatV {\hat{V}}	\def \whatV {\widehat{V}}

\def \tilg {\tilde{g}}

		\def \hatgamma {\hat{\gamma}}	\def \whatgamma {\widehat{\gamma}}	\def\bargamma{\bar{\gamma}}

		\def \hattheta {\hat{\theta}}	\def \whattheta {\widehat{\theta}}

		\def \hatnu {\hat{\nu}}		
				\def\barxi{\bar{\xi}}

\def \oP {o_{\bbP}}
\def \OP {O_{\bbP}}

\def \Var {\text{Var}}

\def \conP {\overset{\bbP}\longrightarrow}

\def \conL {\overset{\scrL}\longrightarrow}

\def \oPstar {o_{\bbP^\ast}}

\def \hatV {\widehat{V}}

\def \vect {\textrm{vec}}

\DeclareMathOperator*{\argmax}{arg\,max}

\DeclareMathAlphabet{\mathpzc}{OT1}{pzc}{m}{it}

\numberwithin{equation}{section}

\theoremstyle{definition}
\newtheorem{remark}{Remark}

\newtheorem{assumption}{Assumption}

\theoremstyle{plain}
\newtheorem{theorem}{Theorem}

\newtheorem{lemma}[theorem]{Lemma}

\newtheorem{algorithm}{Algorithm}

\def \Bnl {\calB_n^{\texttt{NL}}}
\def \Banb {\calB_n^{\texttt{ANB}}}
\def \BanbOne {\calB_{n,1}^{\texttt{ANB}}}
\def \BanbTwo {\calB_{n,2}^{\texttt{ANB}}}

\def \whatBnl {\widehat{\calB}_n^{\,\texttt{NL}}}
\def \whatBanb {\widehat{\calB}_n^{\,\texttt{ANB}}}
\def \cond {\,\big|\,}

\colorlet{Mycolor1}{RedOrange}
\colorlet{Mycolor2}{ForestGreen}

\title{Bias Correction and Robust Inference in Semiparametric Models}
\author{Jungjun Choi and Xiye Yang \\[0.5em] Department of Economics\\ Rutgers University}
\date{First draft: July 2019 \\ This version: \mydate{\today}}

\begin{document}

\maketitle
\setlength{\baselineskip}{1.1\baselineskip}
\hyphenpenalty=5000 \tolerance=1000

\begin{abstract}
This paper analyzes several different biases that emerge from the (possibly) low-precision nonparametric ingredient in a semiparametric model. We show that both the variance part and the bias part of the nonparametric ingredient can lead to some biases in the semiparametric estimator, under conditions weaker than typically required in the literature. We then propose two bias-robust inference procedures, based on multi-scale jackknife and analytical bias correction, respectively. We also extend our framework to the case where the semiparametric estimator is constructed by some discontinuous functionals of the nonparametric ingredient. The simulation study shows that both bias-correction methods have good finite-sample performance.
\medskip

\end{abstract}

\bibliographystyle{ecta}

\section{Introduction} \label{sec:into}

Recently, increasing attention has been drawn to the interplay between the asymptotic properties of semiparametric estimators and their nonparametric ingredients that could have relatively low precision (e.g., the nonparametric ingredient can have a slower-than-$n^{1/4}$ convergence rate), which may render the previously established asymptotic results invalid. Significant progress has been made by one branch of literature \citep{Cattaneo&Crump&Jansson:2010, Catteneo&Crump&Jansson:2013, Cattaneo&Crump&Jansson:2014, CCT:2014, Cattaneo&Jansson:2018} about \textquotedblleft small bandwidth asymptotics\textquotedblright{} for kernel-based semiparametric estimators and establishes bootstrap inference procedure robust to a bias that has non-negligible impacts when the bandwidth is \textquotedblleft small.\textquotedblright{} Another branch of literature \citep{Ichimura&Newey:2017, CCDDHN:2017double, CCDDHN:2018double, CNR:2018, CEINR:2018} has creatively introduced an influence function to the GMM semiparametric two-step estimator, to ensure local robustness to the first-step nonparametric ingredient, a property which, as pointed out by \citep{Cattaneo&Jansson:2018}, can be interpreted as \textquotedblleft large bandwidth asymptotics\textquotedblright{} in the case of kernel-based semiparametric estimators. 

Motivated by these new results, this paper proposes a general framework to analyze the impacts of several different biases that emerge from the low-precision nonparametric ingredient, including kernel and sieve estimators, on the distributional approximations of the associated semiparametric estimator. We generalize the framework used by \citep{Andrews:1994}, \citep{Newey:1994}, and \citep{Newey&McFadden:1994}, by allowing the nonparametric ingredient to have a convergence rate slower than what is required by the original papers (i.e., a faster-than-$n^{1/4}$ convergence rate). In short, we consider the case where the key Condition (2.8) in \citep{Andrews:1994} fails to hold. More specifically, we first replace the linear approximation (Assumption 5.1 in \citep{Newey:1994} and Condition (i) of Theorem 8.1 in \citep{Newey&McFadden:1994}) in the last two cited papers by a quadratic one. Although this requires a higher-order differentiability condition, it enables us to account for a nonlinear bias, which may appear when the nonparametric ingredient converges slower than $n^{1/4}$. Second, we also relax a restriction jointly implied by the stochastic equicontinuity condition and the mean-square continuity condition (Assumptions 5.2 and 5.3 in \citep{Newey:1994}, and Conditions (ii) and (iii) of Theorem 8.1 in \citep{Newey&McFadden:1994}), to account for another \textquotedblleft linear\textquotedblright{} bias (see Remarks \ref{rem:SE} and \ref{rem:MSC} below). Both biases can have non-negligible (in the sense of not being $\oP(n^{-1/2})$) impacts on the distributional approximation of the semiparametric estimator. 

As for the sources of the above biases, recall the well-known bias-variance tradeoff in the nonparametric literature. Our analysis shows that the nonlinear bias is related to the variance part of the nonparametric ingredient, while the other bias comes from the nonparametric bias. Theoretically speaking, it is possible to impose certain restriction(s) on the tuning parameter of the nonparametric ingredient so that one bias becomes $\oP(n^{-1/2})$ (e.g., under- or over-smoothing in the kernel case), just like the above-cited recent literature. However, it is often hard to verify such restriction(s) in practice. Besides, even though one bias could be $\oP(n^{-1/2})$ in an asymptotic sense, its effects may not be sufficiently small to be negligible in finite or small samples. Therefore, we do not impose such restriction(s) and allow the possibility that either one or both of them could be larger than $\oP(n^{-1/2})$. By doing so, our distributional approximation will be robust to a larger range of values of the tuning parameter. When specialized to the kernel-based case, this is equivalent to establishing asymptotic results without distinguishing small and large bandwidths. Consequently, the finite sample performance of the corresponding inference procedures will be less sensitive to the choice of the tuning parameter. 

In addition to the above two biases that appear in general cases, our analysis also indicates that there can be another special bias for the kernel-based semiparametric estimators. We refer to it as the \textquotedblleft singularity bias,\textquotedblright{} which, in our view, is the same as the \textquotedblleft leave-in bias\textquotedblright{} studied by \citep{Cattaneo&Jansson:2018}. In the cited paper, the \textquotedblleft leave-in bias\textquotedblright{} highlights the fundamental difference between the asymptotic separability condition and the stochastic equicontinuity condition therein (see Remark \ref{rem:SE} for more discussions). Since the framework we adopted is somewhat different, we discuss the \textquotedblleft singularity bias\textquotedblright{} mainly from the perspectives of U-statistics and V-statistics. If we use the same empirical measure to construct the nonparametric and the semiparametric estimators, then the first-order term in our quadratic approximation is a V-statistic. In contrast, if we either use the \textquotedblleft leave-one-out\textquotedblright{} version of the empirical measure to construct the nonparametric estimator, or use a smoothed measure to construct the semiparametric estimator, then the first-order term becomes a U-statistic. Typically, the difference between a V-statistic and its corresponding U-statistic is very small, often of order $\OP(n^{-1})$. However, the special structure (we believe it is the convolution structure that matters here) of the kernel-based nonparametric estimator can lead to a potentially much larger difference, yielding this special bias. As a comparison, there is no such bias in the sieve-based case. 

The second main result of this paper is that we propose two different inference procedures that are robust to the aforementioned biases. 
The first one is the multi-scale jackknife (MSJ) method, which utilizes the tuning parameter of the nonparametric ingredient in the role of sample size as in the original jackknife method introduced by \citep{Quenouille:1949AMS}. Similar ideas have been adopted by, for example, \citep{Schucany&Sommers:1977}, \citep{Bierens:1987kernel}, \citep{PSS:1989}, and \citep{Li&Liu&Xiu:2019}. Theoretically speaking, this method can remove \textit{all} aforementioned biases, provided that an appropriate weighting scheme is chosen. In the kernel-based case, this method can automatically remove the \textquotedblleft singularity bias,\textquotedblright{} for that it has the same order as the nonlinear bias. If one knows the orders of other smaller biases, one can use more scales to remove these biases as well (refer to the simulation results). The second one is the analytical-based bias correction (ABC) method. It requires a twice Fr\'{e}chet differentiable assumption (so that one can get the analytical form of the nonlinear bias) and some consistent estimators of both the variance part and the bias part of the nonparametric ingredient. Provided that some other regularity conditions are satisfied, this method can remove or reduce those biases (the remaining bias, if any, will be negligible at a root-$n$ rate).

Last but not least, we show that our framework can be extended to the family of semiparametric estimators that are constructed from discontinuous functionals of the nonparametric ingredients. The requirement is that those discontinuous functionals must have smooth projections, which can be well approximated by quadratic functionals of the nonparametric ingredients. Under certain regularity conditions, the multi-scale jackknife method can yield valid and robust inference. However, the analytical bias correction in this case is more involved, for that one needs to take into account the estimation error and/or bias associated with the unknown smooth projection. Hence, we leave this to future exploration.

The rest of this paper is organized as follows. Section \ref{sec:semiparametric} discusses several key properties of a general class of semiparametric estimators and present our first main result, i.e., a distributional approximation that accounts for various biases. In Section \ref{sec:inference}, we present two inference procedures that are robust to those biases and provide some sufficient conditions to extend the results from the class of twice differentiable functionals to certain discontinuous functionals. Section \ref{sec:simulation} demonstrates the finite sample performance of the two inference procedures through some simulation results. Section \ref{sec:conclusion} concludes. 

\section{Asymptotically Linear Semiparametric Estimators} \label{sec:semiparametric}

Throughout this paper, any random sequence that is $\oP(n^{-1/2})$ will be referred to as \textquotedblleft root-$n$ negligible.\textquotedblright{} We will use $C$ to denote some finite positive number, the value of which may change from line to line. Denote by $\| \|$ the Euclidean norm. 

\subsection{Asymptotic linearity} \label{subsec:linearity}

Let $\theta_0 \in \Theta$ be a finite-dimensional parameter of interest, where $\Theta$ is a subset of some Euclidean space. Suppose that the identification of $\theta_0$ depends on an unknown function $\gamma_0\in \Gamma$, where $\Gamma$ represents certain infinite-dimensional functional space. Let $z_1,\cdots,z_n$ be an \text{i.i.d.} copies of a random vector $z\in\bbR^{d_z}$. We shall use $x$ to denote a real vector in $\bbR^{d_z}$. Suppose that we can sequentially construct two consistent estimators $\hatgamma_n$ and $\hattheta_n$ from this sample. 

Let $\bbP$ and $\bbP_n$ be the true probability measure and the empirical probability measure, respectively. For any signed measure $\bbQ$, let $\bbQ f \coloneqq \int f d\bbQ$ for any function $f$. Then for any functional $g$ of $(z,\theta,\gamma)$, define
\begin{align*}
	G(\theta, \gamma) \coloneqq \bbP g = \bbE[g(z,\theta,\gamma)] \quad\text{and}\quad \whatG_n(\theta,\gamma) \coloneqq \bbP_n g = \frac{1}{n} \sum_{i=1}^n g(z_i, \theta, \gamma).
\end{align*}
Here the notation $g(z_i,\theta,\gamma)$ is to stress that the moment function is evaluated at the sample point $z_i$ under the empirical measure. The functional $g$ can directly and/or indirectly (i.e., through $\gamma$) depend on $z_i$.

\begin{assumption}[AL---Asymptotic Linearity in $g$] \label{asmp:AL} Assume that the estimator $\hattheta_n$ is \textit{asymptotically linear}. That is, 
\begin{align} \label{eq:AL}
	\hattheta_n - \theta_0 = \calJ_n \whatG_n(\theta_0, \hatgamma_n) + \oP(n^{-1/2}) = \frac{1}{n} \sum_{i=1}^n \calJ_n g(z_i, \theta_0, \hatgamma_n) +  \oP(n^{-1/2}),
\end{align}
where $\calJ_n \conP \calJ_0$ for some non-random, finite and non-degenerate $\calJ_0$ (when it is a matrix, all of its eigenvalues are finite and bounded below from zero), and the functional $g$ satisfies that $G(\theta_0,\gamma_0) = \bbE[g(z,\theta_0,\gamma_0)]=0$, which uniquely determines $\theta_0$.  
\end{assumption}

\begin{remark} \label{rem:Pac}
Another way to formulate $\whatG_n$ is to use an estimated probability measure, which is absolutely continuous with respect to the Lebesgue measure. Denote such a measure by $\bbP_n^{\texttt{AC}}$. For instance, it can be obtained by using a kernel-based method. Now consider the case of estimating the average density $\theta_0 = \bbE[ \gamma_0(z) ]$, which implies that $g(z,\theta,\gamma) = \gamma(z) - \theta$. We can then have two different formulations for $\hattheta_n - \theta_0$: one for the average density estimator $\hattheta_n^{\texttt{AD}}$:
\begin{align*}
	\hattheta_n^{\texttt{AD}} - \theta_0 = \whatG_n(\theta_0,\hatgamma_n) = \bbP_n g = \frac{1}{n} \sum_{i=1}^n \big( \hatgamma_n(z_i) - \theta_0 \big),
\end{align*}
and the other one for the integrated squared density estimator $\hattheta_n^{\texttt{ISD}}$ (recall that $x$ is a real vector):
\begin{align*}
	\hattheta_n^{\texttt{ISD}} - \theta_0 = \whatG_n(\theta_0,\hatgamma_n) = \bbP_n^{\texttt{AC}} g = \int \hatgamma_n^2(x) dx - \theta_0.
\end{align*}
In both cases, $\calJ_n = \calJ_0 = I$. 
\end{remark}

\begin{remark}
The requirement on $\calJ_0$ excludes the possibility of weak identification of $\theta$. This may seem to be restrictive. However, we are going to extend the classic theory in a different direction. 

As pointed out by \citep{IAndrews:2016Functional}, the empirical process theory typically implies that the root-$n$ re-scaled sample moment function converges in distribution to the sum of three parts (refer to Equation (1) therein): a mean function, which may allow for various types of identification; a mean-zero Gaussian process, which establishes the central limit theorem; and a residual term, which is typically assumed to be negligible at the root-$n$ rate. While we assume the mean function gives strong identification of $\theta$, we are going to relax the assumption on the residual term and allow it to be non-negligible at the root-$n$ rate.
\end{remark}

We note that $\calJ_n g(z_i,\theta_0,\hatgamma_n)$ gives the influence of a single observation in the leading term of the estimation error $\hattheta_n - \theta_0$. In this sense, it can be viewed as the influence function, following \citep{Hampel:1974}. \citep{Ichimura&Newey:2017} adopt a very similar definition of asymptotic linearity in their equation (2.1). The only difference is that we introduce the term $\calJ_n$, in order to focus on the more essential part $g$ of the influence function. As pointed out by \citep{Ichimura&Newey:2017}, under sufficient regularity conditions, almost all root-$n$ consistent semiparametric estimators satisfy Assumption \ref{asmp:AL}.

\begin{named-exmp}[GMM Semiparametric Estimator]
Consider a GMM-type estimator $\hattheta_n$:
\begin{align*}
	\hattheta_n \coloneqq \argmax_{\theta \in \Theta} - \frac{1}{2} \whatG_n(\theta, \hatgamma_n)^\intercal W_n \whatG_n(\theta, \hatgamma_n),
\end{align*}
where $W_n \conP W_0$, representing the weighting matrix and its limit. Suppose that $g$ is first-order differentiable at $\theta_0$, then one can readily get
\begin{gather*}
	\calJ_n = [ \partial_\theta \whatG_n(\theta_0,\hatgamma_n)^\intercal W_n \partial_\theta \whatG_n(\theta_0,\hatgamma_n)]^{-1} \partial_\theta \whatG_n(\theta_0,\hatgamma_n)^\intercal W_n, \\
	\calJ_0 = [ \partial_\theta G(\theta_0,\gamma_0)^\intercal W_0 \partial_\theta G(\theta_0,\gamma_0)]^{-1} \partial_\theta G(\theta_0,\gamma_0)^\intercal W_0.
\end{gather*}
We have $\calJ_n \conP \calJ_0$, if $\partial_\theta g(\theta_0, \gamma)$ is continuous with respect to $\gamma$ in a neighborhood of $\gamma_0$. 
\end{named-exmp}

The above example shows a subtle difference in the definition of asymptotic linearity between this paper and those in \citep{Ichimura&Newey:2017} and \citep{Cattaneo&Jansson:2018}. In this paper, the term $\calJ_n$ can be random, hence can be different from $\calJ_0$ in a non-trivial way. However, in the GMM examples of the two cited papers, the authors set $\calJ_n \equiv \calJ_0$ (cf. (2.2) in \citep{Ichimura&Newey:2017} and the discussion following Condition AL in \citep{Cattaneo&Jansson:2018}). It is easy to see that if the following condition holds
\begin{align} \label{eq:calJ}
	(\calJ_n - \calJ_0) \whatG_n(\theta_0,\hatgamma_n) = \oP(n^{-1/2}),
\end{align}
then the above definition can be modified to be exactly the same as the two cited papers. A sufficient condition for \eqref{eq:calJ} is $\whatG_n(\theta_0,\hatgamma_n) = \OP(n^{-1/2})$, which indeed holds in a lots of applications. This sufficient condition may not hold in the current paper, since we are going to consider the general case where $\whatG_n(\theta_0,\hatgamma_n)$ could have some bias(es) that can be larger than $\OP(n^{-1/2})$ in order. However, eventually, we will make sure that Condition \eqref{eq:calJ} is satisfied (see Lemma \ref{lem:op1} for details).

\subsection{Quadratic approximation of $\whatG_n(\theta_0,\hatgamma_n)$} \label{subsec:quadraticity}

To begin with, we have the following decomposition (recall that $G(\theta_0,\gamma_0)=0$)
\begin{align*}
	\whatG_n(\theta_0,\hatgamma_n) = \whatG_n(\theta_0,\hatgamma_n) - \whatG_n(\theta_0,\gamma_0) + \whatG_n(\theta_0,\gamma_0) - G(\theta_0,\gamma_0).
\end{align*}
The first difference is the impact of replacing $\gamma_0$ by its estimator in the empirical moment condition, while the second one is the difference between a sample average and its expectation, to which we can apply the central limit theorem (CLT) for \text{i.i.d.} random variables. 

We introduce the following assumption on $g$, in order to get a more detailed evaluation of the first term. 

\begin{assumption}[Quadraticity] \label{asmp:quadraticity} 
Suppose that the following (stochastic) quadratic approximation of the functional $g$ holds around $(\theta_0,\gamma_0)$ for sufficiently large $n$:
\begin{align*}
	g(z_i,\theta_0,\hatgamma_n) =\,& g(z_i,\theta_0,\gamma_0) + g_{\gamma}^{\,\prime}(z_i,\theta_0,\gamma_0,\hatgamma_n-\gamma_0) + \frac{1}{2} g_{\gamma\gamma}^{\,\prime\prime}(z_i,\theta_0,\gamma_0,\hatgamma_n-\gamma_0, \hatgamma_n-\gamma_0) \\
		& + g_R(z_i, \theta_0, \gamma_0, \hatgamma_n-\gamma_0),
\end{align*} 
where $g_{\gamma}^{\,\prime}(z_i,\theta_0,\gamma_0,\cdot)$ is a linear functional, $g_{\gamma\gamma}^{\,\prime\prime}(z_i,\theta_0,\gamma_0,\cdot,\cdot)$ is a bi-linear functional and symmetric in its two inputs (the subscript $\gamma$ indicates that these functionals are from the expansion with respect to $\gamma$, not $z$ or $\theta$), and the functional $g_R$ captures the remainder of this expansion. We assume that $\bbE[\| g_\gamma^\prime(z_i,\theta_0,\gamma_0, \gamma - \gamma_0) \|] \leq C \, \bbE[ \| \gamma(z_i) - \gamma(z_i) \| ]$, $\bbE[\| g_{\gamma\gamma}^{\prime\prime}(z_i,\theta_0,\gamma_0, \gamma - \gamma_0, \gamma - \gamma_0) \|] \leq C \, \bbE[ \| \gamma(z_i) - \gamma(z_i) \|^2 ]$, and $\bbE[ \| g_R(z_i, \theta_0, \gamma_0, \gamma -\gamma_0) \| ] \leq C \, \bbE[ \| \gamma(z_i) - \gamma_0(z_i) \|^3 ]$ for $\gamma$ sufficiently close to $\gamma_0$ and some finite number $C$.
\end{assumption}

Compared to Assumption 5.1 (Linearization) in \citep{Newey:1994} and Condition (i) of Theorem 8.1 in \citep{Newey&McFadden:1994}, the above assumption requires a second-order, instead of first-order, differentiability of $g$ with respect to $\gamma$, which could be a random function, such as $\hatgamma_n$. However, the two cited papers both require that $ \| \hatgamma_n(z_i) - \gamma_0(z_i) \|^2 = \oP(n^{-1/2}) $. In other words, the nonparametric estimator $\hatgamma_n$ must have a faster-than-$n^{1/4}$ convergence rate (i.e., $r>1/4$ and $s>1/4$ in Assumption \ref{asmp:BO} below). Yet, as to be shown later, we just need $ \| \hatgamma_n(z_i) - \gamma_0(z_i) \|^3 = \oP(n^{-1/2})$, which only requires a faster-than-$n^{1/6}$ convergence rate for $\hatgamma_n$. With this slower convergence rate, we may have some non-root-$n$-negligible biases.

Define the following terms using the empirical measure $\bbP_n$:
\begin{align*}
	\whatG_{n,\gamma}^{\,\prime}(\theta_0,\gamma_0,\eta) &\coloneqq \frac{1}{n} \sum_{i=1}^n g_{\gamma}^{\,\prime}(z_i,\theta_0,\gamma_0,\eta), \\
	\whatG_{n,\gamma\gamma}^{\,\prime\prime}(\theta_0,\gamma_0,\eta,\phi) &\coloneqq \frac{1}{n} \sum_{i=1}^n g_{\gamma\gamma}^{\,\prime\prime}(z_i,\theta_0,\gamma_0,\eta,\phi).
\end{align*}
The quadraticity assumption implies that, for sufficiently large $n$, we have
\begin{align*}
	\whatG_n(\theta_0,\hatgamma_n) =\,& \whatG_n(\theta_0,\gamma_0) + \whatG_{n,\gamma}^{\,\prime}(\theta_0,\gamma_0, \hatgamma_n - \gamma_0) + \frac{1}{2} \whatG_{n,\gamma\gamma}^{\,\prime\prime}(\theta_0,\gamma_0,\hatgamma_n - \gamma_0, \hatgamma_n - \gamma_0) \\
		& + \whatG_{n,R}(\theta_0,\gamma_0, \hatgamma_n - \gamma_0),
\end{align*}
where $\whatG_{n,R}(\theta_0,\gamma_0, \hatgamma_n - \gamma_0) = \frac{1}{n}\sum_{i=1}^n g_R(z_i,\theta_0,\gamma_0,\hatgamma_n-\gamma_0)$. 

\begin{remark} \label{rem:PAC}
In the case where we use the measure $\bbP_n^{\text{AC}}$, instead of $\bbP_n$, to construct $\whatG$, we apply Assumption \ref{asmp:quadraticity} to an equivalent functional $\tilg$, which will be evaluated at a real vector $x$, defined as follows. Let $\bbL$ be the Lebesgue measure, $\nu_0$ be the true density function of $z$, which may or may not be part of $\gamma_0$, and $\hatnu_n = d\bbP_n^{\text{AC}}/d\bbL$. Then we have $\bbP g =\bbE[g] = \bbL [g(\cdot,\theta_0, \gamma_0)\nu_0(\cdot)]$ and $\bbP_n^{\texttt{AC}} g = \bbL ( g(\cdot,\theta_0, \hatgamma_n) \hatnu_n(\cdot) )$. Hence, we set $\tilg(\theta, \gamma, \nu) \coloneqq \bbL [g(\cdot,\theta, \gamma) \nu(\cdot)]$. In the special case where $\nu_0$ is part of $\gamma_0$, we can write $\tilg(\theta, \gamma, \nu)$ as $\tilg(\theta,\gamma)$. In the end, we suppose that Assumption \ref{asmp:quadraticity} holds true for the functional $\tilg$ with respect to $(\gamma,\nu)$ around $(\gamma_0, \nu_0)$. 

\end{remark}

Throughout this paper, we assume that $\hatgamma_n$ is a consistent estimator of the unknown function $\gamma_0$. Yet, such a nonparametric estimator is often biased, leading to the well-known bias-variance tradeoff in the nonparametric literature. In the semiparametric literature, it is often assumed that the nonparametric bias is sufficiently small so that this bias is root-$n$ negligible, causing no problems for the associated semiparametric estimator (that is, $G_{\gamma}^{\,\prime}(\theta_0,\gamma_0, \hatgamma_n - \gamma_0) \coloneqq \bbE[ \whatG_{n,\gamma}^{\,\prime}(\theta_0, \gamma_0, \hatgamma_n - \gamma_0) ] = \oP(n^{-1/2})$). Since we aim at relaxing such an assumption, we are going to separate the bias part from the variance part. The idea is to introduce a function $\bargamma_n$ such that $G_{\gamma}^{\,\prime}(\theta_0,\gamma_0, \hatgamma_n - \bargamma_n) \coloneqq \bbE[ \whatG_{n,\gamma}^{\,\prime}(\theta_0, \gamma_0, \hatgamma_n - \bargamma_n) ]$ is identically zero or at least $\oP(n^{-1/2})$, no matter how one chooses the tuning parameter. Then we obtain a more detailed decomposition:
\begin{align}
\begin{split} \label{eq:decG}
	\whatG_n(\theta_0,\hatgamma_n) =\,&  \whatG_n(\theta_0,\gamma_0) + \whatG_{n,\gamma}^{\,\prime}(\theta_0,\gamma_0,\hatgamma_n - \bargamma_n) +  \whatG_{n,\gamma}^{\,\prime}(\theta_0,\gamma_0,\bargamma_n - \gamma_0) \\
		& + \frac{1}{2} \whatG_{n,\gamma\gamma}^{\,\prime\prime}(\theta_0,\gamma_0,\hatgamma_n - \bargamma_n, \hatgamma_n - \bargamma_n) + \whatG_{n,\gamma\gamma}^{\,\prime\prime}(\theta_0,\gamma_0,\hatgamma_n - \bargamma_n, \bargamma_n - \gamma_0) \\
		&+ \frac{1}{2} \whatG_{n,\gamma\gamma}^{\,\prime\prime}(\theta_0,\gamma_0,\bargamma_n - \gamma_0, \bargamma_n - \gamma_0) + \whatG_{n,R}(\theta_0,\gamma_0, \hatgamma_n - \gamma_0).
\end{split}
\end{align}
Here, we would expect to establish a central limit theorem for the sum of the first two terms. The third and fourth terms are the two main biases that we are going to analyze. Intuitively, we may defined $\bargamma_n$ as $\bargamma_n \coloneqq \bbE[ \hatgamma_n ]$. However, this may not necessarily lead to the desired result. Instead, we are going to use the definition $\bargamma_n(z_i) \coloneqq \bbE[ \hatgamma_n(z_i) | z_i ]$, especially when there is a \textquotedblleft singularity bias.\textquotedblright{}

\subsection{V-statistic and U-statistic} \label{subsec:VU}

To begin with, consider the case where we also use the empirical measure $\bbP_n$ to construct $\hatgamma_n$. Without much loss of generality, suppose that there exists some function $\psi$ such that $\hatgamma_n(\cdot) = \bbP_n \psi(\cdot) = \frac{1}{n} \sum_{j=1}^n \psi(\cdot, z_j) $ (\citep{Newey&McFadden:1994} adopt a similar representation in Section 8 therein). Moreover, it is reasonable to assume that  $g_{\gamma}^{\,\prime}(z_i,\theta_0,\gamma_0,\hatgamma_n)$ can be reduced to $g_{\gamma}^{\,\prime}\big(z_i,\theta_0,\gamma_0,\hatgamma_n(z_i) \big) $. Consequently, the linearity of  $g_{\gamma}^{\,\prime}(z,\theta_0,\gamma_0, \cdot)$ implies that
\begin{align*}
    & \whatG_{n,\gamma}^{\,\prime}(\theta_0,\gamma_0,\hatgamma_n ) = \frac{1}{n} \sum_{i=1}^n g_{\gamma}^{\,\prime}\big(z_i,\theta_0,\gamma_0,\hatgamma_n(z_i) \big) \\
    =\,& \frac{1}{n} \sum_{i=1}^n g_{\gamma}^{\,\prime}\big(z_i,\theta_0,\gamma_0, \frac{1}{n} \sum_{i=1}^n \psi(z_i, z_j) \big) 
    = \frac{1}{n^2} \sum_{i,j=1 }^n g_{\gamma}^{\,\prime}\big(z_i,\theta_0,\gamma_0, \psi(z_i, z_j) \big) \\
    =\,& \frac{1}{n^2} \sum_{i=1}^n g_{\gamma}^{\,\prime}\big(z_i,\theta_0,\gamma_0, \psi(z_i, z_i) \big) + \frac{1}{n^2} \sum_{i \neq j} g_{\gamma}^{\,\prime}\big(z_i,\theta_0,\gamma_0, \psi(z_i, z_j) \big),
\end{align*}
where the sum $\sum_{i\neq j}$ is taken over $1\leq i, j\leq n$ with $i\neq j$. 

It is then clear that $\whatG_{n,\gamma}^{\,\prime}(\theta_0,\gamma_0,\hatgamma_n )$ is a V-statistic in this case. Typically, the difference between a V-statistic and its corresponding U-statistic is rather small, often of order $\OP(1/n)$. However, as to be shown in the following example of the kernel density estimator, it sometimes can be larger than $\OP(1/n)$, or even $\OP(n^{-1/2})$. The following example highlights the potentially \textquotedblleft large\textquotedblright{} difference between V- and U-statistics, when the nonparametric ingredient has sufficiently low precision.

\begin{named-exmp}[Kernel Density Estimator]
Suppose that $\gamma_0$ is the density function of each $z_i$. Le $K$ be a kernel function with order $m$ and $K_h(\cdot) \coloneqq K(\cdot /h) /h^{d_z} $. The kernel density estimator $\hatgamma_n$ at a real vector $x\in\bbR^{d_z}$ and at a sampling point $z_i$ are given by
\begin{align*}
    \hatgamma_n(x) = \frac{1}{n} \sum_{i=1}^n K_h(x-z_i) \,\,\text{ and }\,\, \hatgamma_n(z_i) = \frac{K(0)}{nh^{d_z}} + \frac{1}{n} \sum_{\substack{j=1 \\ i\neq j}}^n K_h(z_i-z_j),
\end{align*}
respectively. In this case, we have $\psi(x,y) = K_h(x-y)$ (note that the kernel method is closely related to convolution). In the expression of $\hatgamma_n(z_i)$, the term $\psi(z_i,z_i) = K_h(z_i-z_i) = K(0)/(nh^{d_z})$ is non-random. This shows a difference between $\hatgamma_n(x)$ and $\hatgamma_n(z_i)$, which is quite important when $1/(nh^{d_z})$ is not $o(n^{-1/2})$. It is easy to see that $\whatG_{n,\gamma}^{\,\prime}(\theta_0,\gamma_0,\hatgamma_n ) $ becomes
\begin{align*}
      \frac{1}{nh^{d_z}}  \frac{1}{n} \sum_{i=1}^n g_{\gamma}^{\,\prime}\big(z_i,\theta_0,\gamma_0,K(0) \big) + \frac{1}{n^2} \sum_{i\neq j} g_{\gamma}^{\,\prime}\big(z_i,\theta_0,\gamma_0, K_h(z_i - z_j) \big).
\end{align*}
In general, the first term is of order $\OP(1/(nh^{d_z}))$, which may not be root-$n$ negligible. Since it is from $K_h(z_i-z_i)$, which behaves differently from $K_h(z_i-z_j)$ with $j\neq i$, we refer to it as the \textquotedblleft singularity bias\textquotedblright{} (or maybe \textquotedblleft non-smoothing bias\textquotedblright{}).

On the other hand, we have $\bargamma_n(x) = \bbE[\hatgamma_n(x)] = \int K(u) \gamma_0(x-hu) du$. The plug-in definition then leads to $\bargamma_n(z_i) = \int K(u) \gamma_0(z_i - hu)du$. According to the Law of Iterated Expectation, we readily get
\begin{align*}
    G_{\gamma}^{\,\prime}(\theta_0,\gamma_0, \hatgamma_n - \bargamma_n) = \frac{1}{nh^{d_z}} \bbE\big[ \whatG_{n,\gamma}^{\,\prime}\big(\theta_0,\gamma_0,K(0) \big) \big] + O(\frac{1}{n}) = O\big( \frac{1}{nh^{d_z}} \big).
\end{align*}
The sufficient and necessary condition for this term to be root-$n$ negligible is $n^{1/4} = o(\sqrt{nh^{d_z}})$, which is equivalent to a faster-than-$n^{1/4}$ convergence rate for the kernel density estimator $\hatgamma_n$. Since we aim at relaxing this requirement, the above plug-in definition of $\bargamma_n$ does not suit our purpose. 

To address this problem, we can modify the definition of $\bargamma_n$ at sample points $\{z_i\}_{i=1}^n$, which are more important when we use the empirical measure $\bbP_n$ to construct $\whatG_n$. More specifically, we define ($\bargamma_n(x)$ remains the same as above for any real vector $x$)
\begin{align*}
    \bargamma_n(z_i) \coloneqq \bbE[ \hatgamma_n(z_i) | z_i  ] = \frac{1}{n h^{d_z}} K(0) + \frac{n-1}{n} \int K(u) \gamma_0(z_i - hu) du,
\end{align*}
With this modified $\bargamma_n$, we move the \textquotedblleft singularity bias\textquotedblright{} to $\whatG_{n,\gamma}^{\,\prime}(\theta_0,\gamma_0, \bargamma_n - \gamma_0)$. One can check that $G_{\gamma}^{\,\prime}(\theta_0,\gamma_0, \hatgamma_n - \bargamma_n) = \bbE[ \whatG_{n,\gamma}^{\,\prime}(\theta_0,\gamma_0, \hatgamma_n - \bargamma_n) ]=0$.  
\end{named-exmp}

With the modified definition of $\bargamma_n$, we readily get
\begin{align*}
    \whatG_{n,\gamma}^{\,\prime}(\theta_0,\gamma_0,\hatgamma_n - \bargamma_n) &= \frac{1}{n(n-1)}\sum_{i\neq j} g_{\gamma}^{\,\prime}\big(z_i,\theta_0,\gamma_0, \phi(z_i, z_j) \big) \times \big( 1 - \frac{1}{n} \big),
\end{align*}
where $\phi(z_i,z_j) \coloneqq \psi(z_i,z_j) - \bbE[\psi(z_i,z_j) | z_i]$. Its difference with the associated U-statistic is at most $\OP(n^{-1})$, which is always root-$n$ negligible. However, in this case, we may still have the \textquotedblleft singularity bias\textquotedblright{} in $\whatG_{n,\gamma}^{\,\prime}(\theta_0,\gamma_0,\bargamma_n - \gamma_0)$, if $\hatgamma_n$ is a kernel-based estimator. 

\begin{named-exmp}[Sieve Estimator]
Let $z = (Y, X^\intercal)^\intercal$. Consider a conditional mean model for $Y$ and $X$: $\gamma_0(z,\theta) = \bbE[ \rho(Y,\theta) | X ]$. Following the notation used by \citep{Chen:2007}, we denote by $\{p_{0j}(X), j=1,2,\cdots, k_{m,n}\}$ a sequence of known basis functions in the space of square integrable functions. Let $p^{k_{m,n}}(X) = ( p_{01}(X), \cdots, p_{0k_{m,n}}(X ) )^\intercal$ and $P = ( p^{k_{m,n}}(X_1), \cdots, p^{k_{m,n}}(X_n) )^\intercal$. Then the sieve estimator of $\gamma_0$ is given by
\begin{align*}
    \hatgamma_n(z_i, \theta) = \frac{1}{n} \sum_{j=1}^n \rho(Y_j, \theta) p^{k_{m,n}}(X_j)^\intercal \big(  P^\intercal P )^+ p^{k_{m,n}}(X_i) = \frac{1}{n} \sum_{j=1}^n \psi(z_i, z_j),
\end{align*}
where $(P^\intercal P)^+$ is the Moore-Penrose inverse of $P^\intercal P$. In this case, $\psi(z_i, z_i)$ does not lead to a \textquotedblleft singularity bias.\textquotedblright{} 
\end{named-exmp}

The above two examples show that only the kernel-based estimator may suffer from the \textquotedblleft singularity bias\textquotedblright{} problem. In certain cases, such as the average density estimator to be discussed in the next subsection, it might be desirable to remove this bias in advance. As implied by the example of the sieve estimator, one way to get rid of this bias is to use a global nonparametric estimator. Besides, there are two alternative solutions. However, we stress that it is not always necessary to remove the \textquotedblleft singularity bias\textquotedblright{} in advance (see the discussions in Section \ref{subsec:jackknife}).

One (possible) solution is to use the measure $\bbP_n^{\texttt{AC}}$, instead of $\bbP_n$, to construct $\whatG_n$. For simplicity, recall the integrated density estimator $\hattheta_n^{\texttt{ISD}}$. In this case, the linear functional
\begin{align*}
    \whatG_{n,\gamma}^{\,\prime}(\theta_0,\gamma_0,\hatgamma_n ) = 2 \int \gamma_0(x) \hatgamma_n(x) dx = \frac{2}{n} \sum_{i=1}^n \int \gamma_0(x) \psi(x,z_i) dx
\end{align*}
is a U-statistic of degree 1. In general, even when $\nu_0$ is not part of $\gamma_0$ (recall Remark \ref{rem:PAC}), the above functional is still a U-statistic, hence is not subject to the \textquotedblleft singularity bias.\textquotedblright{} Hence, we don't have to make any adjustment to $\bargamma_n$, as we do not evaluate $\hatgamma_n$ at the sample points. However, as to be shown in the next subsection, this solution increases the level of nonlinearity, hence may bring additional nonlinear bias.

Another solution is to replace the above V-statistic by its corresponding U-statistic. In other words, we can use the \textquotedblleft leave-one-out\textquotedblright{} empirical measure $\bbP_n^{\texttt{LOO}}$ to construct the nonparametric estimator $\hatgamma_n$. That is, let $\hatgamma_n(z_i) \coloneqq \bbP_n^{\texttt{LOO}} \psi(z_i, \cdot) = \frac{1}{n-1} \sum_{j=1, j\neq i}^n \psi(z_i, z_j)$. It is then obvious that
\begin{align*}
    \whatG_{n,\gamma}^{\,\prime}(\theta_0,\gamma_0,\hatgamma_n ) = \frac{1}{n(n-1)} \sum_{i\neq j} g_{\gamma}^{\,\prime}\big(z_i,\theta_0,\gamma_0, \psi(z_i, z_j) \big)
\end{align*}
is a U-statistic of degree 2, following the terminology of \citep{Hoeffding:1948}. It then follows that $\hatgamma_n(z_i) - \bargamma_n(z_j) = \frac{1}{n-1} \sum_{j=1,j\neq i}^n \phi(z_i, z_j)$ and
\begin{align*}
    & \whatG_{n,\gamma}^{\,\prime}(\theta_0,\gamma_0,\hatgamma_n - \bargamma_n) = \frac{1}{n(n-1)} \sum_{i\neq j} g_{\gamma}^{\,\prime}\big(z_i,\theta_0,\gamma_0, \phi(z_i, z_j) \big).
\end{align*}
That is, the term $\whatG_{n,\gamma}^{\,\prime}(\theta_0,\gamma_0,\hatgamma_n - \bargamma_n)$ is also a U-statistic of degree 2. In addition, there is no \textquotedblleft singularity bias\textquotedblright{} in $\whatG_{n,\gamma}^{\,\prime}(\theta_0,\gamma_0,\bargamma_n - \gamma_0)$. Moreover, this will not bring any additional nonlinear biases. Hence, we recommend this method whenever it is feasible. 

\begin{remark}[Stochastic Equicontinuity Condition] \label{rem:SE}
\citep{Cattaneo&Jansson:2018} have insightfully observed that, in the kernel-based case, the \textquotedblleft singularity bias\textquotedblright{} is a key in understanding the difference between the stochastic equicontinuity (SE) condition and the asymptotic separability (AS) condition. We note that the AS condition in the cited paper may involve quadratic terms. Below, we offer a different perspective that is only based on the first-order term in the approximation of $g$.

The stochastic equicontinuity condition given in Assumption 5.2 in \citep{Newey:1994} or Condition (ii) in \citep{Newey&McFadden:1994} (the formulation given by \citep{Andrews:1994} is a bit different. So we defer the discussion to Remark \ref{rem:SEAndrews}) can be written as
\begin{align} \label{eq:SE}
    \frac{1}{n} \sum_{i=1}^n \Big( g_{\gamma}^{\,\prime}(z_i,\theta_0,\gamma_0,\hatgamma_n-\gamma_0) - \int g_{\gamma}^{\,\prime}(z,\theta_0,\gamma_0,\hatgamma_n-\gamma_0) dF_0 \Big) = \oP(n^{-1/2}),
\end{align}
where $F_0$ is the true distribution function of $z$. The integral does not involve the \textquotedblleft singularity bias\textquotedblright{} because one evaluates the functional $g_{n,\gamma}'$ at a real vector $x$, not a sample point $z_i$, when calculating the integral. Therefore, when $\hatgamma_n$ is the original kernel density estimator, the \textquotedblleft singularity bias\textquotedblright{} only appears  in the first term. The sample average of the \textquotedblleft singularity bias\textquotedblright{} is  of order $\OP(\frac{1}{nh^{d_z}})$ (if $g$ only depends on $z_i$ through $\gamma$, this becomes $O(\frac{1}{nh^{d_z}})$, which is not $\oP(n^{-1/2})$ when $\hatgamma_n$ does not have a faster-than-$n^{1/4}$ converges rate.

If one uses the \textquotedblleft leave-one-out\textquotedblright{} kernel estimator or a sieve estimator, then there is no \textquotedblleft singularity bias\textquotedblright{} (this might also be achieved by replacing the input $z$ in the integrand by $z_i$). Hence, it might be possible that the above SE condition also holds true with a low precision $\hatgamma_n$. However, as to be shown in Remark \ref{rem:MSC}, the mean-square continuity condition will fail in such case, when the convergence rate of $\hatgamma_n$ is relatively slow.
\end{remark}

As a summary of the above discussion, no matter how we construct $\whatG_n$ and $\hatgamma_n$, we can always find $\bargamma_n$ such that $\whatG_{n,\gamma}^{\,\prime}(\theta_0, \gamma_n, \hatgamma_n - \bargamma_n)$ is a U-statistic, or its difference with a U-statistic is always root-$n$ negligible. Given such a suitable $\bargamma_n$, we are ready to introduce the following assumption on the asymptotic behavior of the sum of the first two terms in \eqref{eq:decG}. 

\begin{assumption}[AN---Asymptotic Normality] \label{asmp:AN}
For some non-random and positive definite $\Sigma_g$, we have
\begin{gather*}
    \sqrt{n} \, \big( \whatG_n(\theta_0,\gamma_0) + \whatG_{n,\gamma}^{\,\prime}(\theta_0,\gamma_0,\hatgamma_n - \bargamma_n) \big) \conL \calN(0, \Sigma_g).
\end{gather*}
\end{assumption}

\begin{remark}
The first two terms in \eqref{eq:decG} have been intensively studied in the literature, mostly under the assumption that all biases are root-$n$ negligible. Recall that
\begin{align*}
    & \whatG_n(\theta_0,\gamma_0) + \whatG_{n,\gamma}^{\,\prime}(\theta_0,\gamma_0,\hatgamma_n - \bargamma_n) = \frac{1}{n} \sum_{i=1}^n \big( g(z_i,\theta_0,\gamma_0) + g_{\gamma}^{\,\prime}(z_i,\theta_0,\gamma_0,\hatgamma_n - \bargamma_n) \big).
\end{align*}
The functionals $g(z,\theta_0,\gamma_0)$ and $g_{\gamma}^{\,\prime}(z,\theta_0,\gamma_0,\hatgamma_n-\bargamma_n)$ are respectively very similar to, for instance, $m(z,h_0)$ and $D(z,h-h_0)$ studied by \citep{Newey:1994}, or $g(z,\gamma_0)$ and $G(z,\gamma-\gamma_0)$ analyzed by \citep{Newey&McFadden:1994}. Note that when all biases are root-$n$ negligible, the terms $h-h_0$ and $\gamma-\gamma_0$ in the cited papers behave essentially the same as $\hatgamma_n - \bargamma_n$ in the current paper. 
\end{remark}

The previous discussion suggests that both $\whatG_n(\theta_0,\gamma_0)$ and $\whatG_{n,\gamma}^{\,\prime}(\theta_0,\gamma_0, \hatgamma_n-\bargamma_n)$ can be essentially viewed as U-statistics. Hence, although Assumption \ref{asmp:AN} is a high-level assumption, it is a direct result from the well-established theory on U-statistic (see, e.g., \citep{Hoeffding:1948}, \citep{Ustat:1994}, and \citep{Ustat:1996}) in most if not all cases. Therefore, we would expect it to be true under quite general conditions. In particular, it may also hold true for weakly dependent observations. Refer to \citep{Ustat:2006} and the references therein for more details.

\begin{remark} \label{rem:AN}
When $\hatgamma_n(\cdot) = \bbP_n \psi(\cdot) = \frac{1}{n} \sum_{j=1}^n \psi(\cdot, z_j) $, let $\psi_g(z_i,z_j) \coloneq g_\gamma^\prime (z_i, \theta_0, \gamma_0, \psi(z_i,z_j) )$ and $\phi_g(z_i, z_j) \coloneq \psi_g(z_i,z_j) - \bbE[\psi_g(z_i, z_j) | z_i]$.

According to the previous discussions, the term $\whatG_{n,\gamma}^{\,\prime}(\theta_0,\gamma_0,\hatgamma_n - \bargamma_n)$ is (approximately) a U-statistic:
\begin{align*}
	 & U_n = \frac{1}{n(n-1)} \sum_{\substack{i,j=1 \\ j\neq i}}^n g_\gamma^\prime (z_i, \theta_0, \gamma_0, \phi(z_i,z_j) ) = \frac{2}{n(n-1)} \sum_{i=1}^n \sum_{j>i} \frac{1}{2} [ \phi_g(z_i,z_j) + \phi_g(z_j,z_i) ].
\end{align*}
Its projection $\widehat{U}_n$ is given by
\begin{align*}
	\widehat{U}_n = \frac{1}{n} \sum_{i=1}^n \Big( \bbE[ \psi_g(z_j,z_i) | z_i ] - \bbE[ \psi_g(z _j, z_i) ] \Big), \text{ where } j \neq i.
\end{align*}
The U-statistic projection theory implies that $\sqrt{n}(U_n - \widehat{U}_n) \conP 0$. On the other hand, the statistic $\widehat{U}_n$ is a sum of \text{i.i.d.} random variables with zero mean. Hence, the asymptotic normality of $\whatG_{n,\gamma}^{\,\prime}(\theta_0,\gamma_0,\hatgamma_n - \bargamma_n)$ can be established. If we also know its correlation with $\whatG_n(\theta_0,\gamma_0)$, then Assumption \ref{asmp:AN} readily follows.

Consider the average density example, in which $g(z,\theta,\gamma)=\gamma(z) - \theta$. It can be shown that
\begin{align*}
	\sqrt{n} \, \whatG_n(\theta_0,\gamma_0) &= \frac{1}{\sqrt{n}} \sum_{i=1}^n [ \gamma_0(z_i) - \theta_0 ], \\
	\sqrt{n} \, \whatG_{n,\gamma}^{\,\prime}(\theta_0,\gamma_0,\hatgamma_n - \bargamma_n) &= \frac{1}{\sqrt{n}} \sum_{i=1}^n [ \gamma_0(z_i) - \theta_0 ] + \oP(1).
\end{align*}
Hence, Assumption \ref{asmp:AN} holds with $\Sigma_g = 4 \Var[\gamma_0(z)]$. As a comparison, if $\gamma_0$ were known, then we would be able to estimate $\theta_0$ by $\whatG_n(\theta_0,\gamma_0)$, the asymptotic variance of which is $\Var[\gamma_0(z)]$. This shows the efficiency loss due to not knowing $\gamma_0$.
\end{remark}

It is worth mentioning that the main advantage of this U-statistic perspective is that the asymptotic normality result with a root-$n$ rate can be established (provided that the U-statistic is not degenerate), regardless of the convergence rate of $\hatgamma_n - \bargamma_n$, which has no (asymptotic) biases by construction. Hence, if we can correct for those biases, then we can have asymptotic normality result for $\hattheta_n$ even in the case of having a low precision nonparametric ingredient.

\subsection{Possibly non-root-$n$-negligible biases} \label{subsec:biases}

Most previous asymptotic results for semiparametric two-step estimators, e.g., \citep{Andrews:1994}, \citep{Newey:1994}, \citep{Newey&McFadden:1994}, \citep{Chen:2007}, and \citep{Ichimura&Todd:2007}, impose certain conditions so that all the biases are root-$n$ negligible. Recent literature (recall the cited papers in the beginning of introduction) has started to relax such an assumption, so that some biases may have non-trivial impacts on the asymptotic distribution of $\hattheta_n$. 

Intuitively, one would expect the following two terms dominate the last three terms in the decomposition \eqref{eq:decG}:
\begin{align*}
    \Banb \coloneqq \whatG_{n,\gamma}^{\,\prime}(\theta_0,\gamma_0,\bargamma_n - \gamma_0) \quad \text{and} \quad  \Bnl \coloneqq \frac{1}{2} \whatG_{n,\gamma\gamma}^{\,\prime\prime}(\theta_0,\gamma_0,\hatgamma_n - \bargamma_n, \hatgamma_n - \bargamma_n).
\end{align*}
The term $\Banb$ represents the sample average of the nonparametric bias(es), while $\Bnl$ is a nonlinear bias.

\begin{remark}[Mean-square Continuity Condition] \label{rem:MSC}
Together with the stochastic equicontinuity condition (refer to Remark \ref{rem:SE} for the equivalent formulation in the current context), Assumption 5.3 in \citep{Newey:1994} and Condition (iii) of Theorem 8.1 in \citep{Newey&McFadden:1994} imply that there exists $\alpha(z)$ (or $\delta(z)$ in the latter paper) such that $\whatG_{n,\gamma}^{\,\prime}(\theta_0,\gamma_0, \hatgamma_n - \gamma_0) = \frac{1}{n} \sum_{i=1}^n \alpha(z_i) + \oP(n^{-1/2})$ (we modified the original expression to adapt to the current context) and $\bbE[\alpha(z)] = 0$. 

It is easy to see that $\alpha(z) \equiv g_{\gamma}^{\,\prime}(z,\theta_0,\gamma_0,\hatgamma_n - \bargamma_n)$ satisfies the second requirement (this can also be verified from a comparison of the asymptotic variances in the cited papers and in Assumption \ref{asmp:AN}). Then the first condition essentially requires $\Banb = \whatG_{n,\gamma}^{\,\prime}(\theta_0,\gamma_0, \bargamma_n - \gamma_0) = \oP(n^{-1/2})$. However, we are going to relax this restriction and allow $\Banb$, which may or may not include the \textquotedblleft singularity bias,\textquotedblright{} to be non-root-$n$-negligible. Following the discussion in Remark \ref{rem:SE}, even though it might be possible to reformulate the original stochastic equicontinuity condition in the two above-cited papers to make it hold true, the mean-square continuity condition will not hold in the current setting.
\end{remark}

\begin{remark}[Condition (2.8) in \citep{Andrews:1994}] \label{rem:SEAndrews}
A main result that \citep{Andrews:1994} intended to derive from the SE condition is (2.8) therein. Using the notation of the current paper, it can be written as:
\begin{align*}
	\whatG_{n}^{}(\theta_0,\hatgamma_n) -  \whatG_{n}^{}(\theta_0,\gamma_0) = \oP(n^{-1/2}).
\end{align*}
However, both $\Banb$ and $\Bnl$, two components of the left hand side difference, can be non-root-$n$-negligible, when the precision of $\hatgamma_n$ is low.
\end{remark}

Different from the previous discussion about asymptotic normality, the analysis of the above possibly non-root-$n$-negligible biases critically hinges on the order of $\hatgamma_n - \bargamma_n$ and/or $\bargamma_n - \gamma_0$. Therefore, given a suitably defined $\bargamma_n$, we introduce the following high-level assumption on the asymptotic behavior of the nonparametric estimator $\hatgamma_n$. 

\begin{assumption}[Bias Order] \label{asmp:BO}
Suppose that $\calB^{\texttt{NL}} = \bbE[\Bnl]=O(n^{-2r})$ and $\calB^{\texttt{ANB}} = \bbE[\Banb] =O(n^{-s})$, where $r,s>0$ such that
\begin{align*}
	\| \Banb - \calB^{\texttt{ANB}} \| = \oP(n^{-1/2}) \quad and \quad  \| \Bnl - \calB^{\texttt{NL}} \| = \oP(n^{-1/2}).
\end{align*}
Here, we allow $2r$ and/or $s$ to be smaller than or equal to $1/2$. 
\end{assumption}


Typically, the above rates should depend on the tuning parameter of the nonparametric estimator $\hatgamma_n$. Since it is a common practice to set the tuning parameter as a function of $n$ eventually, we express all the rates in the above assumption in terms of a power of $n$, for convenience. 

Compared with the previous requirement that both $\Bnl$ and $\Banb$ are $\oP(n^{-1/2})$, Assumption \ref{asmp:BO} is much weaker. It requires no more than splitting each (asymptotically negligible) bias into two components: one is $\oP(n^{-1/2})$, while the other is not. In this sense, it should be satisfied under very general conditions. 

For example, when $\hatgamma_n(z_i) - \bargamma_n(z_i) = \frac{1}{n-1} \sum_{j\neq i} \phi(z_i,z_j) $ as above, we can obtain
\begin{align*}
	\whatG_{n,\gamma\gamma}^{\,\prime\prime} \big(\theta_0,\gamma_0, \hatgamma_n - \bargamma_n, \hatgamma_n - \bargamma_n \big) = \frac{1}{n-1} U_{n,1} + \frac{n-2}{n-1} U_{n,2},
\end{align*}
where $U_{n,1}$ and $U_{n,2}$ are two U-statistics:
\begin{align*}
	U_{n,1} &= \frac{1}{n(n-1)} \sum_{i\neq j} g_{\gamma\gamma}^{\prime\prime}(z_i, \theta_0, \gamma_0, \phi(z_i,z_j), \phi(z_i,z_j)), \\
	U_{n,2} &= \frac{1}{n(n-1)(n-2)} \sum_{i\neq j\neq l} g_{\gamma\gamma}^{\prime\prime}(z_i, \theta_0, \gamma_0, \phi(z_i,z_j), \phi(z_i,z_l)).
\end{align*}
Then, if one can choose the tuning parameter in such a way that $\| \bbE[U_{n,1}] \|$ is of order $o(n)$, then we can find $r>0$ so that $\calB^{\texttt{NL}} = \frac{1}{n-1} \bbE[ U_{n,1} ] = O(n^{-2r})$. Furthermore, if $\Var( U_{n,1} )$ is of order $o(n)$ with appropriately chosen tuning parameter, we have $ \| \frac{1}{n-1} U_{n,1} - \calB^{\texttt{NL}} \| = \oP(n^{-1/2}) $. On the other hand, the proof of Lemma \ref{lem:BO} in the appendix shows that $U_{n,2}$ has zero mean and a degenerate (U-statistic) kernel. Consequently, as long as the variance of the sum of the nondegenerate projections is of order $o(n)$, one can show that $\| U_{n,2} \| = \oP(n^{-1/2})$. 

As for $\Banb$, note that it is the average of a sequence of \text{i.i.d.} random variables:
\begin{align*}
	\Banb \coloneqq \whatG_{n,\gamma}^{\,\prime}(\theta_0,\gamma_0,\bargamma_n - \gamma_0) = \frac{1}{n} \sum_{i=1}^n g_{\gamma}^{\prime}(z_i, \theta_0, \gamma_0, \bargamma_n(z_i) - \gamma_0(z_i) ).
\end{align*}
Let $\calB^{\texttt{ANB}} = \bbE[ g_{\gamma}^{\prime}(z_i, \theta_0, \gamma_0, \bargamma_n(z_i) - \gamma_0(z_i) ) ]$. Then as long as $g_{\gamma}^{\prime}(z_i, \theta_0, \gamma_0, \bargamma_n(z_i) - \gamma_0(z_i) )$ has degenerate variance, then we readily get $\| \Banb - \calB^{\texttt{ANB}} \| = \oP(n^{-1/2})$. For more details, refer to the appendix.

\begin{named-exmp}[Kernel Density Estimator Continued]
%
%

For the (leave-one-out) average density estimator $\hattheta_n^{\texttt{AD}} = \frac{1}{n} \sum_{i=1}^n \hatgamma_n(z_i) $. We have $\Bnl = 0$ and 
\begin{align*}
    \Banb &= \frac{1}{n} \sum_{i=1}^n [\bargamma_n(z_i) - \gamma_0(z_i)] = \frac{1}{n} \sum_{i=1}^n \int_{\bbR} K(u) [\gamma_0(z_i - hu) - \gamma_0(z_i)] du = \OP( h^m ), \\
    \calB^{\texttt{ANB}} &= \int_{\bbR} \int_{\bbR} K(u) [\gamma_0(x - hu) - \gamma_0(x)] \gamma_0(x) du dx = O(h^m).
\end{align*}

For the integrated squared density estimator $\hattheta_n^{\texttt{ISD}} = \int \hatgamma_n^2(x) dx$, we have
\begin{align*}
	\Bnl &= \int_{\bbR} [ \hatgamma_n(x) - \bargamma_n(x) ]^2 dx = \OP\big( \frac{1}{nh^{d_z}} \big), \\
	\Banb &= 2 \int_{\bbR} \gamma_0(x) [\bargamma_n(x) - \gamma_0(x)] dx = O(h^{m}).
\end{align*}
and
\begin{align*}
	\calB^{\texttt{NL}} &= \frac{1}{nh^{d_z}} \int_{\bbR} \Big(\gamma_0(x) \int_{\bbR} K(u)^2 du \Big) dx = O\Big( \frac{1}{nh^{d_z}} \Big).
\end{align*}
Since $\Banb$ is deterministic, we can set $\calB^{\texttt{ANB}} = \Banb$. Refer to the appendix for an example with the Nadaraya–Watson estimator.
\end{named-exmp}



As mentioned in the previous subsection, there is no \textquotedblleft singularity bias\textquotedblright{} (even with the kernel-based method) when we use the smooth measure $\bbP_n^{\texttt{AC}}$ (recall Remark \ref{rem:Pac}) in the construction of $\whatG_n$ (this gives the integrated square density estimator in the above example). However, it may bring an additional nonlinear bias, when the alternative estimator is linear in $\hatgamma_n$. Besides, we note that the nonlinear bias (when it exists) and the \textquotedblleft singularity bias\textquotedblright{} are of the same order. Hence, they can be corrected simultaneously by using the multi-scale jackknife method (see Section \ref{subsec:jackknife}). 



To make both biases shrink faster than the root-$n$ rate, we need both $r>1/4$ and $s >1/2$, which are consistent with the prevalent requirement of a faster-than-$n^{1/4}$ convergence rate for the nonparametric estimator. Some complications may arise if we have more than one source of bias in $\bargamma_n - \gamma_0$, like in the average density example. Once these conditions are satisfied, one can use some well-established empirical process results, such as the stochastic equicontinuity condition \citep{Andrews:1994, Newey:1994}. However, if $r \leq 1/4$ or $s \leq 1/2$, then either $\Bnl$ or $\Banb$ will not be $\oP(n^{-1/2})$. In such cases, such bias(es) will have some non-trivial impact(s) on the asymptotic behavior of $\hattheta_n$.

\begin{named-exmp}[Kernel Density Estimator Continued] 
In view of the above discussion, no matter we use the original kernel density estimator or its \textquotedblleft leave-one-out\textquotedblright{} version, the necessary and sufficient condition for both $\Bnl$ and $\Banb$ to be root-$n$ negligible is $1/(2m) < \kappa < 1/(2d_z)$, which requires $d_z < m$, i.e., the dimension of the random vector should be smaller than the order of the kernel. If this condition fails, then at least one of the two biases will not be asymptotically negligible at the root-$n$ rate. To some extent, this observation also reflects the curse of dimensionality: if $d_z \geq m$, then there is no way to make both biases root-$n$ negligible. In fact, when $d_z > m$, if the bandwidth satisfies $1/(2d_z) < \kappa < 1/(2m)$, then neither $\Bnl$ nor $\Banb$ is root-$n$ negligible. Motivated by this possibility, we are going to keep both biases in our analysis. This observation also indicates that our bias correction methods may help ameliorate the curse of dimensionality. 
\end{named-exmp}


The following lemma gives the sufficient conditions for the remaining terms in \eqref{eq:decG}, as well as the impact of $\calJ_n-\calJ_0$ on $\hattheta_n$, to be root-$n$ negligible, 

\begin{lemma} \label{lem:op1}
Suppose that Assumptions \ref{asmp:quadraticity} (about $g$) and \ref{asmp:BO} both hold true. Additionally, assume that $\calJ_n - \calJ_0 = \OP\big( \whatG_n(\theta_0, \hatgamma_n) \big)$. 

We have the following conclusions: (i) if $s+2r>1/2$ and $r>1/8$, then $(\calJ_n - \calJ_0) \Bnl = \oP(n^{-1/2})$; (ii) if $s+2r>1/2$ and $s>1/4$, then $(\calJ_n - \calJ_0) \Banb = \oP(n^{-1/2})$; (iii) if $\frac{1}{n} \sum_{i=1}^n \bbE[ \| \hatgamma_n(z_i) - \gamma_0(z_i) \|^3 ] \leq C n^{-3(r\wedge s)}$ for some finite number $C$, $s>1/4$ and $r>1/6$, then 
\begin{gather*}
    \whatG_n(\theta_0,\hatgamma_n) - \whatG_n(\theta_0,\gamma_0) - \whatG_{n,\gamma}^{\,\prime}(\theta_0,\gamma_0,\hatgamma_n - \bargamma_n) - \calB^{\texttt{NL}} - \calB^{\texttt{ANB}} = \oP(n^{-1/2}).
\end{gather*}
\end{lemma}

The assumption $\calJ_n - \calJ_0 = \OP\big( \whatG_n(\theta_0, \hatgamma_n) \big)$ is to accommodate the possibility that $\calJ_n - \calJ_0$ may depend on or be related to $\whatG_n(\theta_0, \hatgamma_n)$, which complicates the proof a bit. In general, the above lemma will also hold if one assumes $\calJ_n - \calJ_0 = \OP(n^{-\iota} )$, and then let $\iota + 2r > 1/2$ in part (i), and $\iota + s > 1/2$ in part (ii). The same conclusions can be verified rather straightforwardly. In such case, the parameter $\iota$ is essentially equivalent to $1/\rho$ in Lemma 1 of \citep{Cattaneo&Jansson:2018}. 



As discussed above, most previous papers on semiparametric estimators require both $\Bnl$ and $\Banb$ to be root-$n$ negligible. Although recent works relax this requirement, they often require one of $\Bnl$ and $\Banb$ is root-$n$ negligible. For instance, Theorem 2 of \citep{Cattaneo&Jansson:2018} effectively require the bias $\Banb$ to be root-$n$ negligible (small bandwidth asymptotics), while \citep{CEINR:2018} implicitly assume the nonlinear bias $\Bnl$ is root-$n$ negligible (large bandwidth asymptotics). 

However, it is often not easy to check whether such restrictions hold or not in practice. Moreover, recall the previous example of the kernel density estimator. It is possible that both biases are non-root-$n$-negligible. In view of these results, we keep both $\Bnl$ and $\Banb$ in our analysis. In a different setup with the non-stationary underlying process and in-fill asymptotics, \citep{Yang:2020Semi} adopts a similar approach. The following theorem gives the first main result of this paper. 

\begin{theorem}[Asymptotic Normality for $\hattheta_n$] \label{thm:infeasibleCLT} 
Suppose that Assumptions \ref{asmp:AL} to \ref{asmp:AN} hold true. Assume that $\calJ_n - \calJ_0 = \OP\big( \whatG_n(\theta_0, \hatgamma_n) \big)$. If $s>1/4$ and $r>1/6$, then we have
\begin{align*}
    \sqrt{n} \big( \hattheta_n - \theta_0 - \calJ_n \calB^{\texttt{NL}} - \calJ_n\calB^{\texttt{ANB}} \big) \conL \calN\big(0, \Sigma_\theta \big),
\end{align*}
where $\Sigma_\theta = \calJ_0 \, \Sigma_g \, \calJ_0^\intercal$ with $\Sigma_g$ given in  Assumption \ref{asmp:AN}.
\end{theorem}

The conditions $s>1/4$ and $r>1/6$ only require a faster-than-$n^{1/6}$ convergence rate for the nonparametric estimator $\hatgamma_n$, consistent with the conclusion of \citep{Cattaneo&Jansson:2018} in the kernel-based case. This is a weaker condition than the typical requirement of a faster-than-$n^{1/4}$ convergence rate (see those cited papers at the beginning of this subsection). 

Besides, we also note that the above central limit theorem (CLT) is infeasible, for that the two biases are evaluated at $(\theta_0, \gamma_0)$, both of which are unknown. In the next section, we are going to discuss how to correct for these biases and conduct robust inference.

\begin{remark}
It might happen that the bias $\Banb$ is identically zero. For example, in the continuous-time setting (with in-fill asymptotics), \citep{Yang:2020Semi} has shown that, when estimating integrated volatility functionals, the counterpart of $\Banb$, which is the first-order effect of the nonparametric bias, is canceled by the discretization error. In the cited paper, what left is the counterpart of the following second-order effect of the nonparametric bias:
\begin{align*}
    \frac{1}{2} \whatG_{n,\gamma\gamma}^{\,\prime\prime}(\theta_0,\gamma_0,\bargamma_n - \gamma_0, \bargamma_n - \gamma_0) = \sum_{l=1}^L \OP(n^{-2s_l}).
\end{align*}
In such case, then one can replace the first-order effect by the above second-order one and replace $s$ by $2s$ in Lemma \ref{lem:op1} and Theorem \ref{thm:infeasibleCLT}.
\end{remark}

\section{Bias-Robust Inference} \label{sec:inference}

\citep{Cattaneo&Jansson:2018} propose a bootstrap-based inference procedure that is robust to the nonlinear bias. We believe that if the bootstrap version of all the above assumptions hold, then the corresponding inference should also be robust to the average nonparametric bias. Since it has been proposed in the literature, we will not discuss it here.

In this section, we are going to discuss two alternative methods to conduct inference that is robust to the possibly non-root-$n$-negligible bias(es). At the end of this section, we will also discuss an extension of our framework to the case where $\hattheta_n$ is constructed as the sample average of some discontinuous functionals of $\hatgamma_n$.

For simplicity, we illustrate the ideas using kernel-based estimators. The linear sieve case would be characterized in a similar manner. Yet, the nonlinear sieve case may require extra non-trivial efforts.

\subsection{Multi-scale jackknife} \label{subsec:jackknife}

The original jackknife estimator, first introduced by \citep{Quenouille:1949AMS}, is essentially a linear combination of estimators computed from samples with different sizes, for that the biases in many estimators depend on the sample size. While in the current context, the biases depend on the tuning parameter. Thus, it is natural to utilize the tuning parameter in the role of the sample size (see, e.g., \citep{Schucany&Sommers:1977}, \citep{Bierens:1987kernel}, and \citep{PSS:1989} among others). However, there is only one bias in these papers. In the context of in-fill asymptotics, \citep{Li&Liu&Xiu:2019} has developed a multi-scale jackknife (MSJ) estimator to correct for various biases for integrated volatility functionals. 

In this subsection, we are going to show that MSJ can remove various biases in the current context, provided that we have some knowledge about the structure of the nonparametric estimator, i.e., knowing how the rates in Assumption \ref{asmp:BO} depend on the tuning parameter.


In the kernel-based case, the semiparametric estimator $\hattheta_n$ depends on the bandwidth $h$. Let $Q$ be a finite positive integer. Then consider a sequence of estimators $\{\hattheta_n(h_q)\}_{q=1}^Q$ and a sequence of real numbers $\{w_q\}_{q=1}^Q$. For example, define the following three-scale jackknife (3SJ) estimator:
\begin{align*}
    \hattheta_n^w = \sum_{q=1}^3 w_q \hattheta_n(h_q),
\end{align*}
where 
\begin{align} \label{eq:kernelwq}
    \sum_{q=1}^3 w_q = 1, \quad \sum_{q=1}^3 w_q h_q^m = o(n^{-1/2}), \quad \sum_{q=1}^3 \frac{w_q}{n h_q^{d_z}} = o(n^{-1/2}).
\end{align}
In practice, for example, we can choose $h_q = \eta_q h$, where $\{\eta_q\}_{q=1}^Q$ is a sequence of positive numbers. In the above three-scale case, the weights $\{ w_q\}_{q=1}^3$ are solved as
\begin{align*}
    \left( \begin{matrix}
    w_1 \\
    w_2 \\
    w_3
\end{matrix} \right) = \left( \begin{matrix}
    1 & 1 & 1 \\
    \eta_1^m & \eta_2^m & \eta_3^m \\
    \eta_1^{-d_z} & \eta_2^{-d_z} & \eta_3^{-d_z} 
\end{matrix} \right)^{-1} \left( \begin{matrix}
    1 \\
    0 \\
    0
\end{matrix} \right).
\end{align*}
Moreover, one can choose a larger $Q$ to remove/reduce more biases. For instance, in the kernel case, the smoothing bias may also have components that are $\OP(h^{m+1})$, $\OP(h^{m+2})$, or of even higher orders (for symmetric kernels, the odd-order terms will be zero).

We consider the general case where we have the smoothing bias $\BanbOne$, the \textquotedblleft singularity bias\textquotedblright{} $\BanbTwo$ and the nonlinear bias $\Bnl$. The reason is that the \textquotedblleft singularity bias\textquotedblright{} may be unavoidable when estimating the asymptotic variance using the bootstrap method. Recall that $\BanbTwo$ and $\Bnl$ are of the same order when both exist. The key is to show that, under condition \eqref{eq:kernelwq}, the following three terms
\begin{gather*}
    \widetilde{\calB}_{n,1}^{\texttt{ANB}} = \sum_{q=1}^Q w_q \, \BanbOne(h_q), \quad  
    \widetilde{\calB}_{n,2}^{\texttt{ANB}} = \sum_{q=1}^Q w_q \, \BanbTwo(h_q), \quad 
    \widetilde{\calB}_{n}^{\texttt{NL}} = \sum_{q=1}^Q w_q \, \Bnl(h_q).
\end{gather*}
are all root-$n$ negligible. Then the following CLT readily follows.

\begin{theorem}[Multi-scale jackknife] \label{thm:jackknifeCLT}
Suppose that all assumptions of Theorem \ref{thm:infeasibleCLT} hold true and that $\hatgamma_n(h_q)$ is a kernel-based nonparametric estimator depending on the bandwidth $h_q$, where $q=1,\cdots,Q$ for some finite $Q$. In addition, assume $h_q\rightarrow0$, $n^2h_q^{3d_z}\rightarrow\infty$, $nh_q^{4m} \rightarrow 0$, and that the general version of condition \eqref{eq:kernelwq} is satisfied. Then we have
\begin{align*}
    \sqrt{n} \big( \hattheta_n^w - \theta_0 \big) \conL \calN(0, \Sigma_\theta^w).
\end{align*}
The asymptotic variance is given by $\Sigma_\theta^w \coloneqq \calJ_0 \, \Sigma_g^w \calJ_0^\intercal$ and $\Sigma_g^w$ is the asymptotic variance of the following (exact or approximate) U-statistic
\begin{align*}
    \whatG(\theta_0, \gamma_0) + \whatG_{n,\gamma}^{\,\prime}\big(\theta_0,\gamma_0, \hatgamma_n^w - \bargamma_n^w \big),
\end{align*}
where $\hatgamma_n^w = \sum_{q=1}^Q w_q \hatgamma_n(h_q)$ and $\bargamma_n^w = \sum_{q=1}^Q w_q \bargamma(h_q)$.

Suppose that the following column vector
\begin{align*}
    \sqrt{n} \Big( \whatG(\theta_0, \gamma_0) + \whatG_n\big(\theta_0, \gamma_0, \hatgamma_n(h_q) - \bargamma_n(h_q) \big) \Big)_{q=1,\cdots,Q}^\intercal 
\end{align*}
converges in distribution to $\calN(0, \Sigma_g^Q)$, then we have $\Sigma_\theta^w = \calJ_0 \, w \, \Sigma_g^Q w^\intercal \calJ_0^\intercal$.
\end{theorem}

For illustration purpose, consider the case where $h_q \propto n^{-\kappa}$ for all $q=1,\cdots,Q$. Then we have $r=(1-\kappa d_z)/2$, $s_1 = \kappa m$ and $s_2=2r$ (if there is \textquotedblleft singularity bias\textquotedblright{}) for the kernel-based estimators. The requirements $r>1/6$ and $s>1/4$ in Theorem \ref{thm:infeasibleCLT} are equivalent to $n^2h_q^{3d_z}\rightarrow\infty$ and $nh_q^{4m} \rightarrow 0$ (the conditions in the above theorem). To put it differently, we need $\kappa \in ( 1/(4m), 2/(3d_z) )$. This set is non-empty if and only if $3d_z < 8 m$, which is  weaker than $d_z < m$ (recall the previous discussion on the curse of dimensionality). As a comparison, we note that $r>1/4 \Leftrightarrow \kappa < 1/(2d_z) \Leftrightarrow nh_q^{2d_z} \rightarrow \infty$ and $s_1 > 1/2 \Leftrightarrow \kappa > 1/(2m) \Leftrightarrow n h_q^{2m}$. 

Intuitively, the statistics $\{\whatG(\theta_0, \gamma_0) + \whatG_n\big(\theta_0, \gamma_0, \hatgamma_n(h_q) - \bargamma_n(h_q) \big)\}_{q=1}^Q$ are constructed from the same sample, hence are \textquotedblleft highly\textquotedblright{} correlated. It would be reasonable to expect that, in some cases, their correlations are approximately one. If so, then the matrix $\Sigma_g^Q$ becomes $\Sigma_g \bm{1}_Q$ (assuming $\Sigma_g$ is a scalar for illustration purpose), where $\bm{1}_Q$ is a $Q$-by-$Q$ matrix with all the elements being one. Then the asymptotic variance $\Sigma_\theta^w = \calJ_0 \Sigma_g w \bm{1}_Q w^\intercal \calJ_0^\intercal= \Sigma_\theta$ (note that $w \bm{1}_Q w^\intercal = ( \sum_{q=1}^Q w_q )^2 = 1$). That is to say, when these estimators are approximately perfectly correlated, there is no efficiency loss by using the MSJ estimator. 

In some cases, it may not be very easy to find the analytical form of the functional $g_{\gamma}^{\,\prime}(\theta_0,\gamma_0,\cdot)$ or its variance. Hence, it may not always be possible to estimate $\Sigma_g^w$ directly. In such cases, one can use the following algorithm to estimate the asymptotic variance $\Sigma_\theta^w$.

\begin{algorithm}[Bootstrap variance estimator] The procedure consists of the following steps: (1) Draw a bootstrap sample $\{z_i^\ast\}_{i=1}^n$ and calculate $\hattheta_n^{w \ast}$. (2) Repeat Step (1) a large number of times, say $P$, and get $\{\hattheta_{n,p}^{w \ast} \}_{p=1}^P$. (3) Compute $\Sigma_\theta^{w \ast}$ as the sample variance-covariance of $\{ \hattheta_{n,p}^{w \ast} \}_{p=1}^P$.
\end{algorithm}

\begin{theorem}[Bootstrap variance] \label{thm:jackknifeVar}
Suppose that the assumptions of Theorem \ref{thm:jackknifeCLT} hold true. In addition, assume that $g^\ast \equiv g$, $g_{\gamma}^{\,\ast\prime} \equiv g_{\gamma}^{\,\prime}$, and both $g(\theta,\gamma)$ and $g_{\gamma}^{\,\prime}(\theta,\gamma,\cdot)$ are Lipschitz continuous with respect to $\theta$ and $\gamma$ in a neighborhood of $(\theta_0,\gamma_0)$. Then $\Sigma_\theta^{w\ast} \conP \Sigma_\theta^w$.
\end{theorem}

Since the \textquotedblleft singularity bias\textquotedblright{} can always be removed together with the nonlinear bias, the bootstrap estimator $\hattheta_n^{w \ast}$ will have no such bias, even if the re-sampled data may include several replicates of the same observation.

If certain bias(es) is/are root-$n$ negligible, then some of the requirements in Condition \eqref{eq:kernelwq} will not be binding, which can then be simplified. For instance, if the smoothing bias is root-$n$ negligible, i.e., $h_q^m = o(n^{-1/2})$ for $q=1,2$, then we only need
\begin{align*}
    \sum_{q=1}^2 w_q = 1 \quad \text{and} \quad \sum_{q=1}^2 \frac{w_q}{n h_q^{d_z}} = o(n^{-1/2}).
\end{align*}
On the other hand, if the nonlinear bias and the \textquotedblleft singularity bias\textquotedblright{} are root-$n$ negligible, i.e., $h_q^{-d_z} = o(n^{1/2})$ for $q=1,2$, then we only need
\begin{align*}
    \sum_{q=1}^2 w_q = 1 \quad \text{and} \quad \sum_{q=1}^2 w_q h_q^m = o(n^{-1/2}).
\end{align*}
In these two cases, the two-scale jackknife (2SJ) estimators are asymptotically normal with a root-$n$ rate.

\subsection{Analytical bias correction} \label{subsec:analytical}


The analytical bias correction method requires more assumptions on the semiparametric model. The idea is to introduce some sufficient conditions so that we can construct consistent estimators of the average nonparametric bias $\calB^{\texttt{ANB}}$ and the nonlinear bias $\calB^{\texttt{NL}}$. 

Suppose that the functional $g$ is twice Fr\'{e}chet differentiable with respect to $\gamma$ around $\gamma_0$. Consider the general case where $\gamma$ is a matrix-valued function, with the row and column numbers being $r_\gamma$ and $c_\gamma$, respectively. Define the following matrix representation of the partial derivative \citep{Kollo&vonRosen:2006}:
\begin{align*}
    \Big( \frac{\partial}{\partial \vect(\gamma)} \Big)^\intercal = \frac{\partial}{\partial [\vect(\gamma)]^\intercal} = \Big( \frac{\partial}{\partial \gamma_{11}}, \cdots, \frac{\partial}{\partial \gamma_{r_\gamma 1}}, \cdots, \frac{\partial}{\partial \gamma_{1 c_\gamma}}, \cdots, \frac{\partial}{\partial \gamma_{r_\gamma c_\gamma}} \Big).
\end{align*}
Let $\bbD_\gamma g= \frac{\partial g}{\partial [\vect(\gamma)]^\intercal}$ and $\bbD_{\gamma\gamma}^2 g =  \frac{\partial}{\partial \vect(\gamma)} \otimes \frac{\partial g}{\partial [\vect(\gamma)]^\intercal}$. Assume that 
\begin{gather*}
    g_{\gamma}^{\,\prime} (z,\theta_0,\gamma_0, \gamma - \gamma_0) = \bbD_{\gamma} g(z,\theta_0,\gamma_0) \, \vect\big( \gamma(z) - \gamma_0(z) \big), \\
    g_{\gamma\gamma}^{\,\prime\prime} (z,\theta_0,\gamma_0, \gamma - \gamma_0) = \big[ \vect\big( \gamma(z) - \gamma_0(z)  \big)^{\otimes 2} \otimes I_{d_g} \big]^\intercal \, \vect\big( \bbD_{\gamma\gamma}^2 g(z,\theta_0,\gamma_0) \big). 
\end{gather*}
Under these assumptions, the two biases can be written as 
\begin{gather*}
    \Banb = \frac{1}{n} \sum_{i=1}^n \bbD_{\gamma} g(z_i,\theta_0,\gamma_0) \, \vect\big( \bargamma_n(z_i) - \gamma_0(z_i) \big), \\
    \Bnl = \frac{1}{n} \sum_{i=1}^n \big[ \vect\big( \hatgamma_n(z_i) - \bargamma_n (z_i)  \big)^{\otimes 2} \otimes I_{d_g} \big]^\intercal \, \vect\big( \bbD_{\gamma\gamma}^2 g(z_i,\theta_0,\gamma_0) \big).
\end{gather*}

Suppose that $n^r \vect\big( \hatgamma_n(x) - \bargamma_n(x) \big) \conL \calN(0, V(x))$ for any $x\in\bbR^{d_z}$. Then, when Assumption \ref{asmp:BO} holds true, we would expect that $\Bnl - \calB^{\texttt{NL}} = \oP(n^{-1/2})$ with the following $\calB^{\texttt{NL}}$:
\begin{align*}
    \calB^{\texttt{NL}} \coloneqq \bbE\Big( \big[ \vect\big( V(z) \big) \otimes I_{d_g} \big]^\intercal \,  \vect\big( \bbD_{\gamma\gamma}^2 g(z,\theta_0,\gamma_0) \big) \Big).
\end{align*}
Suppose that we have a consistent estimator $\hatV_n(\cdot)$ of the asymptotic variance $V(\cdot)$. It then follows that we can estimate $\calB^{\texttt{NL}}$ by 
\begin{align} \label{eq:hatBnl}
    \whatBnl = \frac{1}{n^{1+2r}} \sum_{i=1}^n \big[ \vect\big( \hatV_n(z_i) \big) \otimes I_{d_g} \big]^\intercal \,  \vect\big( \bbD_{\gamma\gamma}^2 g(z_i,\hattheta_n,\hatgamma_n) \big).
\end{align}

On the other hand, suppose that there exists a (point-wise) consistent estimator $\hat{\bargamma}_n$ of $\bargamma_n$. Then we can estimate $\calB^{\texttt{ANB}}$ by
\begin{align} \label{eq:hatBanb}
    \whatBanb = \whatG_{n,\gamma}^{\,\prime}(\hattheta_n, \hatgamma_n, \hat{\bargamma}_n - \hatgamma_n) =  \frac{1}{n} \sum_{i=1}^n \bbD_{\gamma} g(z_i,\hattheta_n,\hatgamma_n) \, \vect\big( \hat{\bargamma}_n(z_i) - \hatgamma_n(z_i) \big).
\end{align}
For simplicity, we assume that there is no \textquotedblleft singularity bias\textquotedblright{} in $\Banb$, since it can be easily removed using the methods discussed in Section \ref{subsec:VU}.


\begin{assumption} \label{asmp:analytical}
Suppose that Assumption \ref{asmp:BO} holds with real numbers $r$ and $s$. Assume that the functional $g$ is twice Fr\'{e}chet differentiable with respect to $\gamma$ around $\gamma_0$, with $\bbE\big( \big\| \bbD_{\gamma\gamma}^2 g(z,\theta_0,\gamma_0) \big\|^2 \big) < \infty$ and
\begin{align*}
    \bbE\big( \| \bbD_{\gamma} g(z,\theta_0,\gamma_0) - \bbD_{\gamma} g(z,\hattheta_n,\hatgamma_n) \|^2 \big) = O( n^{-2(r\wedge s)} ),
\end{align*}
for sufficiently large $n$. 

Moreover, there exist $\whatV_n$ and $\hat{\bargamma}_n$ such that $\hat{\bargamma}_n - \bargamma_n \conP 0$, $\bbE\big( \| \hat{\bargamma}_n(z) - \hatgamma_n(z) \|^2 \big) = o(n^{-2t}),$ and
\begin{align*}
    \bbE\big( \big\| n^{2r} \vect\big( \hatgamma_n(z) - \bargamma_n (z)  \big)^{\otimes 2} - \vect\big( \hatV_n(z) \big) \big\|^2 \big) = o(n^{-2v}), 
\end{align*}
where $t$ and $v$ are some positive real numbers.
\end{assumption}

Assumption \ref{asmp:analytical} is a strengthened version of the combination of Assumptions \ref{asmp:quadraticity} and \ref{asmp:BO}. The twice Fr\'{e}chet differentiable condition implies the quadratic approximation in Assumption \ref{asmp:quadraticity}, with a more detailed structure on the first- and second-order derivatives. In addition, Assumption \ref{asmp:analytical} also imposes certain conditions on the estimators of $V$ and $\bargamma_n$ in Assumption \ref{asmp:BO}.

\begin{theorem}[Analytical bias correction] \label{thm:analytical}
Suppose that Assumptions \ref{asmp:AL} and \ref{asmp:analytical} hold true. Define $\bar{\bargamma}_n(z_i) \coloneqq \bbE[ \hat{\bargamma}_n(z_i) | z_i ]$. Assume that $s>1/4$, $r>1/6$, $t+r\wedge s > 1/2$, $v+2r>1/2$, and
\begin{gather}
    \sqrt{n} \Big( \whatG_n(\theta_0,\gamma_0) + \whatG_{n,\gamma}^{\,\prime}(\theta_0,\gamma_0, 2\hatgamma_n - \hat{\bargamma}_n - 2\bargamma_n + \bar{\bargamma}_n ) \Big) \conL \calN(0, \widetilde{\Sigma}_g), \label{eq:analytical-AN} \\
    \whatG_{n,\gamma}^{\,\prime}(\theta_0,\gamma_0, 2\bargamma_n - \bar{\bargamma}_n - \gamma_0) = \oP(n^{-1/2}). \label{eq:analytical-op1}
\end{gather}
Then we have
\begin{align*}
    \sqrt{n} \big( \hattheta_n - \theta_0 - \calJ_n \whatBnl - \calJ_n\whatBanb \big) \conL \calN\big(0, \calJ_0 \, \widetilde{\Sigma}_g \, \calJ_0^\intercal \big).
\end{align*}
where $\whatBnl$ and $\whatBanb$ are given by \label{eq:hatBnl} and \eqref{eq:hatBanb}, respectively.
\end{theorem}

A possible choice for $\hat{\bargamma}_n$ is $\hatgamma_n$, which then yields $\bar{\bargamma}_n \equiv \bargamma_n$. In this case, condition \ref{eq:analytical-AN} reduces to Assumption \ref{asmp:AN}. Condition \ref{eq:analytical-op1} is then equivalent to $\Banb = \oP(n^{-1/2})$. That is to say, when we couldn't estimate $\Banb$, we can obtain an analytical-based inference only if $\Banb$ is root-$n$ negligible.

In some cases, it is possible to have an estimator $\hat{\bargamma}_n$ different from $\hatgamma_n$. Then Condition \ref{eq:analytical-op1} requires that this estimator can reduce the average nonparametric bias to the extent that the remaining bias becomes root-$n$ negligible. 
Conditions \ref{eq:analytical-op1} and \ref{eq:analytical-AN} together imply that
\begin{align*}
    \whatG_n(\theta_0,\gamma_0) + \whatG_{n,\gamma}^{\,\prime}(\theta_0,\gamma_0, 2\hatgamma_n - \hat{\bargamma}_n - \gamma_0 ) \conL \calN(0, \widetilde{\Sigma}_g).
\end{align*}
That is, the asymptotic variance is determined by the updated estimator $2\hatgamma_n - \hat{\bargamma}_n$. We expect that, in most cases, the left hand side can be written as a U-statistic. Then the above asymptotic normality result shall be satisfied under very general conditions.

\begin{named-exmp}[Kernel density estimator continued]
Let $\hatgamma_n$ be the \textquotedblleft leave-one-out\textquotedblright{} kernel density estimator. In this case, $V(x) = \gamma_0(x) \int K^2(u) du$, which can be easily estimated. Recall that $\bargamma_n(\cdot) = \int K(u)  \gamma_0(\cdot - hu) du$. It then follows that
\begin{align*}
    \hat{\bargamma}_n(\cdot) = \int K(u) \hatgamma_n(\cdot - hu) du, \,\, \bar{\bargamma}_n(\cdot) = \int\int K(u) K(v) \gamma_0(\cdot - hu - hv) du dv.
\end{align*}
The updated estimator becomes
\begin{align*}
    & 2\hatgamma_n(z_i) - \hat{\bargamma}_n(z_i) = \frac{1}{n-1} \sum_{j\neq i} \Big( 2 K_h( z_i - z_j ) - \int K_h(z_i - x) K_h(x - z_j) dx \Big) \\
        =\,& \frac{1}{n-1} \sum_{j\neq i} \Big( 2 K_h( z_i - z_j ) - \int K_h(z_i - z_j - y) K_h(y) dy \Big) = \frac{1}{n-1} \sum_{j\neq i} \tilK_h(z_i - z_j),
\end{align*}
where $\tilK_h(u) = \frac{1}{h^{d_z}} \tilK(u/h)$ and $\tilK(u) = 2 K(u) - \int K(u-v) K(v) dv$ is the twicing kernel studied by \citep{Stuetzle&Mittal:1979} and \citep{NHR:2004}.

According to \citep{NHR:2004}, the twicing kernel enjoys a small bias property, which makes Condition \eqref{eq:analytical-op1} less stringent than requiring that $\Banb$ is root-$n$ negligible. For instance, if $\gamma_0$ is at least $2m$ times differentiable and the order of $K$ is $m$, then $\whatG_{n,\gamma}^{\,\prime}(\theta_0,\gamma_0, 2\bargamma_n - \bar{\bargamma}_n - \gamma_0) = \OP(h^{2m}) = \OP(n^{-2\kappa m})$. Hence, Condition \eqref{eq:analytical-op1} only requires $\kappa > 1/(4m)$ (cf. $\kappa > 1/(2m)$ for $\Banb$ to be root-$n$ negligible). If Condition (2.4) in \citep{NHR:2004} is satisfied with some function $\nu$, then the requirement that $\gamma_0$ is at least $2m$ times differentiable can be replaced by both $\nu$ and $\gamma_0$ are at least $m$ times differentiable.
\end{named-exmp}


The limitation of the analytical bias correction method is that it requires explicit expressions of $\bbD_{\gamma} g$, which is the influence function (refer to \citep{Ichimura&Newey:2017} for more discussions on the calculation of the influence function), and $\bbD_{\gamma\gamma}^2 g$. In some cases, it can be very challenging to compute these derivatives. However, when they are available in analytical forms, the computation cost is lower than the multi-scale jackknife method, for that one only needs to conduct the estimation with one bandwidth.

\subsection{Extension to discontinuous functionals} \label{subsec:discontinuous}

In many applications, the semiparametric estimator is a sample average of some discontinuous functional of the first-step nonparametric estimator. In this subsection, we are going to demonstrate that our framework can be extended to such case if there exists a sufficiently smooth projection of the discontinuous functional. 

\begin{assumption}[ALQP---Asymptotic Linearity in $\check{g}$ with a Quadratic Projection] \label{asmp:AL'} 
Assume that the semiparametric estimator $\check{\theta}_n$ is \textit{asymptotically linear} in a discontinuous functional $\check{g}$:
\begin{align*}
	\check{\theta}_n - \theta_0 = \calJ_n \widecheck{G}_n(\theta_0, \hatgamma_n) + \oP(n^{-1/2}) = \frac{1}{n} \sum_{i=1}^n \calJ_n \, \check{g}(z_i, \theta_0, \hatgamma_n) +  \oP(n^{-1/2}),
\end{align*}
where $\calJ_n \conP \calJ_0$ for some non-random and non-zero $\calJ_0$, and the functional $\check{g}$ satisfies that $\widecheck{G}(\theta_0,\gamma_0) = \bbE[\check{g}(z,\theta_0,\gamma_0)]=0$.  

Moreover, there exists a continuous functional $g$ satisfying Assumption \ref{asmp:quadraticity} and $\bbE[ \check{g}(z_i,\theta,\gamma) ] = \bbE[ g(z_i,\theta,\gamma)]$, $\forall i=1,\cdots,n$, in an open set containing $(\theta_0,\gamma_0)$.
\end{assumption}

Intuitively, the functional $g$ is a smooth projection of $\check{g}$ on some sub-$\sigma$-algebra of the $\sigma$-algebra generated by the sample. Let $\hattheta_n$ be the corresponding estimator defined by $g$. Under Assumption \ref{asmp:AL'} and those conditions of Lemma \ref{lem:op1}, we obtain
\begin{align*}
	& \check{\theta}_n - \theta_0 = (\check{\theta}_n - \hattheta_n) + (\hattheta_n - \theta_0) \\
	=\,& \calJ_n \Big( \widecheck{G}_n(\theta_0, \hatgamma_n) - \whatG_n(\theta_0, \hatgamma_n) + \whatG_n(\theta_0, \gamma_0) +  \whatG_{n,\gamma}^{\,\prime}(\theta_0,\gamma_0, \hatgamma_n - \bargamma_n) 
		+ \Banb + \Bnl \Big) + \oP(n^{-1/2}).
\end{align*}
The property of $g$ implies that $\bbE[ \widecheck{G}_n(\theta_0, \hatgamma_n) - \whatG_n(\theta_0, \hatgamma_n) ] = 0$. That is, the difference $\widecheck{G}_n(\theta_0, \hatgamma_n) - \whatG_n(\theta_0, \hatgamma_n)$ does not contain any biases. Intuitively, it is the sample average of the difference between $\check{g}$ and its smooth projection $g$. Hence, it is reasonable to expect that this difference is asymptotically normal, under certain regularity conditions. 

\begin{assumption}[$\text{AN}'$---Asymptotic Normality] \label{asmp:AN'}
Suppose that there exists a non-random and positive definite $\widecheck{\Sigma}_g$ such that
\begin{align*}
	\widecheck{G}_n(\theta_0, \hatgamma_n) - \whatG_n(\theta_0, \hatgamma_n) + \whatG_n(\theta_0, \gamma_0) + \whatG_{n,\gamma}^{\,\prime}(\theta_0,\gamma_0, \hatgamma_n - \bargamma_n)  \conL \mathcal{N}(0, \widecheck{\Sigma}_g).
\end{align*}
\end{assumption}

\begin{named-exmp}[Hit Rates]
Consider the hit rates example discussed by \citep{CLK:2003}. Let $z=(y,x^\intercal)^\intercal$, where $y$ is a scalar dependent variable and $x\in\bbR^{d_x}$ is a continuous covariate with density $\gamma_0$. The parameter of interest is $\theta_0 = \bbE[ \bbOne(y\geq\gamma_0(x) ) ] =\bbE\big[1 - F_{y|x}\big( \gamma_0(x) | x \big)\big]$, where $F_{y|x}$ is the conditional distribution of $y$ given $x$. Consider a kernel-based semiparametric estimator
\begin{align*}
	\check{\theta}_n = \frac{1}{n} \sum_{i=1}^n \bbOne\big(y_i \geq \hatgamma_n(x_i) \big), \quad \hatgamma(x_i) = \frac{1}{n} \sum_{j\neq i} K_h(z_i - z_j).
\end{align*}
Let $\check{g}(z,\theta,\gamma) = \bbOne\big( y \geq \gamma(x) \big) - \theta$ and $g(z,\theta,\gamma) = \bbE[ \check{g}(z,\theta,\gamma) | x ] = 1 - F_{y|x}\big( \gamma(x) | x \big) - \theta$. Let $\calX_n$ be the $\sigma$-algebra generated by $\{x_i\}_{i=1}^n$. Then we have
\begin{align*}
	\widecheck{G}_n(\theta_0, \hatgamma_n) - \whatG_n(\theta_0, \hatgamma_n) = \frac{1}{n} \sum_{i=1}^{n} \Big( \bbOne\big( y_i \geq \hatgamma_n(x) \big) - 1 + F_{y|x}\big( \hatgamma_n(x_i) | \calX_n \big) \Big).
\end{align*}
The asymptotic normality of the above difference is a direct result of the central limit theory in the \text{i.i.d.} case. If we further know the correlation between this difference and $\whatG_n(\theta_0, \gamma_0) + \whatG_{n,\gamma}^{\,\prime}(\theta_0,\gamma_0, \hatgamma_n - \bargamma_n)$, as well as the variance of the latter, we will be able to find $\widecheck{\Sigma}_g$.
\end{named-exmp}

\begin{theorem}[A Summary Theorem for $\check{\theta}_n$] \label{thm:discontinuous}
(i) Suppose that Assumptions \ref{asmp:BO}, \ref{asmp:AL'}, and \ref{asmp:AN'} hold true. Assume that $\calJ_n - \calJ_0 = \OP\big( \whatG_n(\theta_0, \hatgamma_n) \big)$. If $s>1/4$ and $r>1/6$, then we have
\begin{align*}
	\sqrt{n} \big( \check{\theta}_n - \theta_0 - \calJ_n \calB^{\texttt{NL}} - \calJ_n \calB^{\texttt{ANB}} \big) \conL \calN(0, \calJ_0 \, \widecheck{\Sigma}_g \, \calJ_0^\intercal).
\end{align*} 


(ii) The assumptions of part (i) and Theorem \ref{thm:jackknifeCLT} are all true. Then $\sqrt{n}( \check{\theta}_n^w - \theta_0 ) \conL \calN(0, \widecheck{\Sigma}_\theta^w)$ with $\widecheck{\Sigma}_\theta^w \coloneqq \calJ_0 \, \widecheck{\Sigma}_g^w \calJ_0^\intercal$, where $\widecheck{\Sigma}_g^w$ is the asymptotic variance of 
\begin{align*}
	\sqrt{n} \Big( \sum_{q=1}^Q w_q \big( \widecheck{G}_n(\theta_0, \hatgamma_n(h_q)) - \whatG_n(\theta_0, \hatgamma_n(h_q)) \big) + \whatG_n(\theta_0, \gamma_0) + \whatG_{n,\gamma}^{\,\prime}(\theta_0,\gamma_0, \hatgamma_n^w - \bargamma_n^w) \Big).
\end{align*}
\end{theorem}

The counterpart of Theorem \ref{thm:analytical} seems to be more complicated, for that the smooth projection $g$ may be unknown, as shown in the hit rates example. In such case, we also need to account for the errors and biases that arise from the estimation of $g, g_{\gamma}^{\,\prime}$ and $g_{\gamma\gamma}^{\,\prime\prime}$. Hence, we leave this to future exploration.

\section{Simulation Study} \label{sec:simulation}

We have conducted a Monte Carlo experiment to investigate the finite-sample performance of the multi-scale jackknife (MSJ) method and the analytical bias correction (ABC) method. We considered three different estimators: (1) the average density (AD) estimator, (2) the integrated squared density (ISD) estimator, and (3) the density-weighted average derivative (DWAD) estimator.

In the first two cases, we considered a one-dimensional mixed normal density given by
\begin{align*}
    \gamma_0(x) = \alpha \phi(x;\mu_1, \sigma_1^2) + (1-\alpha) \phi(x;\mu_2, \sigma_2^2),
\end{align*}
where $\mu_1 = -2$, $\sigma_1^2 = 0.5$, $\mu_2 = 1$, $\sigma_2^2 = 1$, and $\alpha=0.4$. The true parameter of interest $\theta_0 = \bbE(\gamma_0(X))$ is given by
\begin{align*}
    \theta_0 = \frac{\alpha^2}{\sqrt{4 \sigma_1^2 \pi}} + \frac{(1-\alpha)^2}{\sqrt{4 \sigma_2^2 \pi}} + 2 \frac{\alpha (1-\alpha)}{\sqrt{ 2\pi(\sigma_1^2 + \sigma_2^2) }} \, \exp\Big( - \frac{1}{2} \frac{(\mu_1 - \mu_2)^2}{\sigma_1^2 + \sigma_2^2} \Big) = 0.0796.
\end{align*}

In the last case, we are interested in estimating
\begin{align*}
    \theta_0 = \bbE\big( \gamma_0(X) \, \partial_X \bbE( Y | X) \big) = -2 \, \bbE( \partial_X \gamma_0(X) Y ),
\end{align*}
where $\gamma_0(\cdot)$ is the density of $X$. We considered a linear model
\begin{align*}
    y_i = x_i^\intercal \beta + \epsilon_i, \,\, x_i \sim \calN(0, I_d), \epsilon_i \sim \calN(0,1).
\end{align*}
For simplicity, we let $\beta=1_d$, a $d$-dimensional vector with all  the elements being one, and focus on estimating $\theta_{01}$.

We employed a Gaussian kernel in all cases. So the order of the kernel is $m=2$ across all cases. We considered three different sample sizes: $n=50, 100,$ and $200$. In each case, we conducted 1,000 simulations. To save space, we only report the results with $n=100$. Refer to the online supplement for more results.

\begin{figure}[htb!]
\centering
\includegraphics[width=1\textwidth]{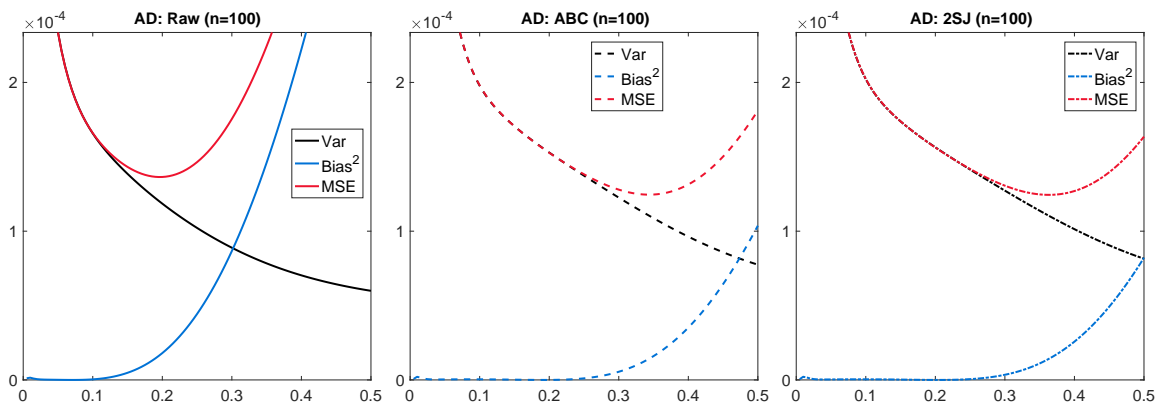}
\caption{AD: Decomposition of Mean Squared Error}
\label{fig:AD100MSE}
\end{figure}

Figure \ref{fig:AD100MSE} shows the decomposition of mean squared error (MSE) for various AD estimators, at different bandwidth values. From left to right, it presents the result for the raw estimator without any bias correction, the analytical bias-corrected (ABC) estimator, and the two-scale jackknife (2SJ) estimator (with $\eta=(1,5/4)$), respectively. 

Since the raw estimator is linear in the kernel function, there is no nonlinear bias $\Bnl$. As shown in the figure, the bias starts to increase with the bandwidth $h$ when $h>0.1$ for the raw estimator. While for the other two estimators, this only occurs approximately when $h>0.25$. In other words, both ABC and 2SJ successfully removed the bias for a substantially large range of bandwidths. For larger values of $h$, although there is still bias left in the ABC and 2SJ estimators, it has been largely reduced. Consequently, the inference based on either ABC or 2SJ will be much less sensitive to the choice of bandwidth. 

\begin{figure}[htb!]
\centering
\includegraphics[width=1\textwidth]{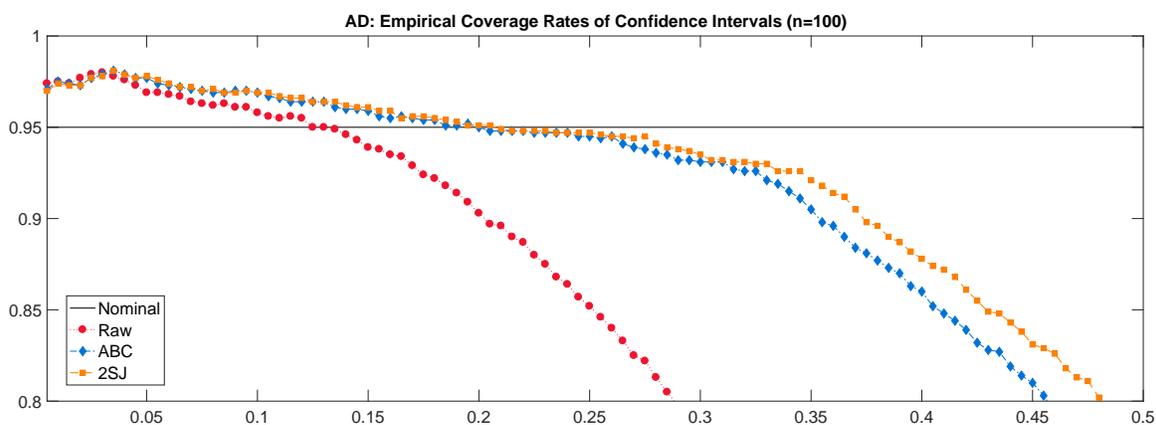}
\caption{AD: Empirical Coverage Rates of Confidence Intervals}
\label{fig:AD100CI}
\end{figure}

For any given bandwidth value, the variance parts of the ABC and 2SJ estimators are larger than that of the raw one. We think these are due to some finite sample effects. As shown by \citep{NHR:2004}, the variance of the twicing-kernel-based semiparametric estimator only depends on the true function(s), not the kernel (cf. the notation following (2.2) therein). This implies that the asymptotic variances of the ABC and the raw estimators should be the same. However, the kernel may have some impacts on the finite-sample variance. While for the 2SJ estimator, it is probably because its two components are not perfectly correlated in such a finite sample. However, the increases are not that large. Hence, the ABC and 2SJ estimators can achieve slightly smaller minimum values for the MSE.

Figure \ref{fig:AD100CI} shows the empirical coverage rates for the 95\% confidence intervals (CIs) associated with the raw, ABC, and 2SJ estimators. The $x$-axis is the bandwidth. The coverage rates are about two percentage points higher than the nominal level when $h$ is small. This might be a result of slightly overestimating the asymptotic variance when $h$ is very small. Not surprisingly, the coverage rates decrease, as bias increases (in absolute value). Since the ABC and 2SJ estimators can remove/reduce bias, their corresponding coverage-rate curves have much slower decreasing rates. More importantly, the curves are nearly flat and very close to the nominal level around the region $[0.2, 0.25]$. According to Figure \ref{fig:AD100MSE}, this is a region where the bias remains very close to zero. Besides, since $h$ is not very small in this region, the variance estimators become more precise, compared to the cases with very small bandwidth values. 

\begin{figure}[htb!]
\centering
\includegraphics[width=1\textwidth]{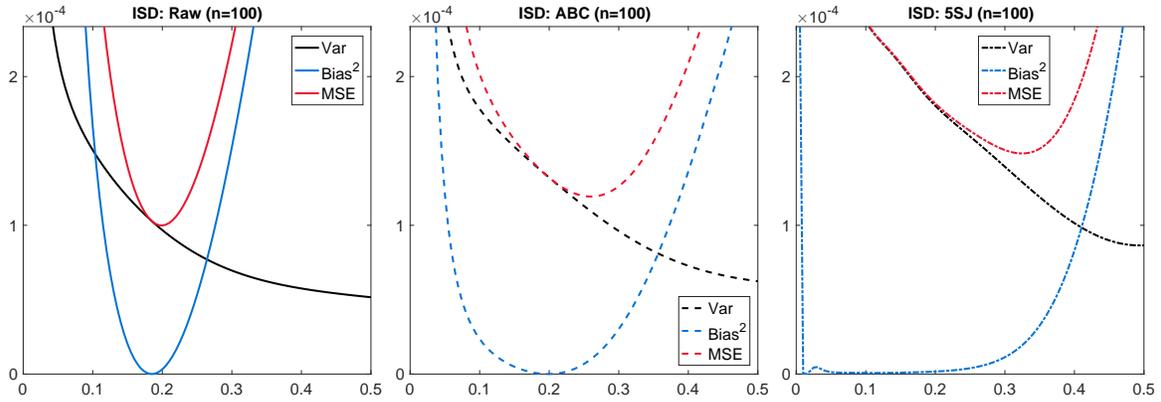}
\caption{ISD: Decomposition of Mean Squared Error}
\label{fig:ISD100MSE}
\end{figure}

\begin{figure}[htb!]
\centering
\includegraphics[width=1\textwidth]{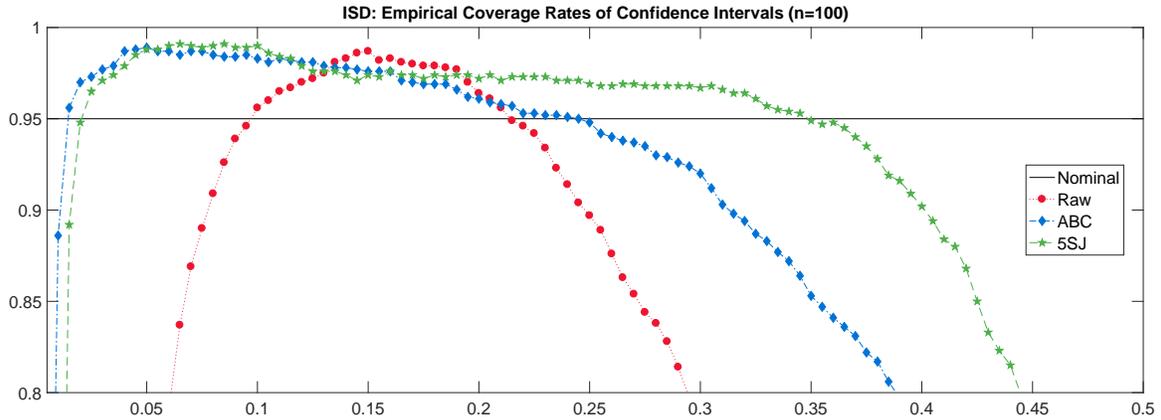}
\caption{ISD: Empirical Coverage Rates of Confidence Intervals}
\label{fig:ISD100CI}
\end{figure}

Figure \ref{fig:ISD100MSE} presents the MSE decomposition results for various ISD estimators. In this case, both the two biases are non-zero. The nonlinear bias $\Bnl$ is positive, while the average nonparametric bias $\Banb$ is negative. This explains why there is a point where the overall bias is zero. Once deviating from this point, the overall bias increases rapidly in magnitude. The ABC method can substantially reduce both biases. One can construct 2SJ to remove/reduce either the nonlinear bias or the average nonparametric bias. However, we found that 3SJ, which is the counterpart to ABC in this scenario, can only effectively remove the nonlinear bias. Hence, we tried higher-scale jackknife and found that 5SJ has a much better performance (we set $\eta=(3/5,4/5,1,6/5,7/5)$). 

According to Figure \ref{fig:ISD100CI}, the coverage rates of the raw estimator are quite sensitive to the bandwidth, which is consistent with the MSE decomposition result. For the ABC and 5SJ estimators, the coverage rates are more robust to the bandwidth, especially in the latter case. This is not surprising, for that 5SJ can remove/reduce more biases by design. Generally speaking, the coverage rates are higher than the nominal level when the overall bias level is relatively small. One possible explanation is that although the true asymptotic variance of the ISD estimator is the same as that of the AD estimator, we employed a more nonlinear estimator, which may be subject to more sources of finite-sample biases, to estimate it in the ISD case.

\begin{figure}[htb!]
\centering
\includegraphics[width=1\textwidth]{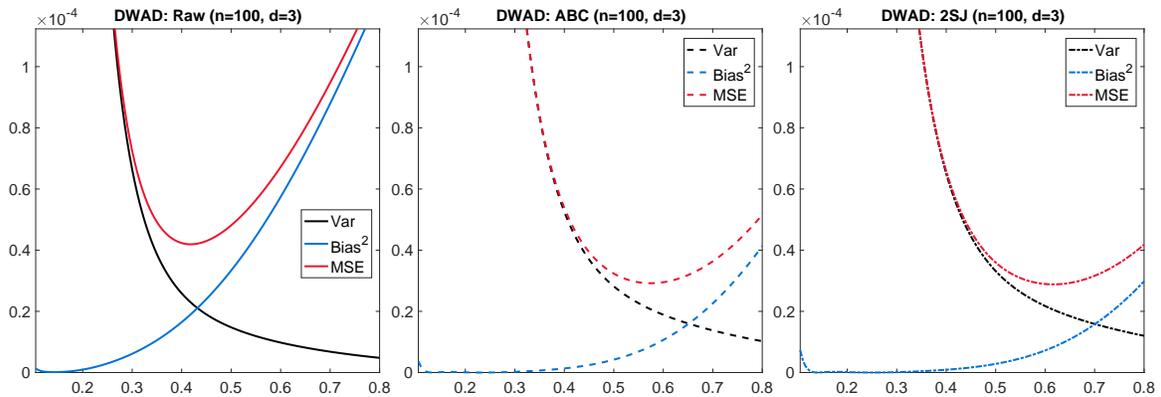}
\caption{DWAD: Decomposition of Mean Squared Error}
\label{fig:DWAD100MSE}
\end{figure}

\begin{figure}[htb!]
\centering
\includegraphics[width=1\textwidth]{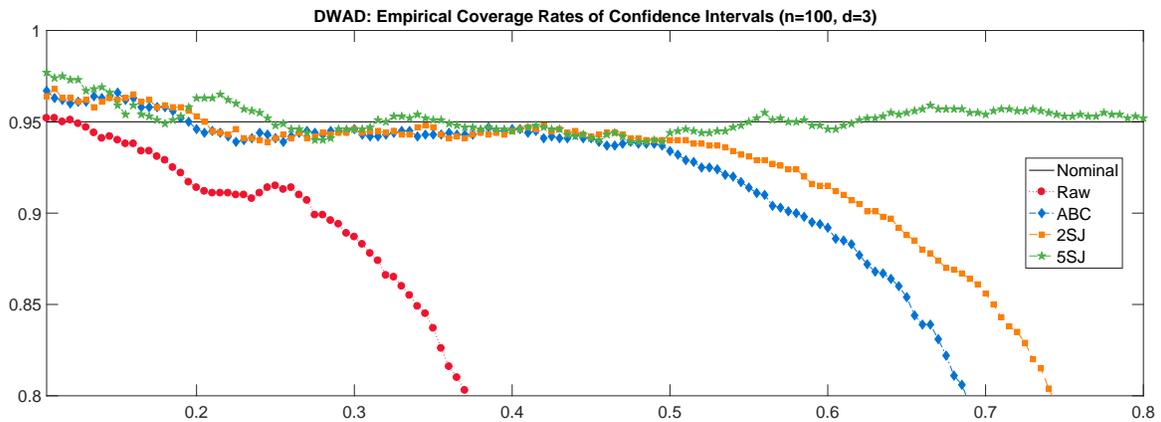}
\caption{DWAD: Empirical Coverage Rates of Confidence Intervals}
\label{fig:DWAD100CI}
\end{figure}

For the DWAD estimator, we present the results with $d=3$, which is larger than the order of the Gaussian kernel ($m=2$). The general patterns are the same as above. In this case, the MSE gains for the ABC and 2SJ estimators are more noticeable. 
When constructing the confidence intervals, we used the variance estimator proposed by \citep{Cattaneo&Crump&Jansson:2014} (Case (b) of Theorem 2 therein), while the one considered by \citep{PSS:1989} leads to over-coverage. The under-coverage of the CI based on the raw estimator is mainly due to the bias. In other cases, the coverage rates are pretty close to the nominal level, when the remaining biases are small. In particular, since the five-scale jackknife estimator successively removes bias for a large range of bandwidth, its CI continues to have good coverage rates across all the bandwidths considered in the simulation.

%

%

\section{Conclusion} \label{sec:conclusion}

This paper extends the classic framework on semiparametric two-step models, which is developed by \citep{Andrews:1994}, \citep{Newey:1994}, and \citep{Newey&McFadden:1994}, to allow for possibly low-precision nonparametric estimator. We have shown that there are two (or even more) different types of biases in the semiparametric estimator, when its nonparametric ingredient has a slower-than-$n^{1/4}$ convergence rate. We also have proposed two different methods to correct for these biases: one is multi-scale jackknife, the other is analytical-based bias correction. Our simulation study suggests that these bias-correction methods work quite well in finite samples for various kernel-based semiparametric two-step estimators.

\bibliography{YangBib}

\begin{thebibliography}{34}
\newcommand{\enquote}[1]{``#1''}
\expandafter\ifx\csname natexlab\endcsname\relax\def\natexlab#1{#1}\fi

\bibitem[\protect\citeauthoryear{Andrews}{Andrews}{1994}]{Andrews:1994}
\textsc{Andrews, D. W.~K.} (1994): \enquote{Asymptotics for Semiparametric
  Econometric Models via Stochastic Equicontinuity,} \emph{Econometrica}, 62,
  43--72.

\bibitem[\protect\citeauthoryear{Andrews and Mikusheva}{Andrews and
  Mikusheva}{2016}]{IAndrews:2016Functional}
\textsc{Andrews, I. and A.~Mikusheva} (2016): \enquote{Conditional Inference
  with Functional Nuisance Parameter,} \emph{Econometrica}, 84, 1571--1612.

\bibitem[\protect\citeauthoryear{Bierens}{Bierens}{1987}]{Bierens:1987kernel}
\textsc{Bierens, H.~J.} (1987): \enquote{Kernel Estimators of Regression
  Functions,} in \emph{Advances in econometrics: Fifth World Congress}, vol.~1,
  99--144.

\bibitem[\protect\citeauthoryear{Borovskikh}{Borovskikh}{1996}]{Ustat:1996}
\textsc{Borovskikh, Y.~V.} (1996): \emph{U-statistics in {B}anach Spaces},
  V.S.P. Intl Science.

\bibitem[\protect\citeauthoryear{Calonico, Cattaneo, and Titiunik}{Calonico
  et~al.}{2014}]{CCT:2014}
\textsc{Calonico, S., M.~D. Cattaneo, and R.~Titiunik} (2014): \enquote{Robust
  nonparametric confidence intervals for regression-discontinuity designs,}
  \emph{Econometrica}, 82, 2295--2326.

\bibitem[\protect\citeauthoryear{Cattaneo, Crump, and Jansson}{Cattaneo
  et~al.}{2010}]{Cattaneo&Crump&Jansson:2010}
\textsc{Cattaneo, M.~D., R.~K. Crump, and M.~Jansson} (2010): \enquote{Robust
  Data-Driven Inference for Density-Weighted Average Derivatives,}
  \emph{Journal of the American Statistical Association}, 105, 1070--1083.

\bibitem[\protect\citeauthoryear{Cattaneo, Crump, and Jansson}{Cattaneo
  et~al.}{2013}]{Catteneo&Crump&Jansson:2013}
---\hspace{-.1pt}---\hspace{-.1pt}--- (2013): \enquote{Generalized Jackknife
  Estimators of Weighted Average Derivatives,} \emph{Journal of the American
  Statistical Association}, 108, 1243--1256.

\bibitem[\protect\citeauthoryear{Cattaneo, Crump, and Jansson}{Cattaneo
  et~al.}{2014}]{Cattaneo&Crump&Jansson:2014}
---\hspace{-.1pt}---\hspace{-.1pt}--- (2014): \enquote{Small Bandwidth
  Asymptotics for Density-Weightede Average Derivatives,} \emph{Econometric
  Theory}, 30, 176--200.

\bibitem[\protect\citeauthoryear{Cattaneo and Jansson}{Cattaneo and
  Jansson}{2018}]{Cattaneo&Jansson:2018}
\textsc{Cattaneo, M.~D. and M.~Jansson} (2018): \enquote{Kernel-Based
  Semiparametric Estimators: Small Bandwidth Asymptotics and Bootstrap
  Consistency,} \emph{Econometrica}, 86, 955--995.

\bibitem[\protect\citeauthoryear{Chen}{Chen}{2007}]{Chen:2007}
\textsc{Chen, X.} (2007): \enquote{Large sample sieve estimation of
  semi-nonparametric models,} in \emph{Handbook of Econometrics}, ed. by J.~J.
  Heckman and E.~E. Leamer, Elsevier, vol.~6, 5549--5632.

\bibitem[\protect\citeauthoryear{Chen, Linton, and van Keilegom}{Chen
  et~al.}{2003}]{CLK:2003}
\textsc{Chen, X., O.~Linton, and I.~van Keilegom} (2003): \enquote{Estimation
  of Semiparametric Models When The Criterion Function Is Not Smooth,}
  \emph{Econometrica}, 71, 1591--1608.

\bibitem[\protect\citeauthoryear{Chernozhukov, Chetverikov, Demirer, Duflo,
  Hansen, and Newey}{Chernozhukov et~al.}{2017}]{CCDDHN:2017double}
\textsc{Chernozhukov, V., D.~Chetverikov, M.~Demirer, E.~Duflo, C.~Hansen, and
  W.~Newey} (2017): \enquote{Double/Debiased/{N}eyman Machine Learning of
  Treatment Effects,} \emph{American Economic Review}, 107, 261--265.

\bibitem[\protect\citeauthoryear{Chernozhukov, Chetverikov, Demirer, Duflo,
  Hansen, Newey, and Robins}{Chernozhukov
  et~al.}{2018{\natexlab{a}}}]{CCDDHN:2018double}
\textsc{Chernozhukov, V., D.~Chetverikov, M.~Demirer, E.~Duflo, C.~Hansen,
  W.~Newey, and J.~Robins} (2018{\natexlab{a}}): \enquote{Double/Debiased
  Machine Learning for Treatment and Structural Parameters,} \emph{The
  Econometrics Journal}, 21, C1--C68.

\bibitem[\protect\citeauthoryear{Chernozhukov, Escanciano, Ichimura, Newey, and
  Robins}{Chernozhukov et~al.}{2018{\natexlab{b}}}]{CEINR:2018}
\textsc{Chernozhukov, V., J.~C. Escanciano, H.~Ichimura, W.~K. Newey, and J.~M.
  Robins} (2018{\natexlab{b}}): \enquote{Locally Robust Semiparametric
  Estimation,} Working paper.

\bibitem[\protect\citeauthoryear{Chernozhukov, Newey, and Robins}{Chernozhukov
  et~al.}{2018{\natexlab{c}}}]{CNR:2018}
\textsc{Chernozhukov, V., W.~K. Newey, and J.~M. Robins} (2018{\natexlab{c}}):
  \enquote{Double/De-Biased Machine Learning Using Regularized Riesz
  Representers,} Working paper.

\bibitem[\protect\citeauthoryear{Dehling}{Dehling}{2006}]{Ustat:2006}
\textsc{Dehling, H.} (2006): \enquote{Limit theorems for dependent
  U-statistics,} in \emph{Dependence in Probability and Statistics}, Springer,
  65--86.

\bibitem[\protect\citeauthoryear{Hampel}{Hampel}{1974}]{Hampel:1974}
\textsc{Hampel, F.~R.} (1974): \enquote{The influence curve and its role in
  robust estimation,} \emph{Journal of the American Statistical Association},
  69, 383--393.

\bibitem[\protect\citeauthoryear{Hirano, Imbens, and Ridder}{Hirano
  et~al.}{2003}]{Hirano&Imbens&Ridder:2003}
\textsc{Hirano, K., G.~W. Imbens, and G.~Ridder} (2003): \enquote{Efficient
  Estimation of Average Treatment Effects Using the Estimated Propensity
  Score,} \emph{Econometrica}, 71, 1161--1189.

\bibitem[\protect\citeauthoryear{Hoeffding}{Hoeffding}{1948}]{Hoeffding:1948}
\textsc{Hoeffding, W.} (1948): \enquote{A Class of Statistics with
  Asymptotically Normal Distribution,} \emph{Annals of Statistics}, 19,
  293--325.

\bibitem[\protect\citeauthoryear{Ichimura and Newey}{Ichimura and
  Newey}{2017}]{Ichimura&Newey:2017}
\textsc{Ichimura, H. and W.~K. Newey} (2017): \enquote{The Influence Function
  of Semiparametric Estimators,} Working paper.

\bibitem[\protect\citeauthoryear{Ichimura and Todd}{Ichimura and
  Todd}{2007}]{Ichimura&Todd:2007}
\textsc{Ichimura, H. and P.~E. Todd} (2007): \enquote{Implementing
  Nonparametric and Semiparametric Estimators,} in \emph{Handbook of
  Econometrics}, ed. by J.~J. Heckman and E.~E. Leamer, Elsevier, vol.~6,
  5549--5632.

\bibitem[\protect\citeauthoryear{Kollo and von Rosen}{Kollo and von
  Rosen}{2006}]{Kollo&vonRosen:2006}
\textsc{Kollo, T. and D.~von Rosen} (2006): \emph{Advanced multivariate
  statistics with matrices}, vol. 579, Springer Science \& Business Media.

\bibitem[\protect\citeauthoryear{Korolyuk and Borovskich}{Korolyuk and
  Borovskich}{1994}]{Ustat:1994}
\textsc{Korolyuk, V.~S. and Y.~V. Borovskich} (1994): \emph{Theory of
  U-statistics}, vol. 273 of \emph{Mathematics and Its Applications}, Springer.

\bibitem[\protect\citeauthoryear{Li, Liu, and Xiu}{Li
  et~al.}{2019}]{Li&Liu&Xiu:2019}
\textsc{Li, J., Y.~Liu, and D.~Xiu} (2019): \enquote{Efficient Estimation of
  Integrated Volatility Functionals via Multiscale Jackknife,} \emph{The Annals
  of Statistics}, 47, 156--176.

\bibitem[\protect\citeauthoryear{Nadaraya}{Nadaraya}{1964}]{Nadaraya:1964}
\textsc{Nadaraya, E.~A.} (1964): \enquote{On estimating regression,}
  \emph{Theory of Probability \& Its Applications}, 9, 141--142.

\bibitem[\protect\citeauthoryear{Newey}{Newey}{1994}]{Newey:1994}
\textsc{Newey, W.~K.} (1994): \enquote{The Asymptotic Variance of
  Semiparametric Estimators,} \emph{Econometrica}, 1349--1382.

\bibitem[\protect\citeauthoryear{Newey, Hsieh, and Robins}{Newey
  et~al.}{2004}]{NHR:2004}
\textsc{Newey, W.~K., F.~Hsieh, and J.~M. Robins} (2004): \enquote{Twicing
  Kernels and A Small Bias Property of Semiparametric Estimators,}
  \emph{Econometrica}, 72, 947--962.

\bibitem[\protect\citeauthoryear{Newey and McFadden}{Newey and
  McFadden}{1994}]{Newey&McFadden:1994}
\textsc{Newey, W.~K. and D.~McFadden} (1994): \enquote{Large Sample Estimation
  and Hypothesis Testing,} in \emph{Handbook of Econometrics}, ed. by R.~F.
  Engle and D.~L. McFadden, Elsevier, vol.~4, 2111--2245.

\bibitem[\protect\citeauthoryear{Powell, Stock, and Stoker}{Powell
  et~al.}{1989}]{PSS:1989}
\textsc{Powell, J.~L., J.~H. Stock, and T.~M. Stoker} (1989):
  \enquote{Semiparametric Estimation of Index Coefficients,}
  \emph{Econometrica}, 57, 1403--1430.

\bibitem[\protect\citeauthoryear{Quenouille}{Quenouille}{1949}]{Quenouille:1949AMS}
\textsc{Quenouille, M.~H.} (1949): \enquote{Problems in Plane Sampling,}
  \emph{The Annals of Mathematical Statistics}, 20, 335--375.

\bibitem[\protect\citeauthoryear{Schucany and Sommers}{Schucany and
  Sommers}{1977}]{Schucany&Sommers:1977}
\textsc{Schucany, W. and J.~P. Sommers} (1977): \enquote{Improvement of Kernel
  Type Density Estimators,} \emph{Journal of the American Statistical
  Association}, 72, 420--423.

\bibitem[\protect\citeauthoryear{Stuetzle and Mittal}{Stuetzle and
  Mittal}{1979}]{Stuetzle&Mittal:1979}
\textsc{Stuetzle, W. and Y.~Mittal} (1979): \enquote{Some Comments on The
  Asymptotic Behavior of Robust Smoothers,} in \emph{Smoothing Techniques for
  Curve Estimation}, Springer, vol. 757 of \emph{Lecture Notes in Mathematics},
  191--195.

\bibitem[\protect\citeauthoryear{Watson}{Watson}{1964}]{Watson:1964}
\textsc{Watson, G.~S.} (1964): \enquote{Smooth regression analysis,}
  \emph{Sankhy{\=a}: The Indian Journal of Statistics, Series A}, 359--372.

\bibitem[\protect\citeauthoryear{Yang}{Yang}{2020}]{Yang:2020Semi}
\textsc{Yang, X.} (2020): \enquote{Semiparametric Estimation in
  Continuous-Time: Asymptotics for Integrated Volatility Functionals with Small
  and Large Bandwidths,} \emph{Journal of Business \& Economic Statistics},
  forthcoming.

\end{thebibliography}

\begin{appendices}

\bigskip

\tableofcontents

\bigskip
\bigskip

\section{Proofs} \label{sec:proofs}

Throughout this section, we let $C$ to denote some positive finite constant, the value of which may vary from line to line.

\subsection{Sufficient conditions for Assumption \ref{asmp:BO}}

Let $\{z_i\}_{i=1}^n$ be an \text{i.i.d.} sample and $\hatgamma_n$ be an estimator of $\gamma_0$. Define $\bargamma_n(z_i) \coloneq \bbE[\hatgamma_n(z_i) | z_i]$. Suppose that
\begin{align*}
	\hatgamma_n(z_i) - \bargamma_n(z_i) = \frac{1}{n-1} \sum_{j\neq i} \big( \psi(z_i, z_j) - \bbE[ \psi(z_i, z_j) | z_i ] \big) = \frac{1}{n-1} \sum_{j\neq i} \phi(z_i, z_j).
\end{align*}

\begin{lemma}[Sufficient conditions for Assumption \ref{asmp:BO}] \label{lem:BO}
Suppose the following conditions hold for $i\neq j\neq l$ and some finite number $C$:
\begin{align*}
\begin{gathered} 
	\bbE[ \| g_{\gamma}^{\prime}(z_i, \theta_0, \gamma_0, \bargamma_n(z_i) - \gamma_0(z_i) ) \| ] \leq C \, \bbE[ \| \bargamma_n(z_i) - \gamma_0(z_i) \| ] \\
	\Var[ g_{\gamma}^{\prime}(z_i, \theta_0, \gamma_0, \bargamma_n(z_i) - \gamma_0(z_i) ) ] \leq C \, \Var[ \bargamma_n(z_i) - \gamma_0(z_i) ],	\\
	\bbE[ \| g_{\gamma\gamma}^{\prime\prime}(z_i, \theta_0, \gamma_0, \phi(z_i,z_j), \phi(z_i,z_j)) \| ] \leq C \, \bbE[ \| \psi(z_i,z_j) \|^2 ], \\
	\Var\big( \bbE[ g_{\gamma\gamma}^{\prime\prime}(z_i, \theta_0, \gamma_0, \phi(z_i,z_j), \phi(z_i,z_j)) ] \big) \leq C \, \bbE\big( \| \psi(z_i, z_j) \|^4 \big), \\
	\Var\big( \bbE[ g_{\gamma\gamma}^{\prime\prime}(z_i, \theta_0, \gamma_0, \phi(z_i,z_j), \phi(z_i,z_j)) | z_i ] \big) \leq C \, \bbE\big( \bbE[ \| \psi(z_i, z_j) \|^2 | z_i ]^2 \big), \\
	\Var\big( \bbE[ g_{\gamma\gamma}^{\prime\prime}(z_i, \theta_0, \gamma_0, \phi(z_i,z_j), \phi(z_i,z_j)) | z_j ] \big) \leq C \, \bbE\big( \bbE[ \| \psi(z_i, z_j) \|^2 | z_j ]^2 \big), \\
	\Var\big( \bbE[ g_{\gamma\gamma}^{\prime\prime}(z_i, \theta_0, \gamma_0, \phi(z_i,z_j), \phi(z_i,z_l))] \big) \leq C \, \bbE\big[ \| \psi(z_i,z_j) \|^2 \| \psi(z_i, z_l) \|^2 \big], \\
	\Var\big( \bbE[ g_{\gamma\gamma}^{\prime\prime}(z_i, \theta_0, \gamma_0, \phi(z_i,z_j), \phi(z_i,z_l)) | z_j, z_l ] \big) \leq C \, \bbE\big[ \bbE[ \| \psi(z_i,z_j) \psi(z_i, z_l) \| | z_j, z_l]^2 \big],
\end{gathered}
\end{align*}
where the orders of the four right hand side terms depend on some turning parameter. If one can choose the tuning parameter in such a way that 
\begin{align}
\begin{gathered} \label{eq:BOcond2}
	 \bbE[ \| \bargamma_n(z_i) - \gamma_0(z_i) \| ] = O(n^{-s}), \quad \Var[ \bargamma_n(z_i) - \gamma_0(z_i) ] = o(1), \\
	 \bbE[ \| \psi(z_i,z_j) \|^2 ] = O(n^{1-2r}), \quad  \bbE[ \| \psi(z_i,z_j) \|^4 ] = o(n^3), \\
	 \bbE\big( \bbE[\psi(z_i, z_j)^2 | z_j ]^2 \big) = o(n^2), \quad,  \bbE\big( \bbE[ \| \psi(z_i, z_j) \|^2 | z_i ]^2 \big) = o(n^2) \\
	 \bbE\big[ \| \psi(z_i,z_j) \|^2 \| \psi(z_i, z_l) \|^2 \big] = o(n^2), \quad \bbE\big[ \bbE[ \| \psi(z_i,z_j) \psi(z_i, z_l) \| | z_j, z_l]^2 \big] = o(n),
\end{gathered}
\end{align}
then Assumption \ref{asmp:BO} is satisfied.
\end{lemma}

\begin{proof}

According to the expression of $\hatgamma_n(z_i) - \bargamma_n(z_i)$, we can obtain
\begin{align*}
	\whatG_{n,\gamma\gamma}^{\,\prime\prime} \big(\theta_0,\gamma_0, \hatgamma_n - \bargamma_n, \hatgamma_n - \bargamma_n \big) = \frac{1}{n-1} U_{n,1} + \frac{n-2}{n-1} U_{n,2},
\end{align*}
where $U_{n,1}$ and $U_{n,2}$ are two U-statistics:
\begin{align*}
	U_{n,1} &= \frac{1}{n(n-1)} \sum_{i\neq j} g_{\gamma\gamma}^{\prime\prime}(z_i, \theta_0, \gamma_0, \phi(z_i,z_j), \phi(z_i,z_j)), \\
	U_{n,2} &= \frac{1}{n(n-1)(n-2)} \sum_{i\neq j\neq l} g_{\gamma\gamma}^{\prime\prime}(z_i, \theta_0, \gamma_0, \phi(z_i,z_j), \phi(z_i,z_l)).
\end{align*}

For $i\neq j \neq l$, define 
\begin{gather*}
	\eta_1(z_{i}, z_{j}) = \frac{1}{2} \sum_{ i_1 \neq i_2 \in \{ i, j\}} g_{\gamma\gamma}^{\prime\prime}(z_{i_1}, \theta_0, \gamma_0, \phi(z_{i_1}, z_{i_2}), \phi(z_{i_1}, z_{i_2})), \\
	\eta_2(z_{i}, z_{j}, z_{l}) = \frac{1}{6} \sum_{ i_1 \neq i_2 \neq i_3 \in \{ i, j, l\}} g_{\gamma\gamma}^{\prime\prime}(z_{i_1}, \theta_0, \gamma_0, \phi(z_{i_1}, z_{i_2}), \phi(z_{i_1}, z_{i_3})).
\end{gather*}
Then these two functions are symmetric in its arguments, hence are the kernels of $U_{n,1}$ and $U_{n,2}$ respectively. We can then re-write the two U-statistics as
\begin{align*}
	U_{n,1} = \left( \begin{matrix}
		n \\
		2
	\end{matrix} \right)^{-1} \sum_{i<j} \eta_1(z_i, z_j) \quad \text{and} \quad
	U_{n,2} = \left( \begin{matrix}
			n \\
			3
		\end{matrix} \right)^{-1} \sum_{i<j<l} \eta_2(z_i, z_j, z_l).
\end{align*}
According to the assumptions of Lemma \ref{lem:BO}, we have
\begin{align*}
	\calB^{\texttt{NL}} = \frac{1}{n-1} \bbE[U_{n,1}] = \frac{1}{n-1} \bbE[ g_{\gamma\gamma}^{\prime\prime}(z_i, \theta_0, \gamma_0, \phi(z_i,z_j), \phi(z_i,z_j)) ].
\end{align*}

Following \cite{Hoeffding:1948}, we define
\begin{align*}
	\eta_{1,1}(x_1, \cdots, x_c) = \bbE[ \eta_1(x_1,\cdots,x_c, z_{c+1}, \cdots, z_2) ], \quad c=1,2.
\end{align*}
Hence, $\eta_{1,2}(z_1, z_2) = \eta_1(z_1, z_2)$ and $\eta_{1,1}(z_1) = \bbE[\eta_1(z_1, z_2) | z_1]$. The results given by \cite{Hoeffding:1948} imply that
\begin{align*}
	\Var(U_{n,1}) &= \left( \begin{matrix}
			n \\
			2
		\end{matrix} \right)^{-1} \sum_{c=1}^2 \left( \begin{matrix}
			 2  \\
			 c  \\
		\end{matrix} \right) \left( \begin{matrix}
			 n-2  \\
			 2-c  \\
		\end{matrix} \right) \Var(\eta_{1,c}) \\
		&= \frac{2}{n(n-1)} \Big( 2(n-2) \Var(\eta_{1,1}) + \Var(\eta_{1,2}) \Big),
\end{align*}
According to the assumption, we get
\begin{align*}
	\Var(\eta_{1,1}) &= \Var\big( \bbE[\eta_1(z_1, z_2) | z_1] \big) \leq C \, \Big( \bbE\big( \bbE[ \| \psi(z_i, z_j) \|^2 | z_i ]^2 \big) + \bbE\big( \bbE[ \| \psi(z_i, z_j) \|^2 | z_j ]^2 \big) \Big), \\
	\Var(\eta_{1,2}) &= \Var\big( \eta_1(z_1, z_2) \big) \leq C \, \bbE[ \| \psi(z_i, z_j) \|^4 ].
\end{align*}
Therefore, we get
\begin{align*}
	& \frac{1}{n-1} \big( U_{n,1} - \bbE[U_{n,1}] \big) \\
	=\,& \OP\Big( n^{-3/2} \sqrt{ \bbE\big( \bbE[\psi(z_i, z_j)^2 | z_i ]^2 \big) + \bbE\big( \bbE[\psi(z_i, z_j)^2 | z_j ]^2 \big) }\Big) + \OP\Big( n^{-5/2} \sqrt{\bbE[ \| \psi(z_i, z_j) \|^4 ]} \Big).
\end{align*}

As for $U_{n,2}$, we can also define $\eta_{2,c}$ for $c=1,2,3$. More specifically, we have $\eta_{2,3}(z_1, z_2, z_3) = \eta_2(z_1, z_2, z_3)$ and
\begin{align*}
	\eta_{2,2}(z_1, z_2) = \bbE[ \eta_2(z_1, z_2, z_3) | z_1, z_2 ] \quad \eta_{2,1}(z_1) = \bbE[ \eta_2(z_1, z_2, z_3) | z_1 ].
\end{align*}
According to the bi-linear property of $g_{\gamma\gamma}^{\prime\prime}$ with its last two arguments and the property of $\phi$, we can derive that
\begin{align*}
	& \bbE[ g_{\gamma\gamma}^{\prime\prime}(z_i, \theta_0, \gamma_0, \phi(z_i,z_j), \phi(z_i,z_l)) | z_i ] \\
	=\,& \bbE\big[ \bbE[ g_{\gamma\gamma}^{\prime\prime}(z_i, \theta_0, \gamma_0, \phi(z_i,z_j), \phi(z_i,z_l)) | z_i, z_j] | z_i \big] \\
	=\,& \bbE\big[ g_{\gamma\gamma}^{\prime\prime}(z_i, \theta_0, \gamma_0, \phi(z_i,z_j), \bbE[ \phi(z_i,z_l) | z_i, z_j] ) | z_i \big] \\
	=\,& \bbE\big[ g_{\gamma\gamma}^{\prime\prime}(z_i, \theta_0, \gamma_0, \phi(z_i,z_j), \bbE[ \phi(z_i,z_l) | z_i] ) | z_i \big] = 0, \\
	& \bbE[ g_{\gamma\gamma}^{\prime\prime}(z_i, \theta_0, \gamma_0, \phi(z_i,z_j), \phi(z_i,z_l)) | z_j ] \\
	=\,& \bbE\big[ \bbE[ g_{\gamma\gamma}^{\prime\prime}(z_i, \theta_0, \gamma_0, \phi(z_i,z_j), \phi(z_i,z_l)) | z_i, z_j] | z_j \big] \\
	=\,& \bbE\big[ g_{\gamma\gamma}^{\prime\prime}(z_i, \theta_0, \gamma_0, \phi(z_i,z_j), \bbE[ \phi(z_i,z_l) | z_i, z_j] ) | z_j \big] \\
	=\,& \bbE\big[ g_{\gamma\gamma}^{\prime\prime}(z_i, \theta_0, \gamma_0, \phi(z_i,z_j), \bbE[ \phi(z_i,z_l) | z_i] ) | z_j \big] = 0.
\end{align*}
Hence, $\eta_{2,1}(z_1) \equiv 0$. The above results also implies that $\bbE[U_{n,2}]=0$.

On the other hand, it can be shown that
\begin{align*}
	& \Var\big( \eta_{2,2}(z_1, z_2) \big) \leq C \, \Var\big( \bbE[g_{\gamma\gamma}^{\prime\prime}(z_i, \theta_0, \gamma_0, \phi(z_i,z_j), \phi(z_i,z_l) | z_j, z_l] \big) \\
	 \leq\,& C \, \Var\big( g_{\gamma\gamma}^{\prime\prime}(z_i, \theta_0, \gamma_0, \phi(z_i,z_j), \phi(z_i,z_l) \big) \leq C \, \bbE\big[ \bbE[\psi(z_i,z_j) \psi(z_i, z_l) | z_j, z_l]^2 \big], \\
	 & \Var\big( \eta_{2,3}(z_1, z_2, z_3) \big) \leq C \, \bbE\big[ \bbE[\psi(z_i,z_j) \psi(z_i, z_l) | z_j, z_l]^2 \big].
\end{align*}
It then follows that
\begin{align*}
	\Var( U_{n,2} ) &= 0 + \frac{6}{n(n-1)(n-2)} \Big( (n-3) \Var\big( \eta_{2,2}(z_1, z_2) \big) + \Var\big( \eta_{2,3}(z_1, z_2, z_3) \big) \Big) \\
		&\leq C \frac{1}{n^2} \bbE\big[ \bbE[ \| \psi(z_i,z_j) \psi(z_i, z_l) \| \,|\, z_j, z_l]^2 \big] +C \frac{1}{n^3} \bbE\big[ \| \psi(z_i,z_j) \|^2 \| \psi(z_i, z_l) \|^2 \big].
\end{align*} 
This implies that
\begin{align*}
	U_{n,2} = \OP\Big( n^{-1} \sqrt{ \bbE\big[ \bbE[ \| \psi(z_i,z_j) \psi(z_i, z_l) \| \,|\, z_j, z_l]^2 \big] } \Big) + \OP\Big( n^{-3/2} \sqrt{ \bbE[ \psi(z_i, z_j)^2 \psi(z_i, z_l)^2 ] }  \Big) .
\end{align*}

To sum up, if one can choose the turning parameter in such a way that
\begin{gather*}
	\bbE[ g_{\gamma\gamma}^{\prime\prime}(z_i, \theta_0, \gamma_0, \phi(z_i,z_j), \phi(z_i,z_j)) ] = O(n^{1-2r}), \quad \bbE\big( \bbE[ \| \psi(z_i, z_j) \|^2 | z_i ]^2 \big) = o(n^2), \\
	\bbE\big( \bbE[ \| \psi(z_i, z_j) \|^2 | z_j ] \|^2 \big) = o(n^2), \quad \bbE[ \| \psi(z_i, z_j) \|^4 ] = o(n^3) \\
	\bbE\big[ \bbE[\psi(z_i,z_j) \psi(z_i, z_l) | z_j, z_l]^2 \big] = o(n), \quad \bbE[ \psi(z_i, z_j)^2 \psi(z_i, z_l)^2 ] = o(n^2).
\end{gather*}
then we have
\begin{align*}
	\calB^{\texttt{NL}} = O(n^{-2r}) \quad \text{and} \quad \| \Bnl - \calB^{\texttt{NL}} \| = \oP(n^{-1/2}).
\end{align*}

As for the second one, note that $\Banb$ is the average of a sequence of \text{i.i.d.} random variables with negligible variance:
\begin{align*}
	\Banb \coloneqq \whatG_{n,\gamma}^{\,\prime}(\theta_0,\gamma_0,\bargamma_n - \gamma_0) = \frac{1}{n} \sum_{i=1}^n g_{\gamma}^{\prime}(z_i, \theta_0, \gamma_0, \bargamma_n(z_i) - \gamma_0(z_i) ).
\end{align*}
Let $\calB^{\texttt{ANB}} = \bbE[ g_{\gamma}^{\prime}(z_i, \theta_0, \gamma_0, \bargamma_n(z_i) - \gamma_0(z_i) ) ]$. We then have
\begin{gather*}
	\| \calB^{\texttt{ANB}} \| \leq \bbE[ \| g_{\gamma}^{\prime}(z_i, \theta_0, \gamma_0, \bargamma_n(z_i) - \gamma_0(z_i) ) \| ] \leq C \, \bbE[ \|  \bargamma_n(z_i) - \gamma_0(z_i)  \| ], \\
	\Var( \Banb ) = \frac{1}{n} \Var\big( g_{\gamma}^{\prime}(z_i, \theta_0, \gamma_0, \bargamma_n(z_i) - \gamma_0(z_i) ) \big) \leq \frac{C}{n} \Var\big( \bargamma_n(z_i) - \gamma_0(z_i) \big).
\end{gather*}
Hence, the desired result readily follows. This completes the proof.

\end{proof}

Suppose $\hatgamma_n$ is the \textquotedblleft leave-one-out\textquotedblright{} kernel density estimator. We have $\psi(z_i, z_j) = K_h(z_i - z_j)$. It then follows that (refer to the appendix for detailed calculation)
\begin{gather*}
	\frac{1}{n-1} \bbE[ \psi(z_i, z_j)^2 ] = \frac{1}{(n-1)h^{d_z}} \int K^2(u) \gamma_0(x) \gamma_0(x-hu) du dx = O\Big(\frac{1}{nh^{d_z}} \Big).
\end{gather*}
Similarly, we can obtain $\bbE[ \psi(z_i, z_j)^4 ] = O\big( h^{-3d_z} \big)$.

Moreover, one can show that
\begin{align*}
	\bbE[\psi(z_i, z_j)^2 | z_i ] = \frac{1}{h^{d_z}} \int_{\bbR} K^2(u) \gamma_0(z_i-hu) du. 
\end{align*}
It then follows that
\begin{align*}
	& \bbE\big( \bbE[\psi(z_i, z_j)^2 | z_i ]^2 \big) \\
	=\,& \frac{1}{h^{2d_z}}  \int_{\bbR}  \int_{\bbR} K^2(u) K^2(v) \gamma_0(x-hu) \gamma_0(x-hv) \gamma_0(x) dudvdx = O\Big(\frac{1}{h^{2d_z}} \Big).
\end{align*}
In addition, we can derive that
\begin{align*}
	& \bbE\big[ \bbE[\psi(z_i,z_j) \psi(z_i, z_l) | z_j, z_l]^2 \big] \\
	=\,& \bbE\Big(  \int_{\bbR} \frac{1}{h^{2d_z}} K\big(\frac{x-z_j}{h}\big) K\big(\frac{x-z_l}{h}\big) \gamma_0(x) dx \Big)^2 \\
	=\,& \frac{1}{h^{4d_z}}  \int_{\bbR} \int_{\bbR} \int_{\bbR} \int_{\bbR} K\big(\frac{x-x_j}{h}\big) K\big(\frac{y-x_j}{h}\big) K\big(\frac{x-x_l}{h}\big) K\big(\frac{y-x_l}{h}\big) \\
		& \times \gamma_0(x) \gamma_0(y) \gamma_0(x_j) \gamma_0(x_l) dx dy dx_j dx_l \\
	=\,& \frac{1}{h^{d_z}}  \int_{\bbR} \int_{\bbR} \int_{\bbR} \int_{\bbR} K(u) K(v) K(w) (u-v+w) \gamma_0(x) \gamma_0(x-hu) \gamma_0(x-hv) \\
		& \times \gamma_0(x-hu+hw) dudvdw dx = O\Big( \frac{1}{h^{d_Z}} \Big).
\end{align*}
Similarly, we get
\begin{align*}
	& \bbE\big[ \| \psi(z_i,z_j) \psi(z_i, z_l) \|^2 \big] \\
		=\,& \int_{\bbR} \int_{\bbR} \int_{\bbR} \frac{1}{h^{4d_z}} K^2\big(\frac{x - x_j}{h}\big) K^2\big(\frac{x-x_l}{h}\big) \gamma_0(x) \gamma_0(x_j) \gamma_0(x_l) dx dx_j dx_l \\
		=\,& \frac{1}{h^{2d_z}} \int_{\bbR} \int_{\bbR} \int_{\bbR} K^2(u) K^2(v) \gamma_0(x) \gamma_0(x-hu) \gamma_0(x-hv) dx du dv = O\Big( \frac{1}{h^{2d_z}} \Big).
\end{align*}

Hence, the corresponding conditions given in \eqref{eq:BOcond2} only require that
\begin{align*}
	n h^{d_z} = n^{2r} \rightarrow \infty.
\end{align*}
Note that the convergence rate of $\hatgamma_n - \bargamma_n$ is given by $\sqrt{nh^{dz}}$. Hence, the above condition merely requires that $\hatgamma_n-\bargamma_n$ converges to zero. 

On the other hand, we have
\begin{align*}
	\bargamma_n(z_i) - \gamma_0(z_i) = \int K(u) [\gamma_0(z_i - hu) - \gamma_0(z_i)] du.
\end{align*}
It the easy to see that the second last condition given in \eqref{eq:BOcond2} is satisfied with $h^m = n^{-s}$, where $m$ is the order of the kernel $K$. The last condition in \eqref{eq:BOcond2} only requires that the nonparametric bias $\bargamma_n - \gamma_0$ is asymptotically negligible. 

To briefly sum up, in the kernel density case, the conditions in \eqref{eq:BOcond2} essentially requires $\hatgamma_n$ to be a consistent nonparametric estimator of $\gamma_0$.

\subsection{Proof of Lemma \ref{lem:op1} and Theorem \ref{thm:infeasibleCLT}}

Recall that
\begin{align*}
\begin{split} 
	\whatG_n(\theta_0,\hatgamma_n) =\,&  \whatG_n(\theta_0,\gamma_0) + \whatG_{n,\gamma}^{\,\prime}(\theta_0,\gamma_0,\hatgamma_n - \bargamma_n) + \Banb + \Bnl \\
		& + \whatG_{n,\gamma\gamma}^{\,\prime\prime}(\theta_0,\gamma_0,\hatgamma_n - \bargamma_n, \bargamma_n - \gamma_0) + \frac{1}{2} \whatG_{n,\gamma\gamma}^{\,\prime\prime}(\theta_0,\gamma_0,\bargamma_n - \gamma_0, \bargamma_n - \gamma_0) \\
		&+ \whatG_{n,R}(\theta_0,\gamma_0,\hatgamma_n - \gamma_0).
\end{split}
\end{align*}

Following the argument in Section \ref{subsec:VU}, the term $\whatG_{n,\gamma\gamma}^{\,\prime\prime}(\theta_0,\gamma_0,\hatgamma_n - \bargamma_n, 1)$ is a U-statistic, hence is of order $\OP(n^{-1/2})$. Therefore, $\whatG_{n,\gamma\gamma}^{\,\prime\prime}(\theta_0,\gamma_0,\hatgamma_n - \bargamma_n, \bargamma_n - \gamma_0) = \oP(n^{-1/2})$. Moreover, it is also easy to show that $ \frac{1}{2} \whatG_{n,\gamma\gamma}^{\,\prime\prime}(\theta_0,\gamma_0,\bargamma_n - \gamma_0, \bargamma_n - \gamma_0) = \OP( n^{-2 \underline{s}} )$. As for the last term, we have
\begin{align*}
	\bbE[ \| \whatG_{n,R}(\theta_0,\gamma_0,\hatgamma_n - \gamma_0) \| ] \leq \frac{1}{n} \sum_{i=1}^n C \, \bbE[ \| \hatgamma_n(z_i) - \gamma_0(z_i) \|^3 ] \leq C n^{-3(r \wedge \underline{s})}.
\end{align*}
It then follows that
\begin{align*}
	\whatG_n(\theta_0,\hatgamma_n) =\,& \OP(n^{-1/2}) + \Banb + \Bnl + \OP( n^{-2 \underline{s}} ) + \OP(n^{-3(r \wedge \underline{s})}).
\end{align*}

Cauchy-Schwartz inequality implies that
\begin{align*}
	\bbE\big( \| ( \calJ_n - \calJ_0 ) \Bnl  \| \big) \leq\,& C \Big( \bbE\big( \| \whatG_n(\theta_0,\hatgamma_n) \|^2 \big) \, \bbE\big( \| \Bnl \|^2 \big) \Big)^{1/2} \\
		\leq\,& C \Big( ( n^{-2\underline{s}} + n^{-4r} ) n^{-4r} \Big)^{1/2} \leq C (n^{-(\underline{s} + 2r)} + n^{-4r} ).
\end{align*}
The right hand side is $o(n^{-1/2})$ if $\underline{s} + 2r >1/2$ and $r>1/8$. This proves part (i).

As for part (ii), similar argument yields that
\begin{align*}
	\bbE\big( \| ( \calJ_n - \calJ_0 ) \Banb  \| \big) \leq\,& C \Big( \bbE\big( \| \whatG_n(\theta_0,\hatgamma_n) \|^2 \big) \, \bbE\big( \| \Banb \|^2 \big) \Big)^{1/2} \\
		\leq\,& C \Big( ( n^{-2\underline{s}} + n^{-4r} ) n^{-2\underline{s}} \Big)^{1/2} \leq C (n^{-2\underline{s}} + n^{-(2r + \underline{s})} ).
\end{align*}
The right hand side is $o(n^{-1/2})$ if $\underline{s} + 2r >1/2$ and $\underline{s}>1/4$.

The above discussion indicates that
\begin{align*}
	& \whatG_n(\theta_0,\hatgamma_n) -  \whatG_n(\theta_0,\gamma_0) - \whatG_{n,\gamma}^{\,\prime}(\theta_0,\gamma_0,\hatgamma_n - \bargamma_n) - \Banb - \Bnl \\	
	=\,& \oP(n^{-1/2}) + \OP( n^{-2 \underline{s}} ) + \OP(n^{-3(r \wedge \underline{s})}).
\end{align*}
If $\underline{s} > 1/4$ and $r>1/6$, then it is easy to see that the right hand side is root-$n$ negligible. This completes the proof of Lemma \ref{lem:op1}.

Lastly, note that when $\underline{s} > 1/4$ and $r>1/6$, all the three conclusions of Lemma \ref{lem:op1} are true. Then, Theorem \ref{thm:infeasibleCLT} readily follows from Lemma \ref{lem:op1} and Assumption \ref{asmp:AN}.

\subsection{Proof of Theorems \ref{thm:jackknifeCLT} and \ref{thm:jackknifeVar} }

We first prove Theorems \ref{thm:jackknifeCLT}. Given the assumptions, it is easy to derive that
\begin{align*}
	&\quad \hattheta_n^w - \theta_0 = \sum_{q=1}^Q w_q \big( \hattheta_n(h_q) - \theta_0 \big) = \calJ_0 \sum_{q=1}^Q w_q \whatG_n\big (\theta_0,\hatgamma_n(h_q) \big) + \oP(n^{-1/2}) \\
		&= \calJ_0 \Big( \whatG_n(\theta_0,\gamma_0) + \whatG_{n,\gamma}^{\,\prime}\big(\theta_0,\gamma_0, \hatgamma_n^w - \bargamma_n^w \big) +  \widetilde{\calB}_{n,1}^{\texttt{ANB}} + \widetilde{\calB}_{n,2}^{\texttt{ANB}} + \widetilde{\calB}_{n}^{\texttt{NL}} \Big) + \oP(n^{-1/2}),
\end{align*}
where
\begin{gather*}
	\widetilde{\calB}_{n,1}^{\texttt{ANB}} = \sum_{q=1}^Q w_q \BanbOne(h_q), \quad  
	\widetilde{\calB}_{n,2}^{\texttt{ANB}} = \sum_{q=1}^Q w_q \BanbTwo(h_q), \quad 
	\widetilde{\calB}_{n}^{\texttt{NL}} = \sum_{q=1}^Q w_q \Bnl(h_q).
\end{gather*}
We are going to show these three terms are root-$n$ negligible.

First of all, it is quite straightforward to show that
\begin{align*}
	\widetilde{\calB}_{n,2}^{\texttt{ANB}} &= \sum_{q=1}^Q w_q \BanbTwo(h_q) = \sum_{q=1}^Q w_q \frac{1}{n} \sum_{i=1}^n g_\gamma^{\,\prime}\big(z_i,\theta_0,\gamma_0, \frac{1}{nh_q^{d_z}} K(0) \big) \\
		&= \Big( \sum_{q=1}^Q w_q \frac{1}{nh_q^{d_z}} \Big) \times \Big( \frac{1}{n} \sum_{i=1}^n g_\gamma^{\,\prime}\big(z_i,\theta_0,\gamma_0, K(0) \big) \Big)= \OP\Big( \sum_{q=1}^Q w_q \frac{1}{nh_q^{d_z}} \Big) \\
		&= \oP(n^{-1/2}).
\end{align*}
Now, for notation simplicity, assume that $\bargamma_n(z_i) - \gamma_0(z_i)$ does not involve the \textquotedblleft singularity bias,\textquotedblright{} as we have removed it. We readily get
\begin{align*}
	\widetilde{\calB}_{n,1}^{\texttt{ANB}} &= \sum_{q=1}^Q w_q \BanbOne(h_q) = \sum_{q=1}^Q w_q \frac{1}{n} \sum_{i=1}^n g_\gamma^{\,\prime}\big(z_i,\theta_0,\gamma_0, \bargamma_n(z_i;h_q) - \gamma_0(z_i) \big) \\
		&= \sum_{q=1}^Q w_q \frac{1}{n} \sum_{i=1}^n g_\gamma^{\,\prime}\Big(z_i,\theta_0,\gamma_0, \int K(u) [\gamma_0(z_i - h_q u) - \gamma_0(z_i)] du \Big) \\
		&= \Big( \sum_{q=1}^Q w_q h_q^m \Big) \times \frac{1}{m} \sum_{i=1}^n g_\gamma^{\,\prime}\Big(z_i,\theta_0,\gamma_0, \frac{1}{m!} \vect\big( \gamma_0^{(m)}(z_i) \big)^\intercal \int K(u) u^{\otimes m} du \Big) \\
		&= \OP\Big( \sum_{q=1}^Q w_q h_q^m \Big) = \oP(n^{-1/2}).
\end{align*}
It can also be shown that
\begin{align*}
	\widetilde{\calB}_{n}^{\texttt{NL}} &= \sum_{q=1}^Q w_q \Bnl(h_q) = \sum_{q=1}^Q w_q \frac{1}{n} \sum_{i=1}^n g_{\gamma\gamma}^{\,\prime\prime}\big(z_i,\theta_0,\gamma_0, \hatgamma_n(h_q) - \bargamma_n(h_q), \hatgamma_n(h_q) - \bargamma_n(h_q) \big) \\
		&=  \sum_{q=1}^Q w_q \frac{1}{nh_q^{d_z}}  \frac{1}{n}  \sum_{i=1}^n g_{\gamma\gamma}^{\,\prime\prime}\big(z_i,\theta_0,\gamma_0, \sqrt{nh^{d_z}}[\hatgamma_n(h_q) - \bargamma_n(h_q)], \sqrt{nh^{d_z}}[\hatgamma_n(h_q) - \bargamma_n(h_q)] \big).
\end{align*}
Here, we note that $\{ \sqrt{nh^{d_z}}[\hatgamma_n(h_q) - \bargamma_n(h_q)] \}_{q=1}^Q$ are standardized so that the asymptotic distributions are the same for different $q$. Let 
\begin{align*}
	\xi_n^q = \frac{1}{n}  \sum_{i=1}^n g_{\gamma\gamma}^{\,\prime\prime}\big(z_i,\theta_0,\gamma_0, \sqrt{nh^{d_z}}[\hatgamma_n(h_q) - \bargamma_n(h_q)], \sqrt{nh^{d_z}}[\hatgamma_n(h_q) - \bargamma_n(h_q)] \big).
\end{align*}
Then, intuitively, $\{\xi_n^q \}_{q=1}^Q$ are asymptotically identical to each other. Hence, we have $ \bbE[\xi_n^q  \xi_n^{q'} ] - \bbE[(\barxi_n)^2] \conP 0$, where $\barxi_n = \frac{1}{Q} \sum_{q=1}^Q \xi_n^q$. As a consequent, we obtain
\begin{align*}
	n\, \bbE[ (\widetilde{\calB}_{n}^{\texttt{NL}})^2 ] = n \sum_{q,q'=1}^Q \Big( \frac{w_q}{n h_q^{d_z}} \frac{w_{q'}}{n h_{q'}^{d_z}} 
	\bbE[\xi_n^q \xi_n^{q'} ] \Big) \conP 0.
\end{align*}
Chebyshev's inequality implies that $\widetilde{\calB}_{n}^{\texttt{NL}} = \oP(n^{-1/2})$. Therefore, the conclusion of Theorem \ref{thm:jackknifeCLT} readily follows.

Now let us turn to Theorem \ref{thm:jackknifeVar}. When $g^\ast \equiv g$, $g_{\gamma}^{\,\ast\prime} \equiv g_{\gamma}^{\,\prime}$, and $g_{\gamma\gamma}^{\,\ast\prime\prime} \equiv g_{\gamma\gamma}^{\,\prime\prime}$, one can follow the above argument and show that
\begin{gather*}
	\widetilde{\calB}_{n,1}^{\texttt{ANB} \ast} = \oPstar(n^{-1/2}), \quad  
	\widetilde{\calB}_{n,2}^{\texttt{ANB} \ast}  = \oPstar(n^{-1/2}), \quad 
	\widetilde{\calB}_{n}^{\texttt{NL} \ast}  = \oPstar(n^{-1/2}).
\end{gather*}
It then follows that
\begin{align*}
	& \hattheta_n^{w \ast} - \hattheta_n^{w} = \sum_{q=1}^Q w_q ( \hattheta_n^\ast(h_q) - \hattheta_n(h_q ) = \sum_{q=1}^Q w_q \calJ_n^\ast \whatG_n\big(\hattheta_n, \hatgamma_n^\ast(h_q) \big) \\
	=\,& \sum_{q=1}^Q w_q \calJ_0^\ast \whatG_n\big(\hattheta_n, \hatgamma_n^\ast(h_q) \big) + \oPstar(n^{-1/2}) \\
	=\,& \calJ_0^\ast \sum_{q=1}^Q w_q \Big( \whatG_n^\ast\big(\hattheta_n, \hatgamma_n(h_q) \big) + \whatG_{n,\gamma}^{\,\ast\prime}(\hattheta_n, \hatgamma_n(h_q), \hatgamma_n^{\ast}(h_q) - \bargamma_n^{\ast}(h_q) \Big) + \oPstar(n^{-1/2}),
\end{align*}
where
\begin{gather*}
	\whatG_n^\ast\big(\hattheta_n, \hatgamma_n(h_q) \big) = \frac{1}{n} \sum_{i=1}^n g\big(z_i^\ast, \hattheta_n, \hatgamma_h(h_q) \big), 
\end{gather*}
and
\begin{align*}
	& \whatG_{n,\gamma}^{\prime\ast} \big(\hattheta_n, \hatgamma_n(h_q), \hatgamma_n^\ast(h_q) - \bargamma_n^\ast(h_q) \big) = \frac{1}{n} \sum_{i=1}^n g_{\gamma}^{\,\prime}\big(z_i^\ast, \hattheta_n, \hatgamma_n(h_q), \hatgamma_n^\ast(h_q) - \bargamma_n^\ast(h_q) \big) \\
		=\,& \frac{1}{n(n-1)} \sum_{\substack{i,j=1 \\ i\neq j}}^n g_{\gamma}^{\,\prime}\big(z_i^\ast, \hattheta_n, \hatgamma_n(h_q), \phi(z_i^\ast, z_j^\ast; h_q) \big) + \oPstar(n^{-1}).
\end{align*}

Note that the functional forms are the same as in the original case. The only difference is that now we use the bootstrap sample $\{z_i^\ast\}_{i=1}^n$, rather than the original sample $\{z_i\}_{i=1}^n$. In view of these, under the bootstrap measure $\bbP_n^\ast$, the asymptotic variance 
\begin{align*}
	\Sigma_g^{w\ast} \coloneqq \Var^\ast\bigg( \sqrt{n} \Big( \sum_{q=1}^Q w_q \whatG_n^\ast\big(\hattheta_n, \hatgamma_n(h_q) \big) + \whatG_{n,\gamma}^{\,\ast\prime}(\hattheta_n, \hatgamma_n, \hatgamma_n^{w\ast} - \bargamma_n^{w\ast} \big)  \Big) \bigg)
\end{align*}
should converge in probability to the sample variance $\Sigma_g^w(\hattheta_n, \hatgamma_n)$ of 
\begin{align*}
	\Big\{ \sum_{q=1}^Q w_q g\big(z_i, \hattheta_n, \hatgamma_n(h_q) \big) + g_{\gamma}^{\,\prime}(z_i, \hattheta_n, \hatgamma_n, \hatgamma_n^w - \bargamma_n^w ) \Big\}_{i=1}^n.
\end{align*}
As $n$ goes to infinity, the Lipschitz continuous assumption on $g$ and $g_{\gamma}^{\,\prime}$ implies that $\Sigma_g^w(\hattheta_n, \hatgamma_n) \conP \Sigma_g^w(\theta_0, \gamma_0) = \Sigma_g^w$. Together with $\calJ_0^\ast \conP \calJ_0$, we readily get $\Sigma_{\theta}^{w\ast} \conP \Sigma_{\theta}^{w}$.



\subsection{Proof of Theorem \ref{thm:analytical} }

According to the assumption, we readily get the following decomposition
\begin{align*}
	& \hattheta_n - \theta_0 - \calJ_n \whatBnl - \calJ_n\whatBanb \\
	=\,& \big( \hattheta_n - \theta_0 - \calJ_n [\Bnl + \Banb + \whatG_n(\theta_0,\gamma_0) + \whatG_{n,\gamma}^{\,\prime}(\theta_0,\gamma_0,\hatgamma_n - \bargamma_n) ] \big) \\
		& + \calJ_n \big( \whatG_n(\theta_0,\gamma_0) + \whatG_{n,\gamma}^{\,\prime}(\theta_0,\gamma_0,\hatgamma_n - \bargamma_n) + \Banb - \whatBanb \big) + \calJ_n ( \Bnl - \whatBnl ) \\
	=\,& \calJ_n \big( \whatG_n(\theta_0,\gamma_0) + \whatG_{n,\gamma}^{\,\prime}(\theta_0,\gamma_0,\hatgamma_n - \bargamma_n) + \Banb - \whatBanb \big) + \calJ_n ( \Bnl - \whatBnl ) + \oP(n^{-1/2}).
\end{align*}
Note that 
\begin{align*}
	& \whatG_{n,\gamma}^{\,\prime}(\theta_0,\gamma_0,\hatgamma_n - \bargamma_n) + \Banb - \whatBanb \\
	=\,& \whatG_{n,\gamma}^{\,\prime}(\theta_0,\gamma_0,\hatgamma_n - \gamma_0) - \whatG_{n,\gamma}^{\,\prime}(\hattheta_n, \hatgamma_n, \hat{\bargamma}_n - \hatgamma_n)  \\
	=\,& \whatG_{n,\gamma}^{\,\prime}(\theta_0,\gamma_0,\hatgamma_n - \gamma_0) - \whatG_{n,\gamma}^{\,\prime}(\theta_0, \gamma_0, \hat{\bargamma}_n - \hatgamma_n) + \whatG_{n,\gamma}^{\,\prime}(\theta_0, \gamma_0, \hat{\bargamma}_n - \hatgamma_n) \\
		& - \whatG_{n,\gamma}^{\,\prime}(\hattheta_n, \hatgamma_n, \hat{\bargamma}_n - \hatgamma_n) \\
	=\,& \whatG_{n,\gamma}^{\,\prime}(\theta_0,\gamma_0, 2\hatgamma_n - \hat{\bargamma}_n - \gamma_0 ) + [ \whatG_{n,\gamma}^{\,\prime}(\theta_0, \gamma_0, \hat{\bargamma}_n - \hatgamma_n) - \whatG_{n,\gamma}^{\,\prime}(\hattheta_n, \hatgamma_n, \hat{\bargamma}_n - \hatgamma_n) ] \\
	=\,& \whatG_{n,\gamma}^{\,\prime}(\theta_0,\gamma_0, 2\hatgamma_n - \hat{\bargamma}_n - 2\bargamma_n + \bar{\bargamma}_n ) + \whatG_{n,\gamma}^{\,\prime}(\theta_0,\gamma_0, 2\bargamma_n - \bar{\bargamma}_n - \gamma_0) \\
		& + [ \whatG_{n,\gamma}^{\,\prime}(\theta_0, \gamma_0, \hat{\bargamma}_n - \hatgamma_n) - \whatG_{n,\gamma}^{\,\prime}(\hattheta_n, \hatgamma_n, \hat{\bargamma}_n - \hatgamma_n) ].
\end{align*}
According to Conditions \eqref{eq:analytical-AN} and \eqref{eq:analytical-op1}, it is sufficient to show the following
\begin{gather*}
	\whatG_{n,\gamma}^{\,\prime}(\theta_0, \gamma_0, \hat{\bargamma}_n - \hatgamma_n) - \whatG_{n,\gamma}^{\,\prime}(\hattheta_n, \hatgamma_n, \hat{\bargamma}_n - \hatgamma_n) = \oP(n^{-1/2}), \quad \Bnl - \whatBnl = \oP(n^{-1/2}).
\end{gather*}

First, triangle inequality and Cauchy-Schwartz inequality imply that
\begin{align*}
	& \bbE\big( \big\| \whatG_{n,\gamma}^{\,\prime}(\theta_0, \gamma_0, \hat{\bargamma}_n - \hatgamma_n) - \whatG_{n,\gamma}^{\,\prime}(\hattheta_n, \hatgamma_n, \hat{\bargamma}_n - \hatgamma_n) \big\| \big) \\
	=\,& \bbE\Big( \big\| \frac{1}{n} \sum_{i=1}^n \big( \bbD_{\gamma} g(z_i,\theta_0,\gamma_0) - \bbD_{\gamma} g(z_i,\hattheta_n,\hatgamma_n) \big) \, \vect\big( \hat{\bargamma}_n(z_i) - \hatgamma_n(z_i) \big) \big\| \Big) \\
	\leq\,& \frac{1}{n} \sum_{i=1}^n \bbE\Big( \big\| \big( \bbD_{\gamma} g(z_i,\theta_0,\gamma_0) - \bbD_{\gamma} g(z_i,\hattheta_n,\hatgamma_n) \big) \, \vect\big( \hat{\bargamma}_n(z_i) - \hatgamma_n(z_i) \big) \big\| \Big) \\
	=\,&  \bbE\Big( \big\| \big( \bbD_{\gamma} g(z,\theta_0,\gamma_0) - \bbD_{\gamma} g(z,\hattheta_n,\hatgamma_n) \big) \, \vect\big( \hat{\bargamma}_n(z) - \hatgamma_n(z) \big) \big\| \Big) \\
	\leq\,& \Big( \bbE\big( \| \bbD_{\gamma} g(z,\theta_0,\gamma_0) - \bbD_{\gamma} g(z,\hattheta_n,\hatgamma_n) \|^2 \big) \, \bbE\big( \| \hat{\bargamma}_n(z) - \hatgamma_n(z) \|^2 \big) \Big)^{1/2} \\
	\leq\,& C n^{- (r \wedge s)} \times n^{-t} = C n^{- ( r \wedge s + t ) } = o(n^{-1/2}).
\end{align*}


Second, note that the term $\Bnl - \whatBnl $ writes as
\begin{align*}
	\frac{1}{n^{1+2r}} \sum_{i=1}^n \big\{ \big[ n^{2r} \vect\big( \hatgamma_n(z_i) - \bargamma_n (z_i)  \big)^{\otimes 2} - \vect\big( \hatV_n(z_i) \big)  \big] \otimes I_{d_g} \big\}^\intercal \, \vect\big( \bbD_{\gamma\gamma}^2 g(z_i,\theta_0,\gamma_0) \big)
\end{align*}
A similar argument yields that
\begin{align*}
	& \bbE[ \| \Bnl - \whatBnl \| ] \\	
	\leq\,& \frac{1}{n^{2r}} \bbE \Big( \big\| \big[ n^{2r} \vect\big( \hatgamma_n(z) - \bargamma_n (z)  \big)^{\otimes 2} - \vect\big( \hatV_n(z) \big) \big] \otimes I_{d_g} \big\}^\intercal \, \vect\big( \bbD_{\gamma\gamma}^2 g(z,\theta_0,\gamma_0) \big)  \big\| \Big) \\
	\leq\,& \frac{C}{n^{2r}} \Big( \bbE\big( \big\| n^{2r} \vect\big( \hatgamma_n(z) - \bargamma_n (z)  \big)^{\otimes 2} - \vect\big( \hatV_n(z) \big) \big\|^2 \big) \, \bbE\big( \big\| \bbD_{\gamma\gamma}^2 g(z,\theta_0,\gamma_0) \big\|^2 \big) \Big)^{1/2} \\
	\leq\,& C n^{-2r - v} = o(n^{-1/2}).
\end{align*}
This completes the proof.

\subsection{Proof of Theorem \ref{thm:discontinuous} }

Recall that, in this case, the functional $g$ is a smoothed projection of $\check{g}$ on some sub-$\sigma$-algebra of the $\sigma$-algebra generated by the sample. 

(i) Let $\hattheta_n$ be the corresponding estimator defined by $g$. Under Assumption \ref{asmp:AL'} and those conditions of Lemma \ref{lem:op1}, we obtain
\begin{align*}
	& \check{\theta}_n - \theta_0 = (\check{\theta}_n - \hattheta_n) + (\hattheta_n - \theta_0) \\
	=\,& \calJ_n \Big( \widecheck{G}_n(\theta_0, \hatgamma_n) - \whatG_n(\theta_0, \hatgamma_n) + \whatG_n(\theta_0, \gamma_0) + \whatG_{n,\gamma}^{\,\prime}(\theta_0,\gamma_0, \hatgamma_n - \bargamma_n) \\
		& + \Banb + \Bnl \Big) + \oP(n^{-1/2}).
\end{align*}
Assumption \ref{asmp:AN'}, $\calJ_n - \calJ_0 = \OP\big( \whatG_n(\theta_0, \hatgamma_n) \big)$=$\oP(1)$ (recall that $\whatG_n(\theta_0, \hatgamma_n) \conP G(\theta_0, \gamma_0)=0$), and those conditions on $\underline{s}$ and $r$ further imply that
\begin{align*}
	& \sqrt{n} \big( \check{\theta}_n - \theta_0 - \calJ_n \Bnl - \calJ_n \Banb \big) \\
	=\,& \sqrt{n} \calJ_n \Big( \widecheck{G}_n(\theta_0, \hatgamma_n) - \whatG_n(\theta_0, \hatgamma_n) + \whatG_n(\theta_0, \gamma_0) + \whatG_{n,\gamma}^{\,\prime}(\theta_0,\gamma_0, \hatgamma_n - \bargamma_n)  \Big) + \oP(1) \\
	=\,& \sqrt{n} \calJ_0 \Big( \widecheck{G}_n(\theta_0, \hatgamma_n) - \whatG_n(\theta_0, \hatgamma_n) + \whatG_n(\theta_0, \gamma_0) + \whatG_{n,\gamma}^{\,\prime}(\theta_0,\gamma_0, \hatgamma_n - \bargamma_n)  \Big) + \oP(1) \\
	\conL \,& \calN(0, \calJ_0 \, \widecheck{\Sigma}_g \, \calJ_0^\intercal).
\end{align*}

(ii) As shown in the proof of part (i), the additional term $\check{\theta}_n - \hattheta_n$ does not make an essential difference with the continuous functional case under Assumption \ref{asmp:AN'}. In view of this, the proof for the multi-scale jackknife estimator is quite similar to the proof of Theorem \ref{thm:jackknifeCLT}. Hence, we omit it here to save space.

\bigskip

\section{Additional Simulation Results} \label{sec:addsim}

\subsection{Average Density (AD) Estimator} 

Recall in the main-text that the nonlinear bias for the AD estimator is identically zero: $\Bnl\equiv 0$. The averaged nonparametric bias is given by:
\begin{align*}
    \Banb = \frac{1}{n} \sum_{i=1}^n [\bargamma_n(z_i) - \gamma_0(z_i)] = \frac{1}{n} \sum_{i=1}^n \int K(u) [\gamma_0(z_i - hu) - \gamma_0(z_i)] du = \OP( h^m ).
\end{align*}
It is estimated by
\begin{align*}
	\whatBanb = \frac{1}{n} \sum_{i=1}^n [ \hat{\bargamma}_n(z_i) - \hatgamma_n(z_i)] = \frac{1}{n} \sum_{i=1}^n \frac{1}{n-1} \sum_{j\neq i} \tilK_h(z_j - z_i),
\end{align*}
where $\tilK_h(u) = \frac{1}{h^{d_z}} \tilK(u/h)$ and $\tilK(u) = 2 K(u) - \int K(u-v) K(v) dv$ is the twicing kernel studied by \cite{Stuetzle&Mittal:1979} and \cite{NHR:2004}.

Figures \ref{fig:AD50MSE} and \ref{fig:AD50CI} present the decomposition of MSE and the empirical coverage rates of confidence intervals for the case $n=50$. We also plot in Figure \ref{fig:AD50Density} the densities of the $t$-statistics $\textstyle \sqrt{n} ( \hattheta_n - \theta_0 ) / \sqrt{\Var(\hattheta_n)}$ for different estimators (bias-corrected or not) at several selected bandwidth. As shown in the figure, the locations of those densities shift away from the standard normal density as bias becomes large (in magnitude).

\begin{figure}[!htbp]
\centering
\includegraphics[width=1\textwidth]{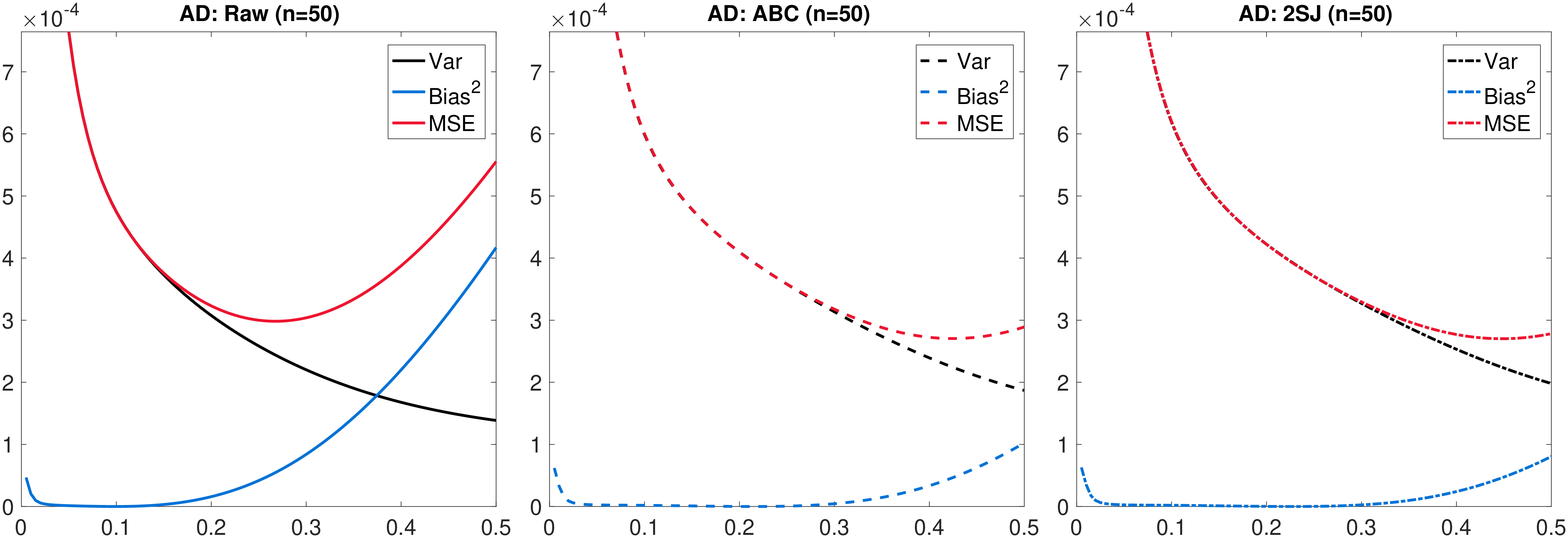}
\caption{AD: Decomposition of Mean Squared Error}
\label{fig:AD50MSE}
\end{figure}

\begin{figure}[!htbp]
\centering
\includegraphics[width=1\textwidth]{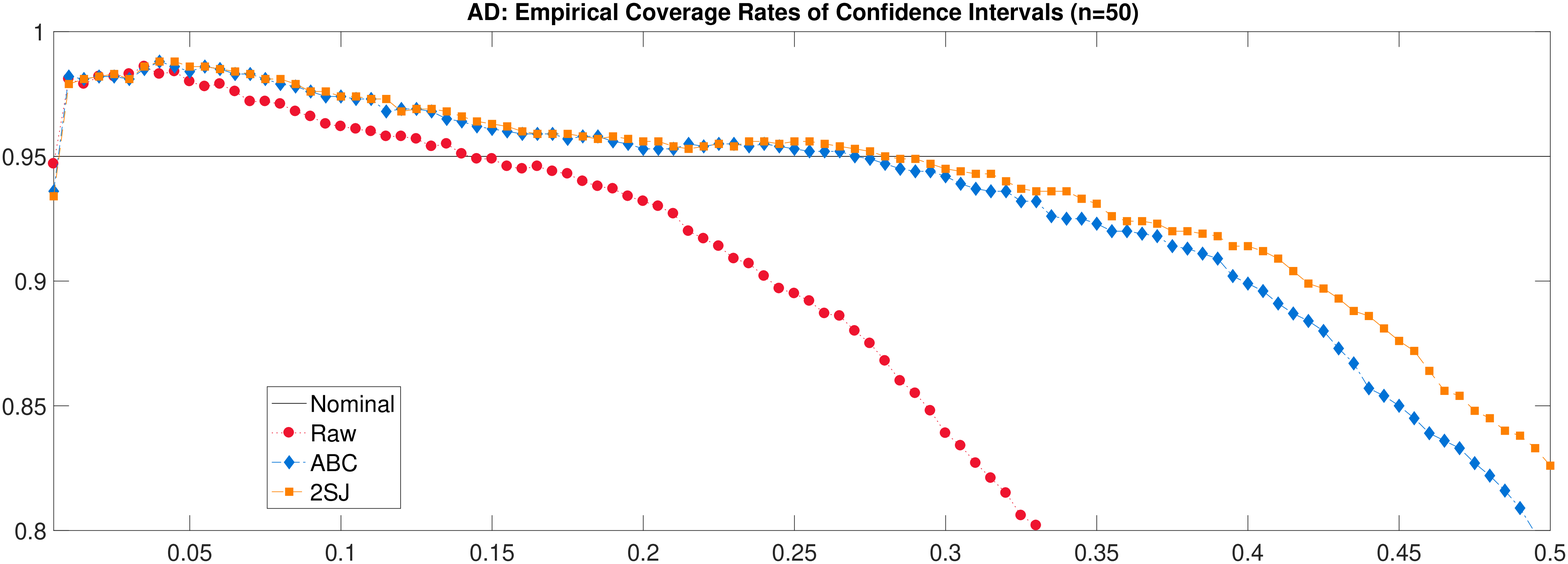}
\caption{AD: Empirical Coverage Rates of Confidence Intervals}
\label{fig:AD50CI}
\end{figure}

\begin{figure}[!htbp]
\centering
\includegraphics[width=1\textwidth]{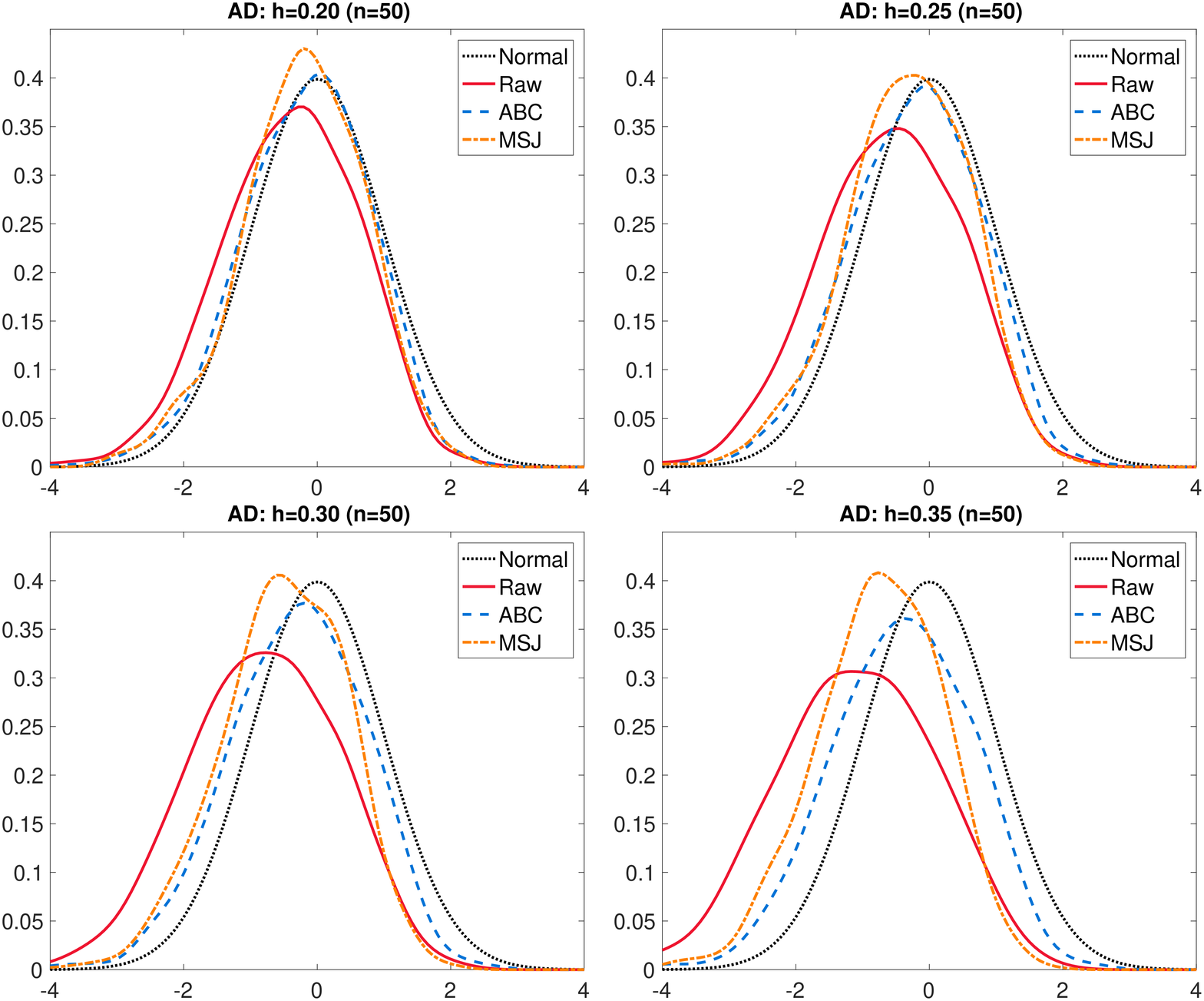}
\caption{AD: Densities of t-statistics and standard normal R.V.}
\label{fig:AD50Density}
\end{figure}

Figures \ref{fig:AD200MSE}, \ref{fig:AD200CI}, and \ref{fig:AD200Density} summarize the results for the case $n=200$. Although the general pattern remains the same, it is clear that now the range of bandwidths with correct coverage rage shrinks. This is mainly because the \textquotedblleft reasonable choice \textquotedblright{} of bandwidth decreases as the sample size increase, as indicated by the relation $h=n^{-\kappa}$ with $\kappa >0$.

\begin{figure}[!htbp]
\centering
\includegraphics[width=1\textwidth]{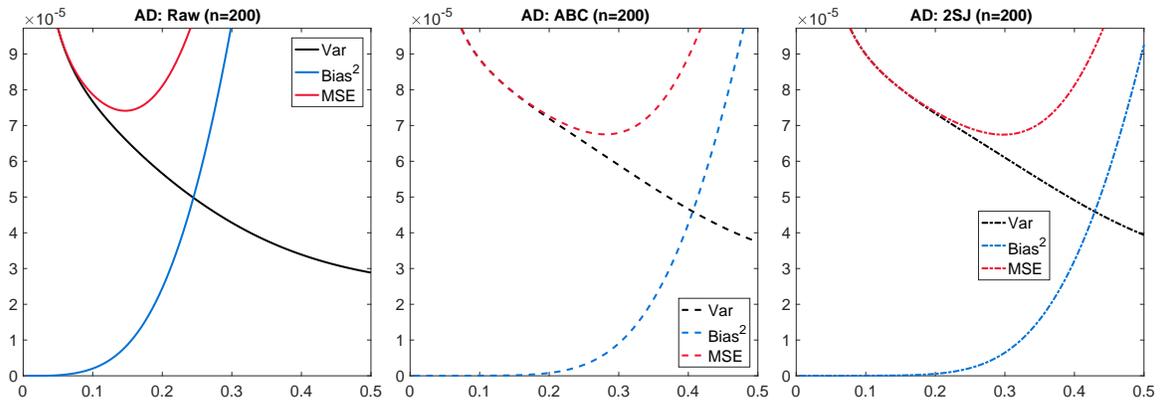}
\caption{AD: Decomposition of Mean Squared Error}
\label{fig:AD200MSE}
\end{figure}

\begin{figure}[!htbp]
\centering
\includegraphics[width=1\textwidth]{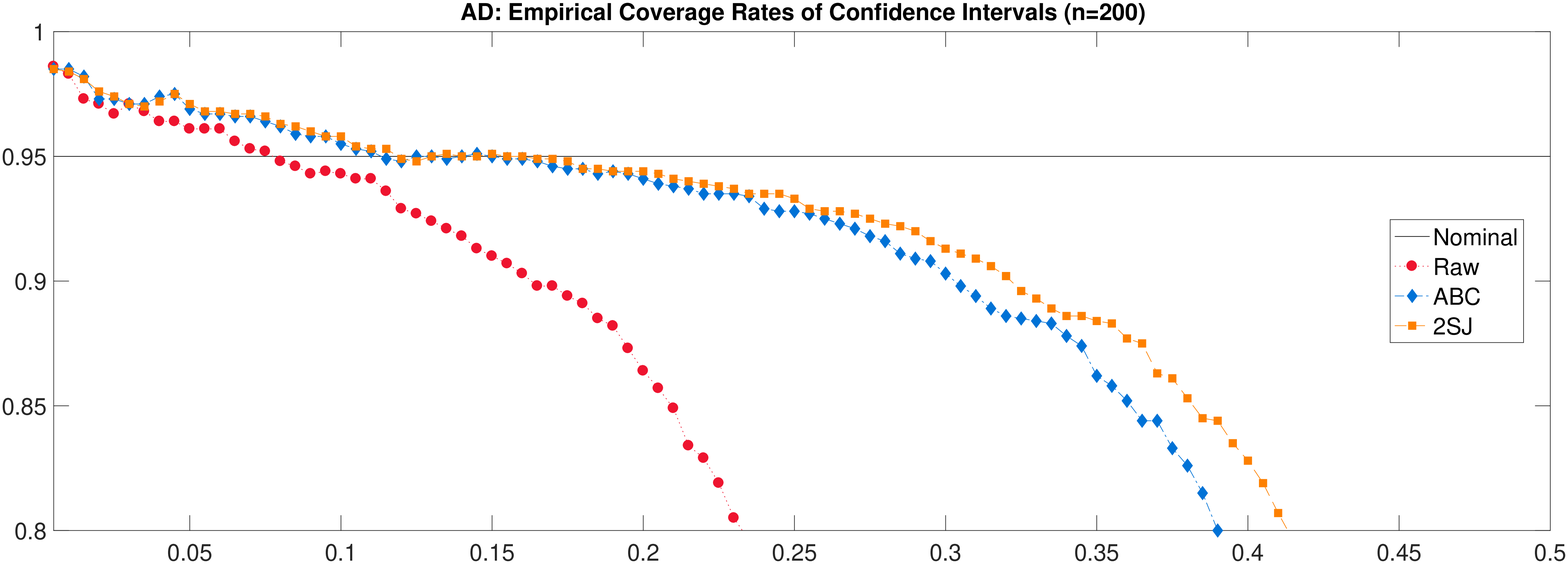}
\caption{AD: Empirical Coverage Rates of Confidence Intervals}
\label{fig:AD200CI}
\end{figure}

\begin{figure}[!htbp]
\centering
\includegraphics[width=1\textwidth]{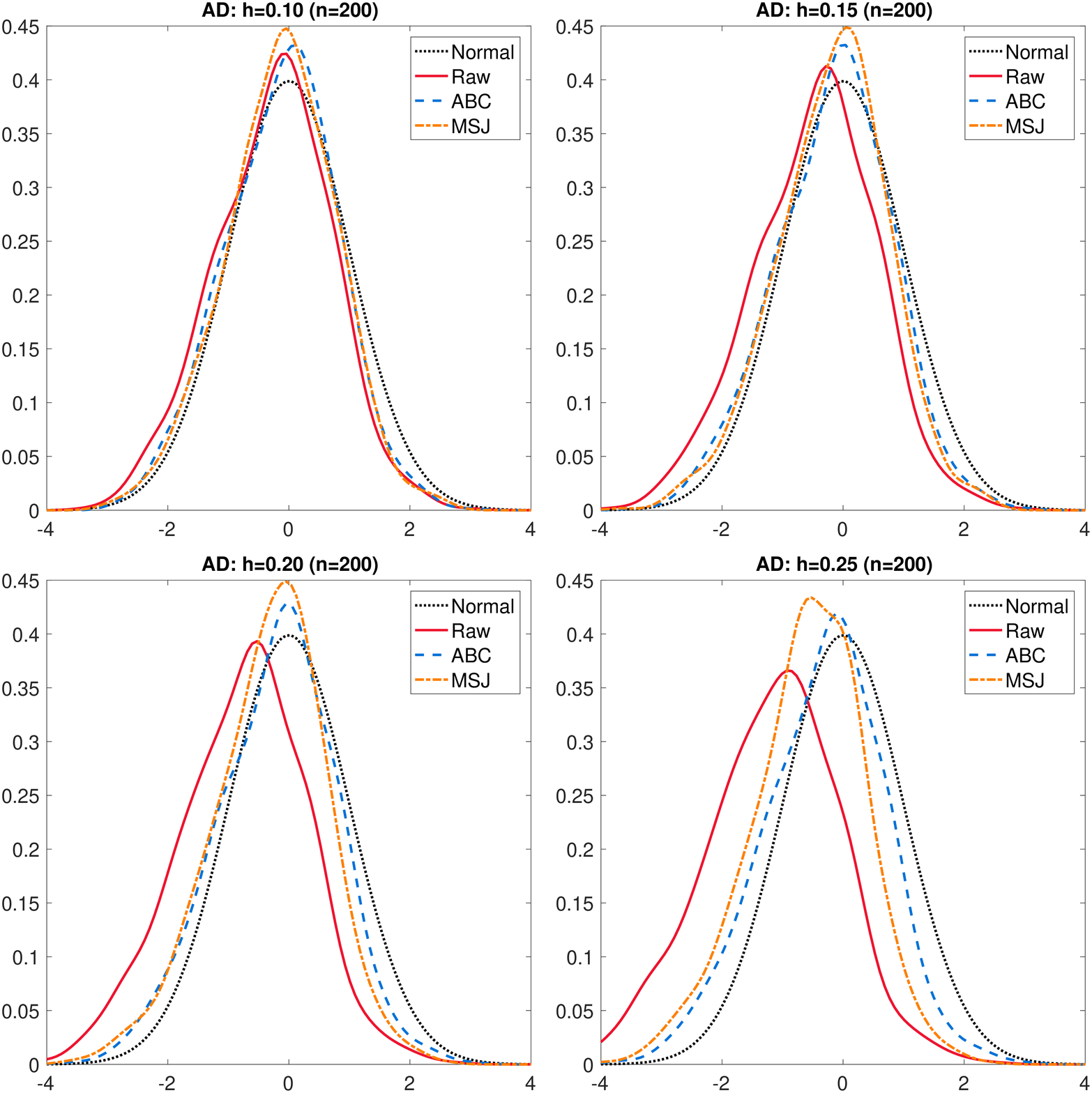}
\caption{AD: Densities of t-statistics and standard normal R.V.}
\label{fig:AD200Density}
\end{figure}

The results for $n=1000$ are given in Figures \ref{fig:AD1000MSE}, \ref{fig:AD1000CI}, and \ref{fig:AD1000Density}. Since the bandwidth $h$ decreases as the sample size gets larger, we only show the results for $h\leq 0.3$ in this case. 

\begin{figure}[!htbp]
\centering
\includegraphics[width=1\textwidth]{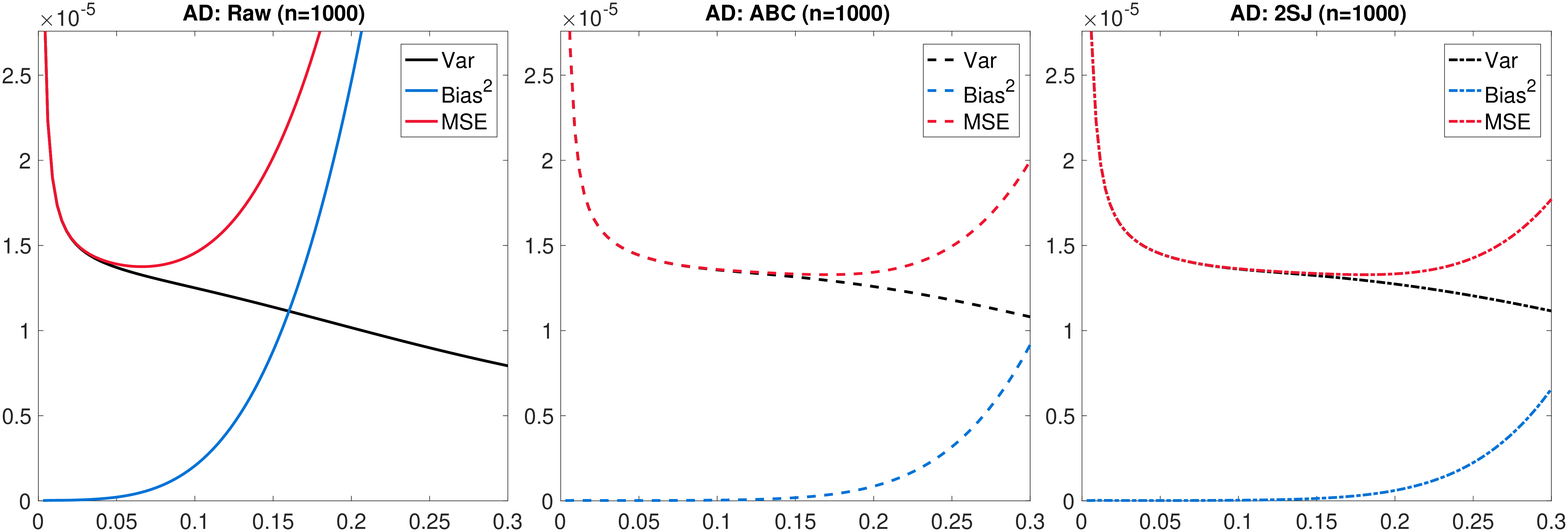}
\caption{AD: Decomposition of Mean Squared Error}
\label{fig:AD1000MSE}
\end{figure}

\begin{figure}[!htbp]
\centering
\includegraphics[width=1\textwidth]{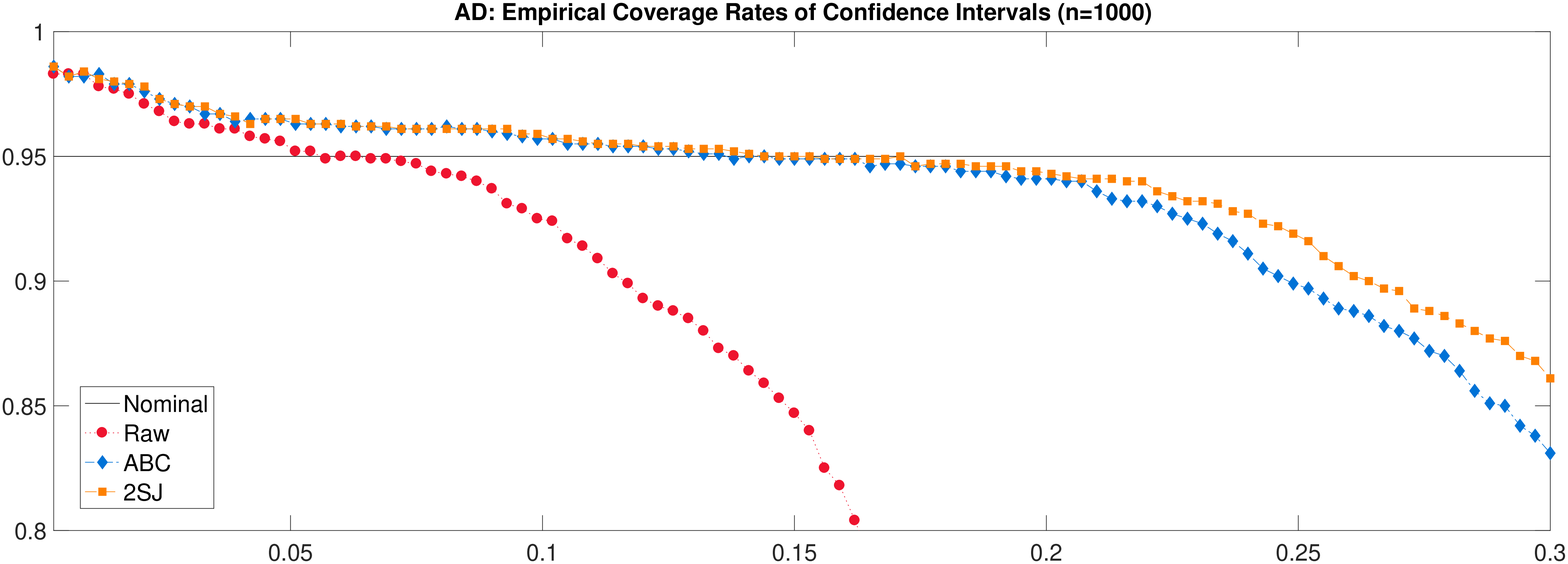}
\caption{AD: Empirical Coverage Rates of Confidence Intervals}
\label{fig:AD1000CI}
\end{figure}

\begin{figure}[!htbp]
\centering
\includegraphics[width=1\textwidth]{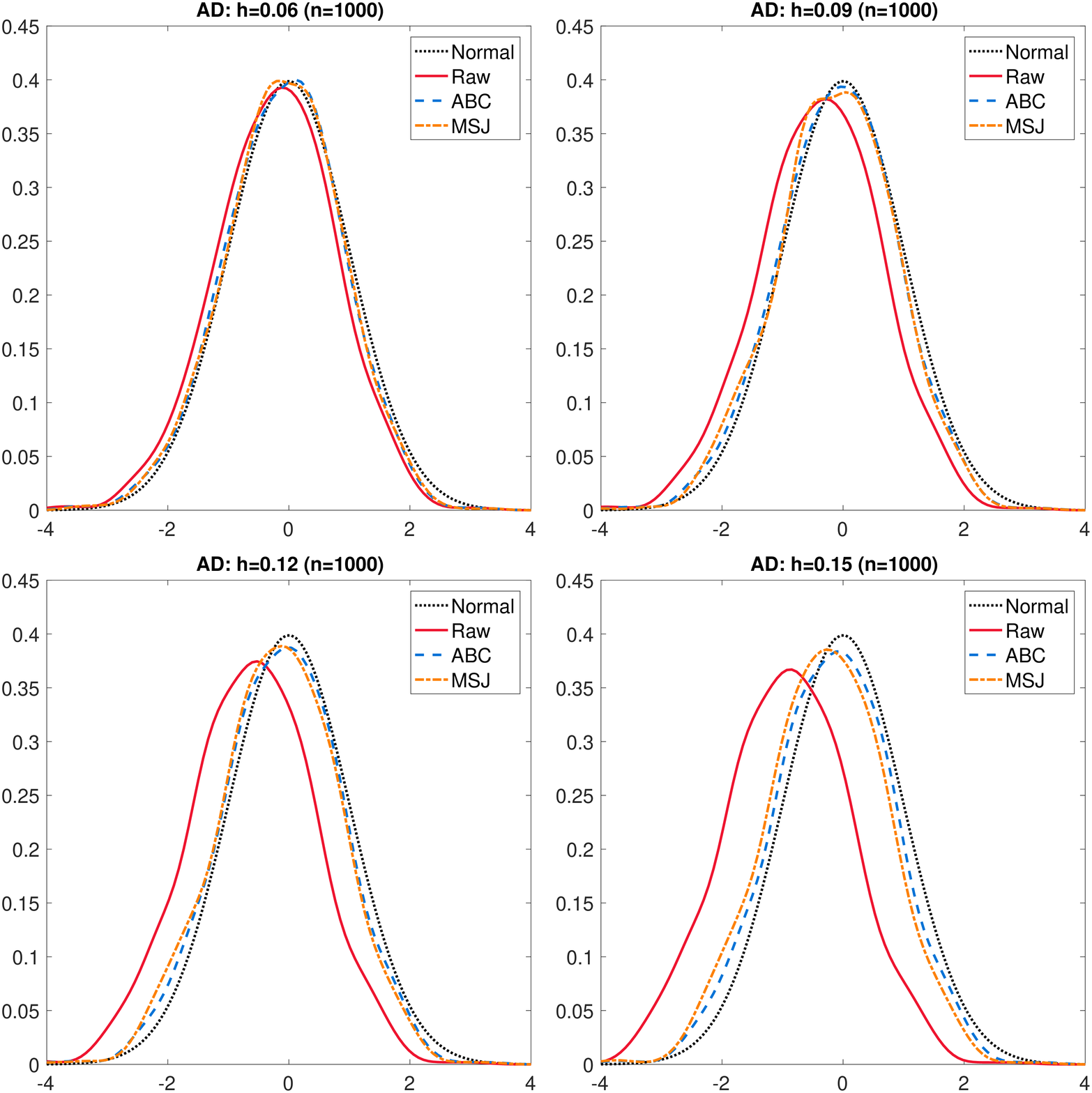}
\caption{AD: Densities of t-statistics and standard normal R.V.}
\label{fig:AD1000Density}
\end{figure}

\vfill
\newpage

\subsection{Integrated Squared Density (ISD) Estimator} 

In the main-text, we have obtain the following expressions of the two biases:
\begin{align*}
	 \Bnl &= \int [ \hatgamma_n(x) - \bargamma_n(x) ]^2 dx, \\
	 \Banb &= 2 \int \gamma_0(x) [\bargamma_n(x) - \gamma_0(x)] dx.
\end{align*}

It can be shown that
\begin{align*}
	& [ \hatgamma_n(x) - \bargamma_n(x) ]^2 \\
	=\,& \frac{1}{n^2} \sum_{i,j=1}^n K_h(x-z_i) K_h(x-z_j) - \frac{2}{n} \sum_{i=1}^n K_h(x-z_i) \int K(u) \gamma_0(x-hu) du \\
		&\quad + \Big( \int K(u) \gamma_0(x-hu) du \Big)^2 \\
	=\,& \frac{1}{n^2} \sum_{i=1}^n [K_h(x-z_i)]^2  + \frac{1}{n^2} \sum_{i,j=1}^n K_h(x-z_i) K_h(x-z_j) \\
		&\quad - \frac{2}{n} \sum_{i=1}^n K_h(x-z_i) \int K(u) \gamma_0(x-hu) du + \Big( \int K(u) \gamma_0(x-hu) du \Big)^2.
\end{align*}
It then follows that
\begin{align*}
	\Bnl = \frac{1}{n^2} \sum_{i=1}^n \int [K_h(x-z_i)]^2 dx + \oP\Big( \frac{1}{nh^{d_z}}\Big),
\end{align*}
where the first term is of order $1/(nh^{d_z})$. Hence, we estimate the nonlinear bias by
\begin{align*}
	\whatBnl = \frac{1}{n^2} \sum_{i=1}^n \int [K_h(x-z_i)]^2 dx.
\end{align*}

As for the averaged nonparametric bias, we estimate it by
\begin{align*}
	\whatBanb = 2 \int \hatgamma_n(x) [ \hat{\bargamma}_n(x) - \hatgamma_n(x) ] dx,
\end{align*}
where $\hat{\bargamma}_n$ is constructed in the same way as in the previous subsection. 

Since we used Gaussian kernel in the simulation, all the above integrals are calculated over the interval $[\min(z_i) - 4, \max(z_i) +4]$ with 500 grids. In Matlab, one can also use the function \texttt{vpaintegral}, but the computation time is significantly longer. 

The results for $n=50,200,$ and 1000 are given below. Here we used a five-scale jackknife. Since the odd moments of a Gaussian kernel with zero mean are all zero, the \textquotedblleft over-smoothing\textquotedblright{} biases, are only non-zero for even orders. In view if this, the weights are given as below:
\begin{align*}
	\left( \begin{matrix}
		w_1 \\
		w_2 \\
		w_3 \\
		w_4 \\
		w_5
	\end{matrix} \right) = \left( \begin{matrix}
		1 & 1 & 1 & 1 & 1 \\
		\eta_1^2 & \eta_2^2 & \eta_3^2 & \eta_4^2 & \eta_5^2 \\
		\eta_1^4 & \eta_2^4 & \eta_3^4 & \eta_4^4 & \eta_5^4 \\
		\eta_1^6 & \eta_2^6 & \eta_3^6 & \eta_4^6 & \eta_5^6 \\
		\eta_1^{-1} & \eta_2^{-1} & \eta_3^{-1} & \eta_4^{-1} & \eta_5^{-1} \\
	\end{matrix} \right)^{-1} \left( \begin{matrix}
		1 \\
		0 \\
		0 \\
		0 \\
		0
	\end{matrix} \right).
\end{align*}
We set $\eta=(3/5,4/5,1,6/5,7/5)$ in the simulation.

\begin{figure}[!htbp]
\centering
\includegraphics[width=1\textwidth]{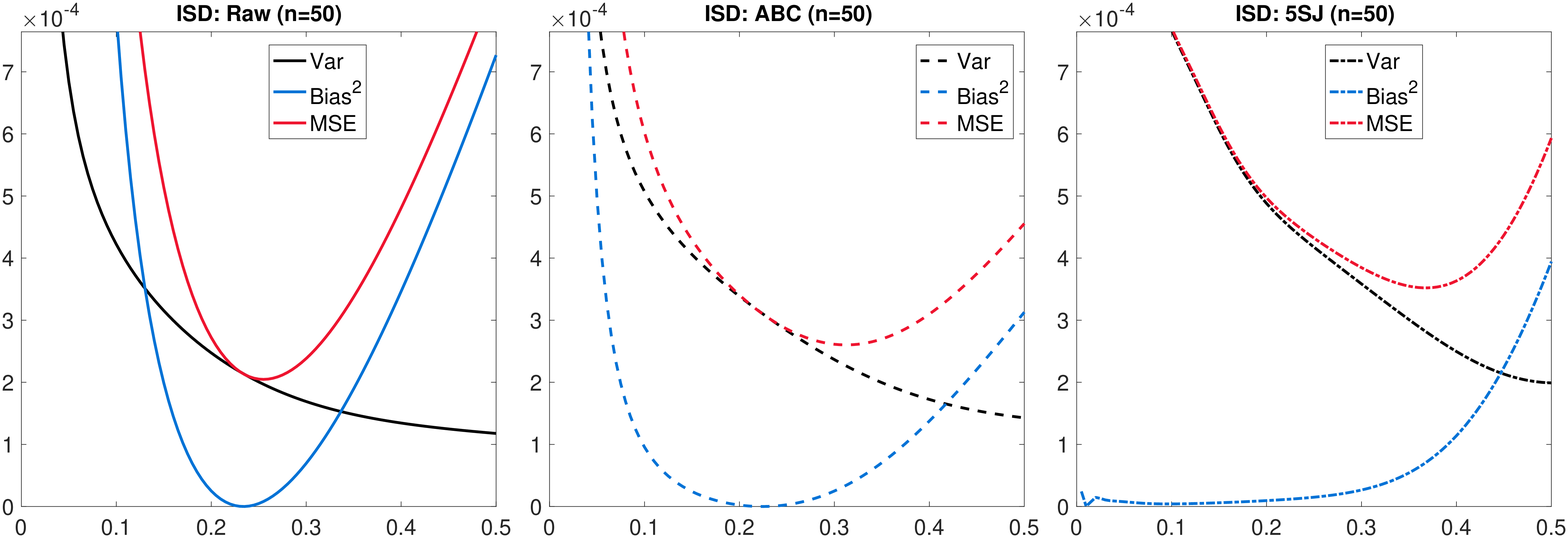}
\caption{ISD: Decomposition of Mean Squared Error}
\label{fig:ISD50MSE}
\end{figure}

\begin{figure}[!htbp]
\centering
\includegraphics[width=1\textwidth]{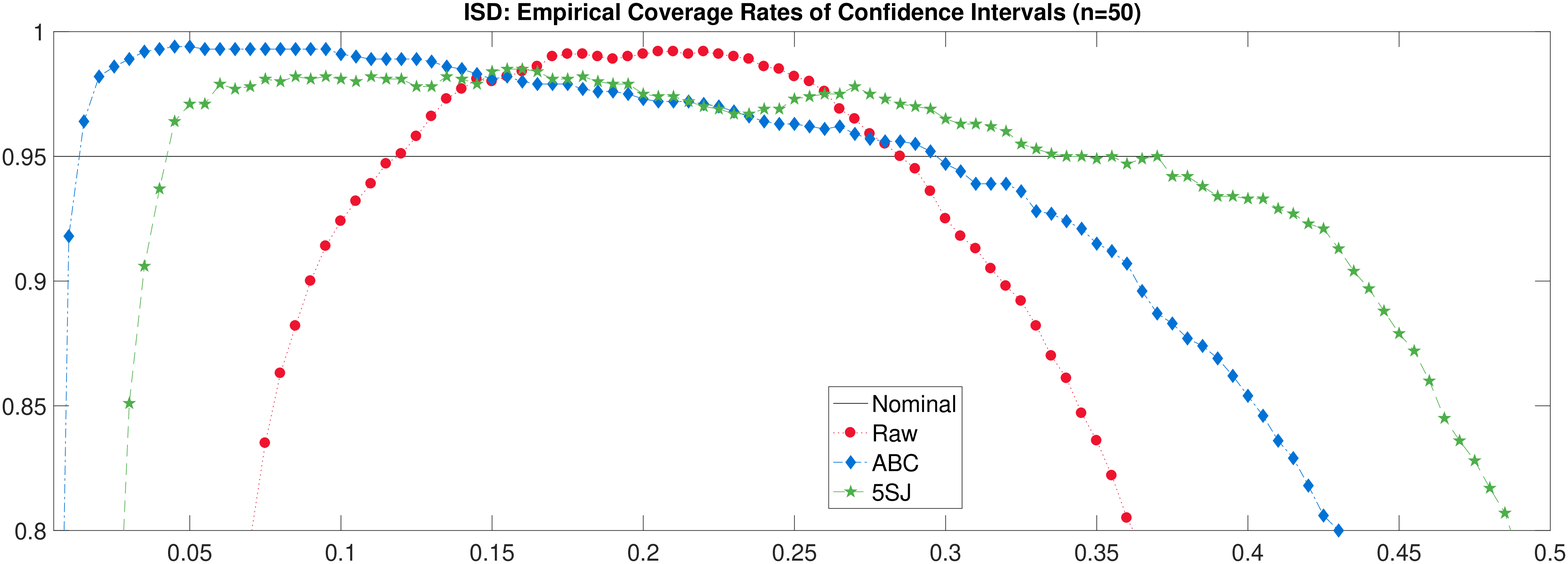}
\caption{ISD: Empirical Coverage Rates of Confidence Intervals}
\label{fig:ISD50CI}
\end{figure}

\begin{figure}[!htbp]
\centering
\includegraphics[width=1\textwidth]{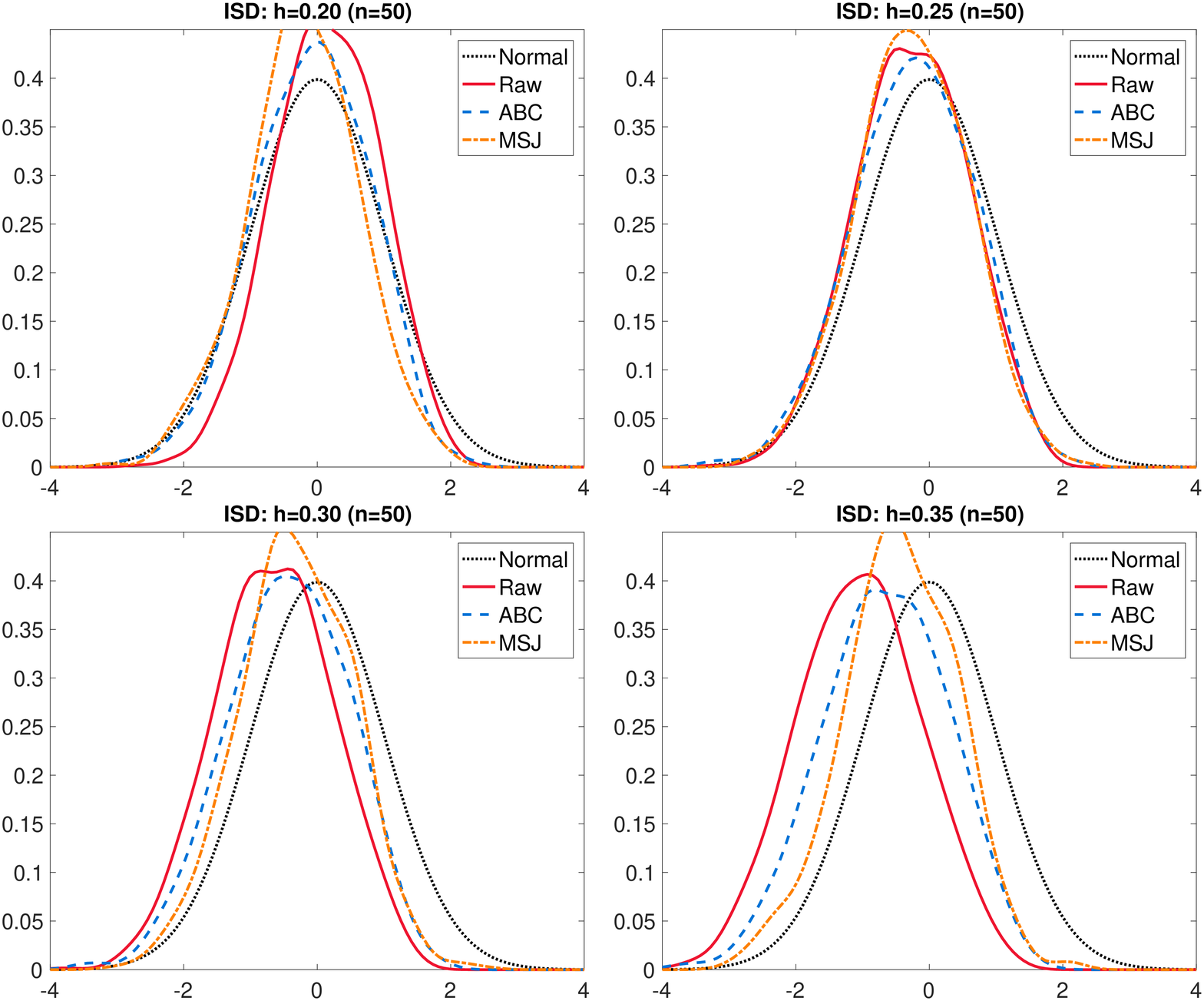}
\caption{ISD: Densities of t-statistics and standard normal R.V.}
\label{fig:ISD50Density}
\end{figure}

\begin{figure}[!htbp]
\centering
\includegraphics[width=1\textwidth]{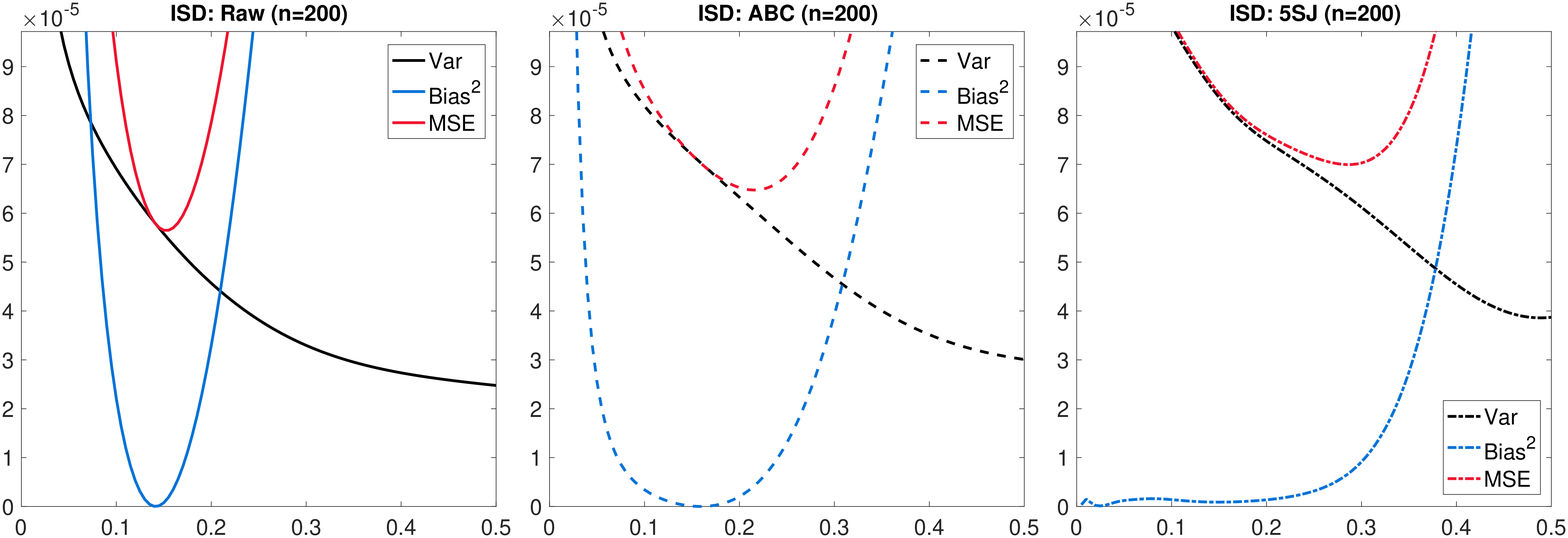}
\caption{ISD: Decomposition of Mean Squared Error}
\label{fig:ISD200MSE}
\end{figure}

\begin{figure}[!htbp]
\centering
\includegraphics[width=1\textwidth]{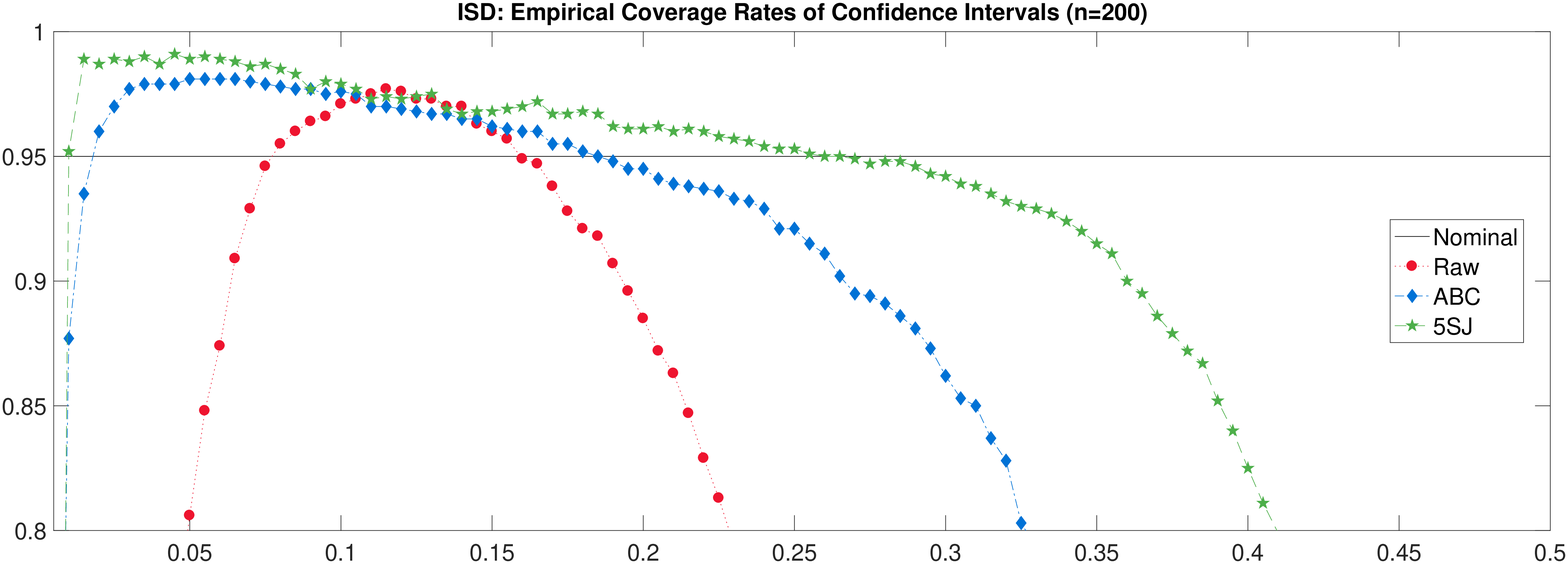}
\caption{ISD: Empirical Coverage Rates of Confidence Intervals}
\label{fig:ISD200CI}
\end{figure}

\begin{figure}[!htbp]
\centering
\includegraphics[width=1\textwidth]{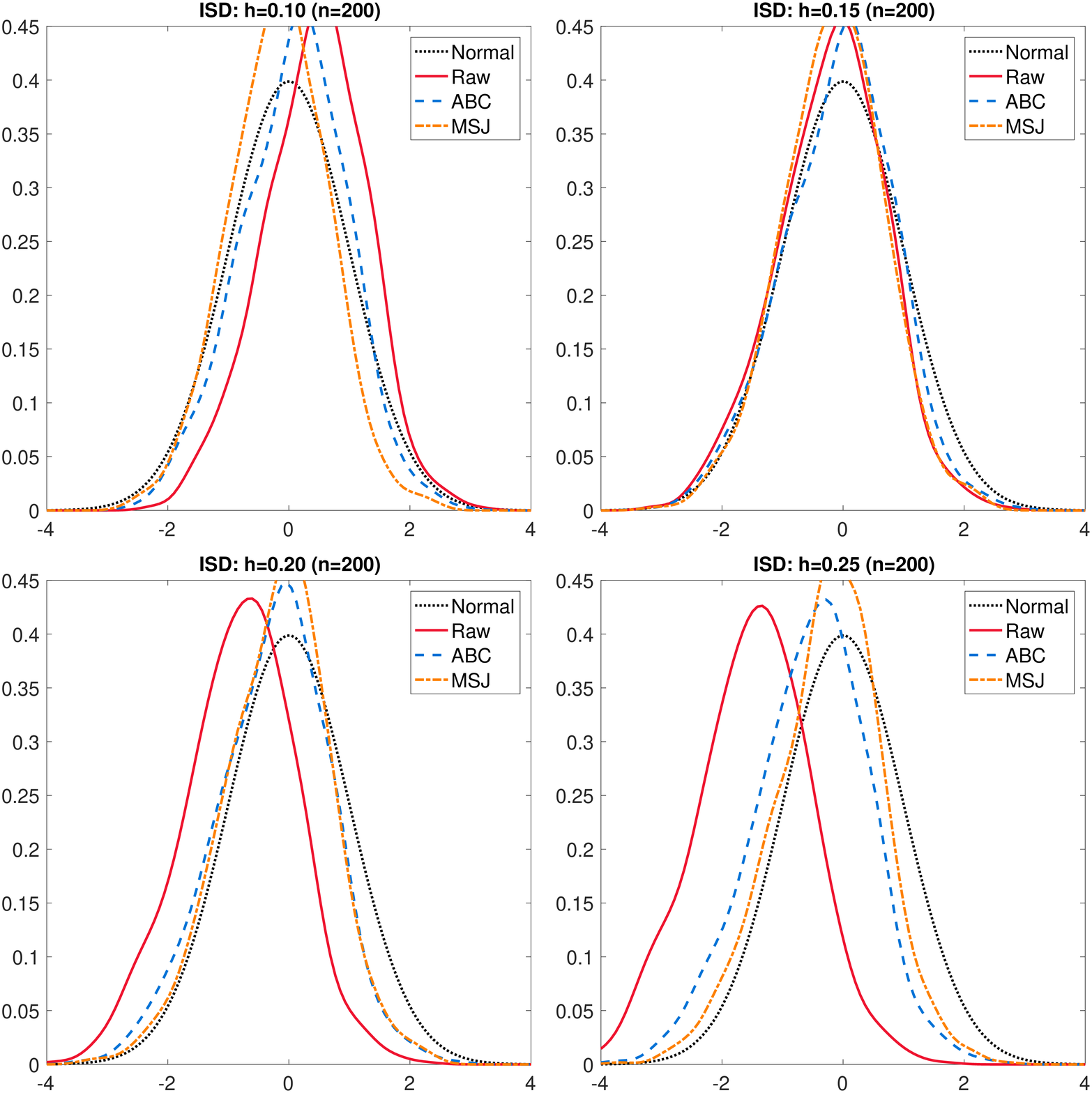}
\caption{ISD: Densities of t-statistics and standard normal R.V.}
\label{fig:ISD200Density}
\end{figure}

\begin{figure}[!htbp]
\centering
\includegraphics[width=1\textwidth]{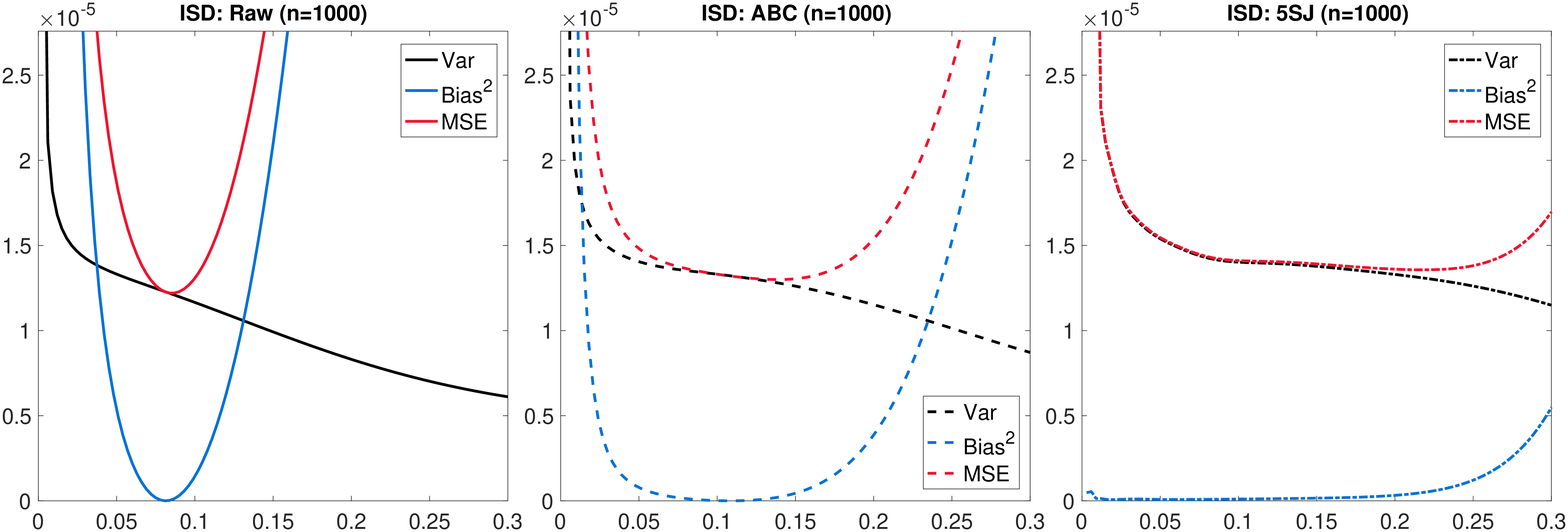}
\caption{ISD: Decomposition of Mean Squared Error}
\label{fig:ISD1000MSE}
\end{figure}

\begin{figure}[!htbp]
\centering
\includegraphics[width=1\textwidth]{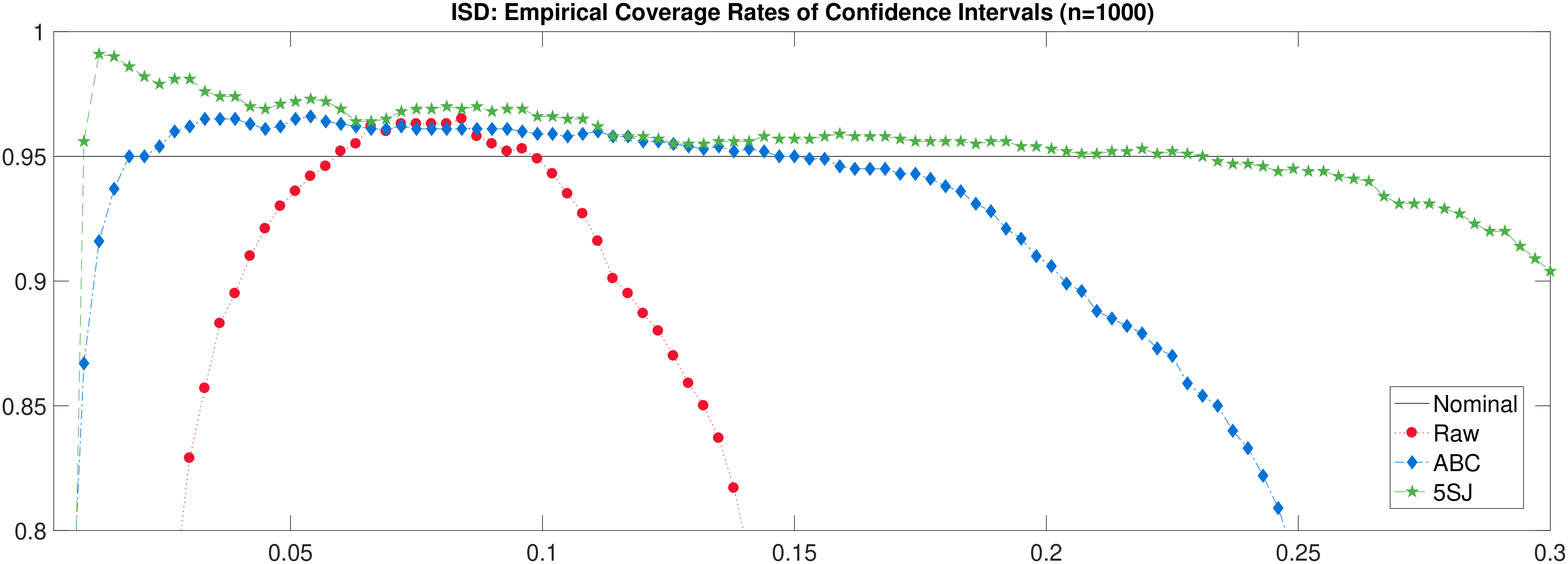}
\caption{ISD: Empirical Coverage Rates of Confidence Intervals}
\label{fig:ISD1000CI}
\end{figure}

\begin{figure}[!htbp]
\centering
\includegraphics[width=1\textwidth]{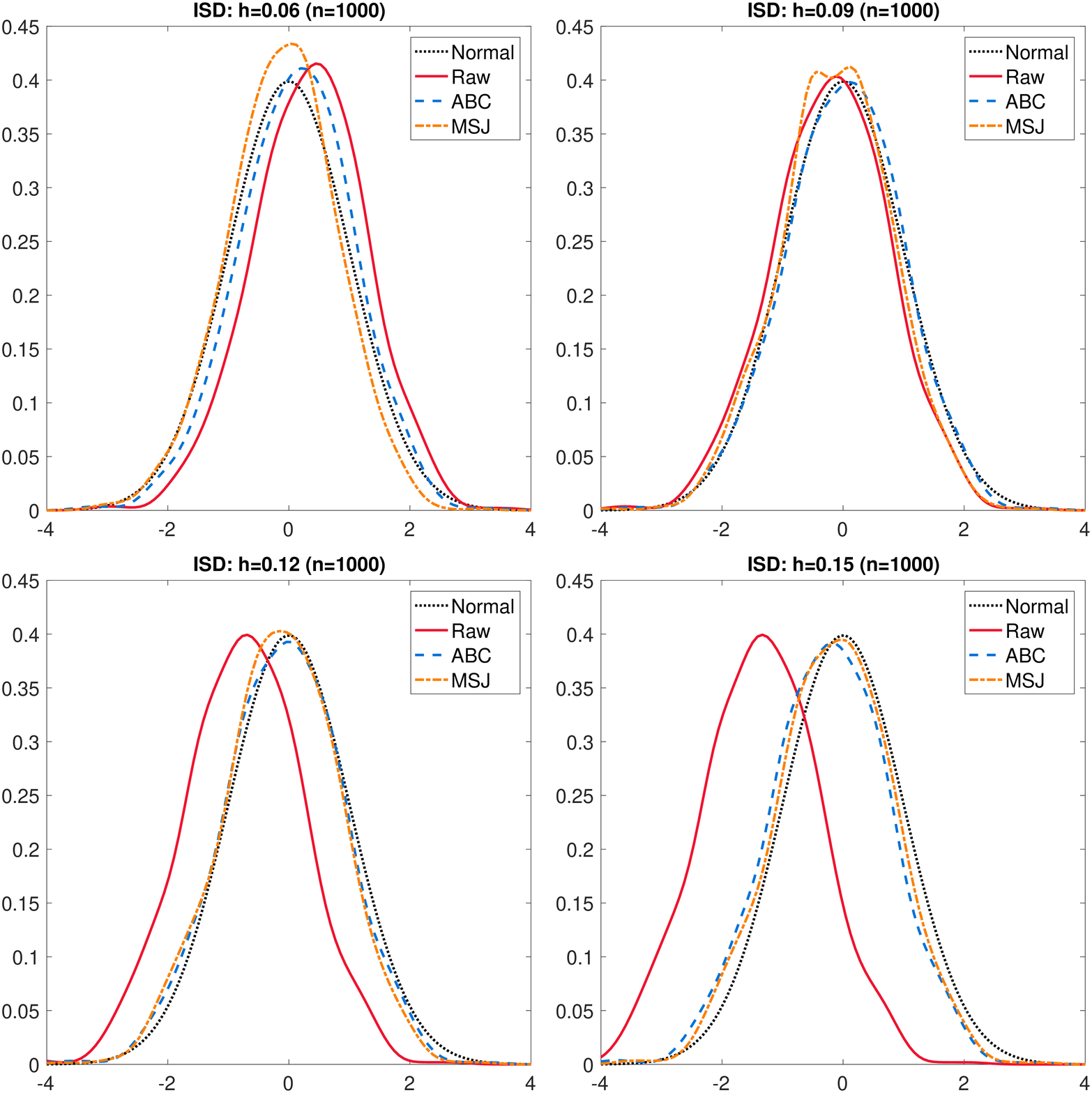}
\caption{ISD: Densities of t-statistics and standard normal R.V.}
\label{fig:ISD1000Density}
\end{figure}

\subsection{Density-Weighted Average Derivative (DWAD) Estimator} 

The DWAD estimator, which corresponds to the case $|\bar{\lambda}|=1$ and $w=2$ in \cite{NHR:2004}, is given by ($f$ is the density of $x_i$)
\begin{align*}
	\hattheta_n = - \frac{2}{n} \sum_{i=1}^n \frac{\partial \hat{f}}{\partial x}(x_i) \, y_i,   
\end{align*}
where
\begin{align*}
	 \frac{\partial \hat{f}}{\partial x}(x_i) = \frac{1}{(n-1) h^{d}} \frac{-1}{h} \sum_{j\neq i} K'\Big( \frac{x_j - x_i}{h} \Big).
\end{align*}
Therefore, we get
\begin{align*}
	\hattheta_n = - \frac{2}{n} \sum_{i=1}^n \frac{\partial \hat{f}}{\partial x}(x_i) \, y_i = \frac{2}{n(n-1)} \sum_{i=1}^n \sum_{j\neq i} \frac{1}{h^{d+1}} K'\Big( \frac{x_j - x_i}{h} \Big) \, y_i.
\end{align*}

In this case, we have $\hatgamma_n = \partial \hat{f} / \partial x$. It can be shown that
\begin{align*}
	\whatBnl = \hat{\bar{\theta}}_n - \hattheta_n, \,\,\text{where }  \hat{\bar{\theta}}_n = - \frac{2}{n(n-1)} \sum_{i=1}^n \sum_{j\neq i} K_h( x_j - x_i) \frac{\partial \hat{f}}{\partial x}(x_j) \, y_i.
\end{align*} 
If one choose the density of $\calN(0,I_d)$ as the kernel function, then the equivalent kernel used in $\hat{\bar{\theta}}_n$ is $\calN(0, 2 I_d)$. This means that the equivalent kernel for $\whatBnl$ is essentially a twicing kernel.

The simulation results for $n=50,200,$ and 1000 are give below. We note that the coverage rates in the case $n=1000$ is slightly higher than the nominal level when $h$ is small. A possible explanation is that the variance correction term provided by \cite{Cattaneo&Crump&Jansson:2014} (Case (b) of Theorem 2 therein) is of order $1/n$. Hence its correction effect becomes smaller when the sample size increases.

\begin{figure}[!htbp]
\centering
\includegraphics[width=1\textwidth]{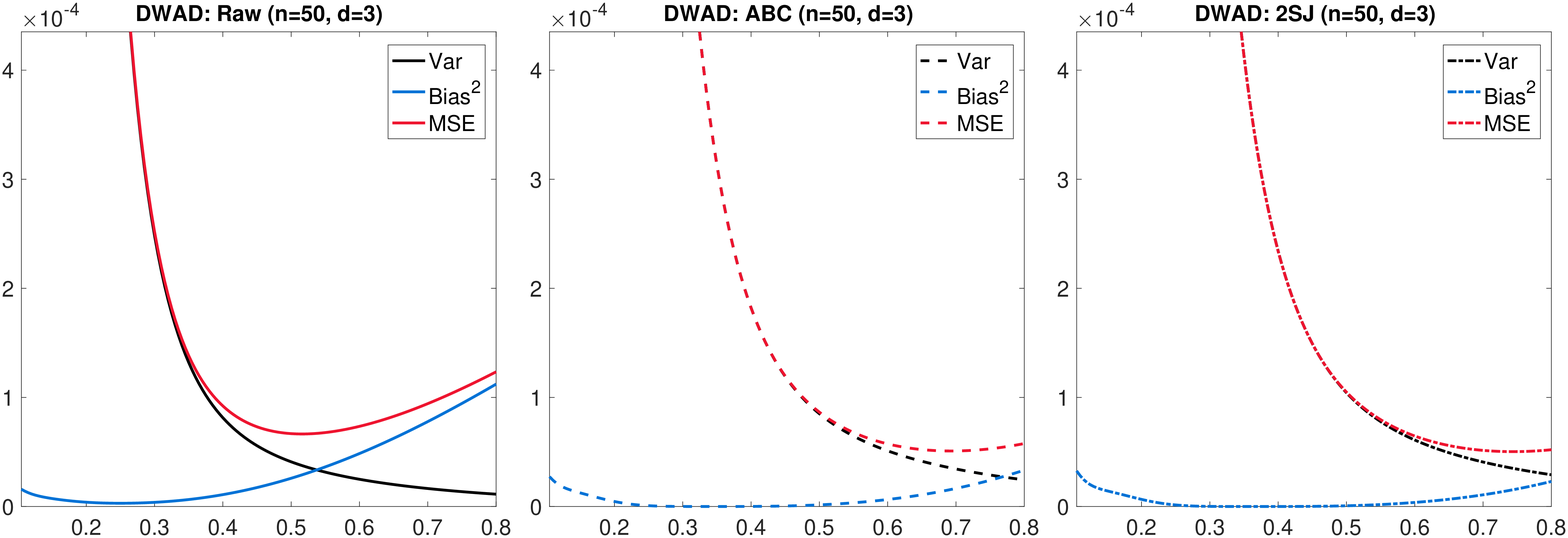}
\caption{DWAD: Decomposition of Mean Squared Error}
\label{fig:DWAD50MSE}
\end{figure}

\begin{figure}[!htbp]
\centering
\includegraphics[width=1\textwidth]{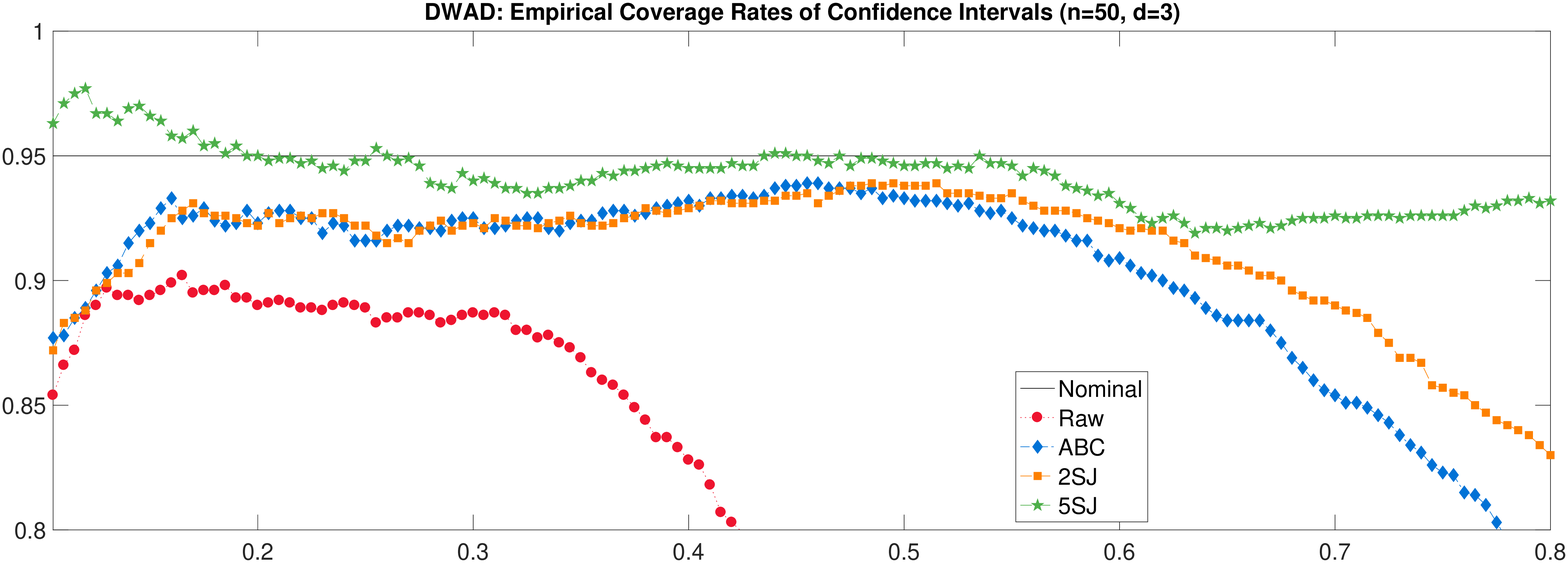}
\caption{DWAD: Empirical Coverage Rates of Confidence Intervals}
\label{fig:DWAD50CI}
\end{figure}

\begin{figure}[!htbp]
\centering
\includegraphics[width=1\textwidth]{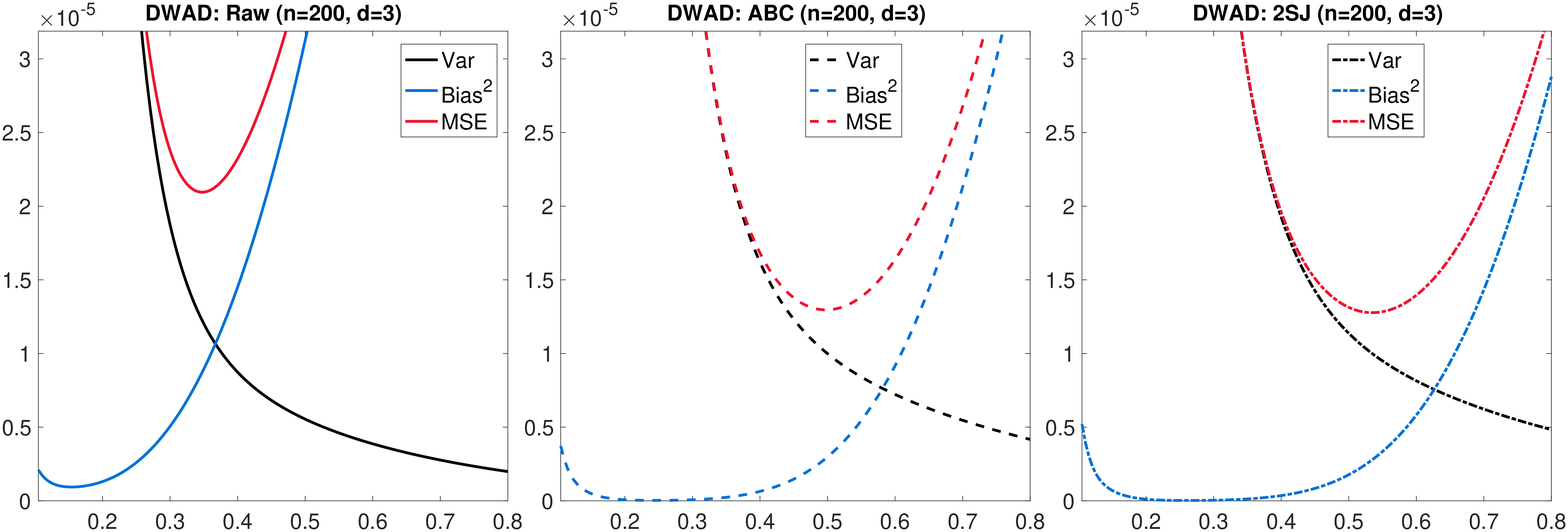}
\caption{DWAD: Decomposition of Mean Squared Error}
\label{fig:DWAD200MSE}
\end{figure}

\begin{figure}[!htbp]
\centering
\includegraphics[width=1\textwidth]{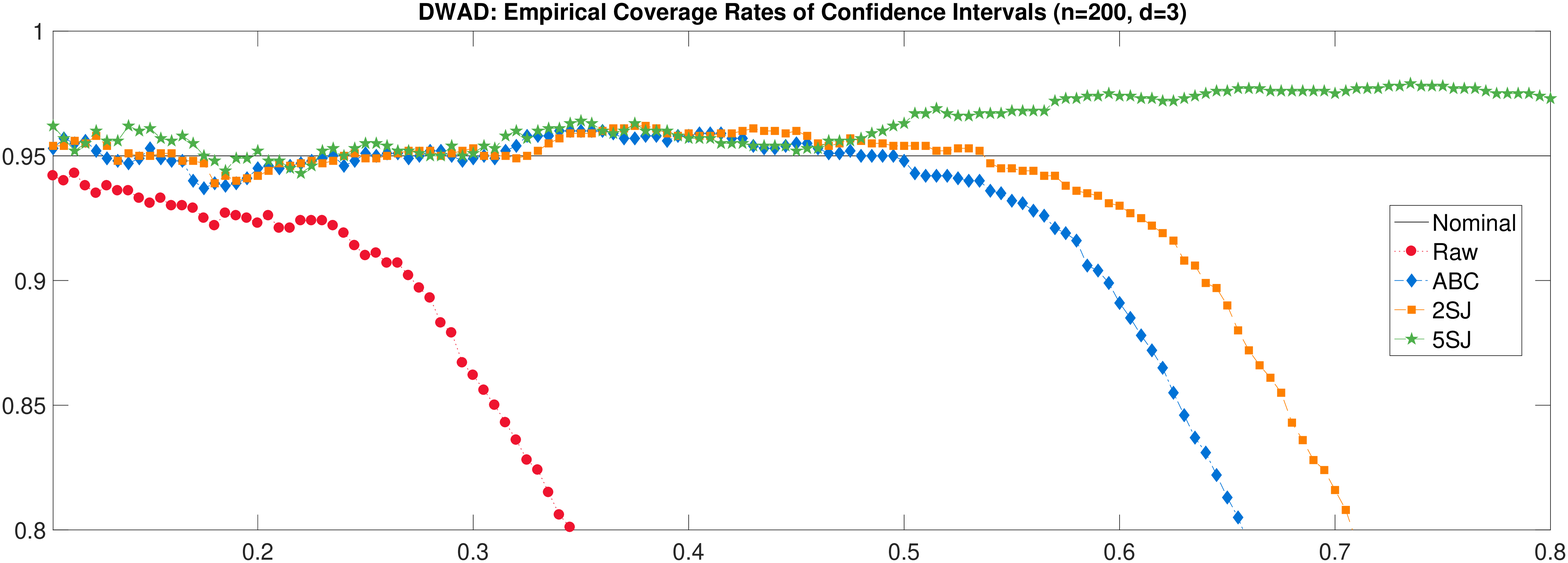}
\caption{DWAD: Empirical Coverage Rates of Confidence Intervals}
\label{fig:DWAD200CI}
\end{figure}

\begin{figure}[!htbp]
\centering
\includegraphics[width=1\textwidth]{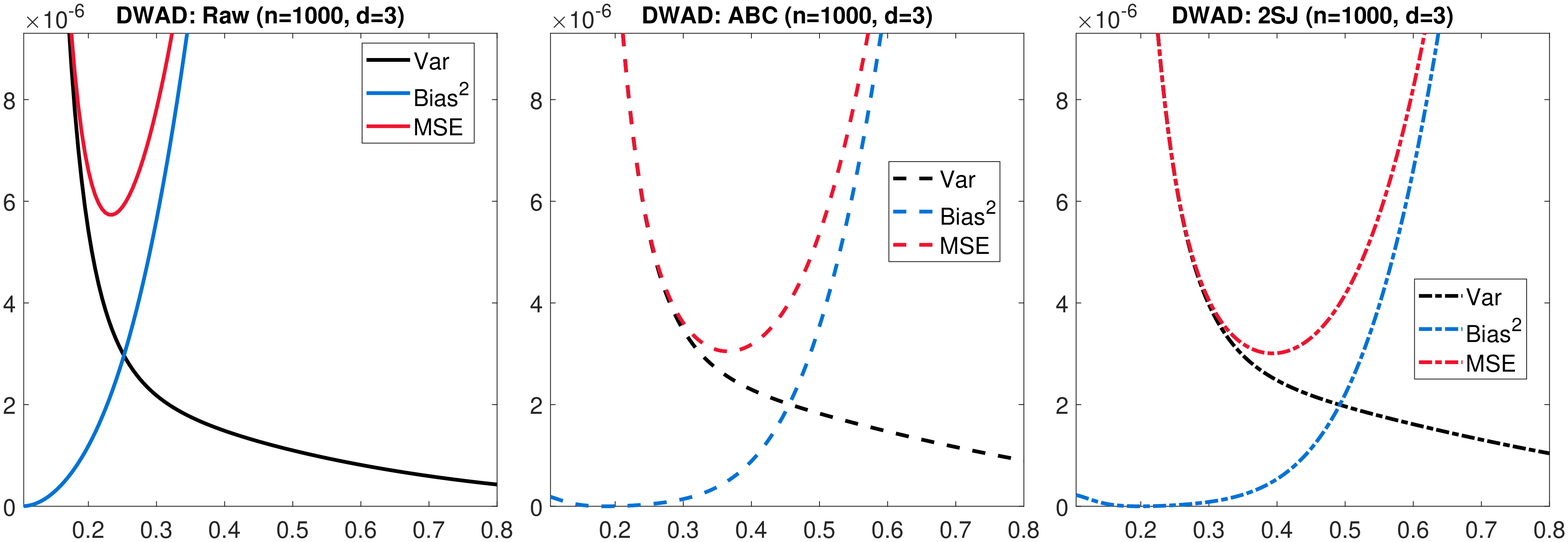}
\caption{DWAD: Decomposition of Mean Squared Error}
\label{fig:DWAD1000MSE}
\end{figure}

\begin{figure}[!htbp]
\centering
\includegraphics[width=1\textwidth]{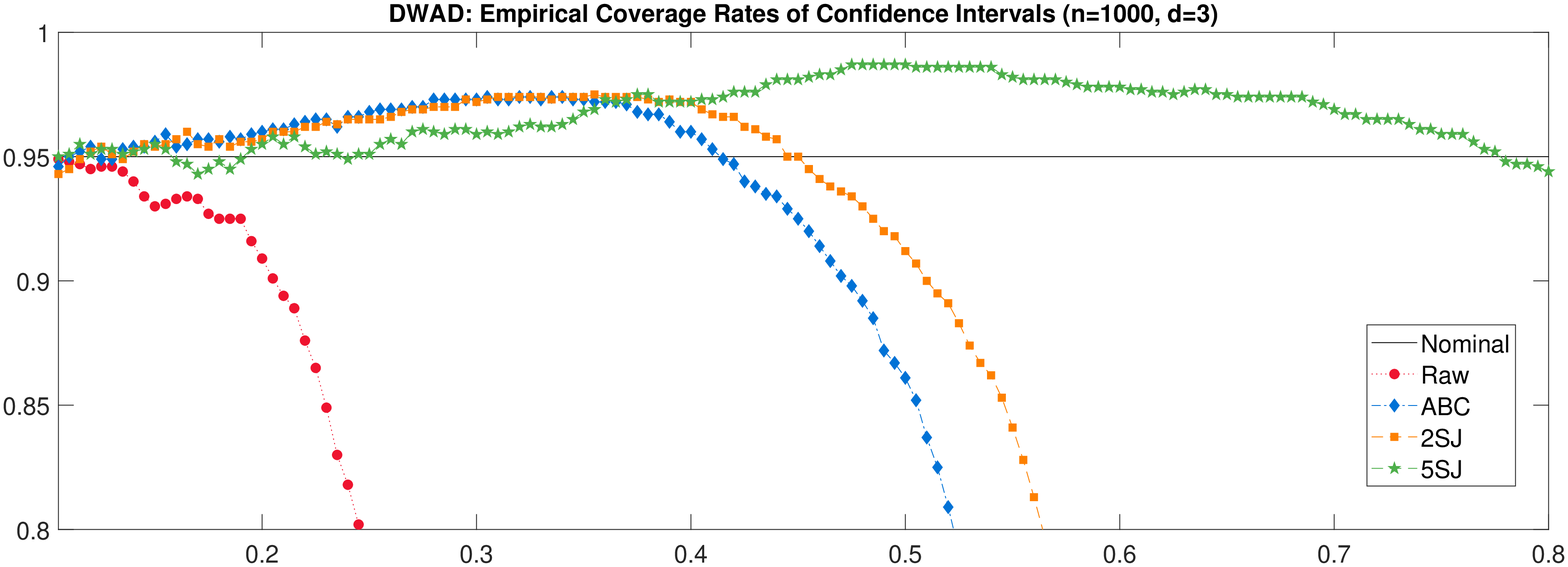}
\caption{DWAD: Empirical Coverage Rates of Confidence Intervals}
\label{fig:DWAD1000CI}
\end{figure}

\vfill
\newpage

\bigskip

\section{Additional Example: Average Treatment Effect Estimator} \label{sec:ATE}


Let $U$ and $V$ be two random vectors/variables. Denote by $a$ and $b$ two known scalar-valued functions (can be extended to the vector or matrix cases). Suppose $\{ X_i \}_{i=1}^n$ is an \text{i.i.d.} sample of $d$-dimensional random vectors. Hence, for $i\neq j$, we have $\bbE[ a(U_j) | X_j, X_i ] = \bbE[ a(U_j) | X_j]$ and $\bbE[ b(V_j) | X_j, X_i ] = \bbE[ b(V_j) | X_j]$.

Define the following estimators at some real vector $x$:
\begin{align}\label{eq:hatgammaUV}
	\whatgamma_n^U(x) = \frac{1}{n} \sum_{j=1}^n \calK_h^U (X_j, x) \, a(U_j) \quad \text{and} \quad
	\whatgamma_n^V(x) = \frac{1}{n} \sum_{j=1}^n \calK_h^V (X_j, x) \, b(V_j),
\end{align}
where $h$ is the bandwidth. As discussed in the main-text, when evaluated at some sample point $X_i$, we use the \textquotedblleft leave-one-out\textquotedblright{} versions of the above estimators. 

\begin{named-exmp}[Nadaraya-Watson estimator]
Consider a nonparametric regression model:
\begin{align*}
	Y = \bm{m}(X) + \epsilon, \,\, \text{where } \bbE[\epsilon|X] =0.
\end{align*}

Let $U = Y$, $V=1$, $a$ and $b$ both be the identity function, $\calK_h^U(s,t) = \calK_h^V(s,t) = K_h(s-t) = K\big( (s-t)/h \big) / h^{d}$, where $K$ is a scalar-valued kernel function. Then the Nadaraya-Watson estimator (see \cite{Nadaraya:1964} and \cite{Watson:1964}) of the unknown function $\bm{m}(\cdot)$ is given by
\begin{align*}
	\widehat{\bm{m}}(x) = \frac{\sum_{j=1}^n K\big( \frac{X_j - x}{h}  \big) Y_j}{\sum_{j=1}^n K\big( \frac{X_j - x}{h} \big)} =  \frac{\whatgamma_n^Y}{\whatgamma_n^1} =  \frac{\whatgamma_n^U}{\whatgamma_n^V}  
\end{align*}
We note that \cite{Nadaraya:1964} adopted a representation that is essentially the same as above (the author used $\varphi$ and $\psi$, instead of $\hatgamma^U$ and $\hatgamma^V$, and consider the case $d=1$). 
\end{named-exmp}

Let $D_i$ be a binary treatment for subject $i$. Denote by $Y_{i}^0$ and $Y_{i}^1$ the potential outcomes for subject $i$  in the cases of not being treated and being treated, respectively. The observed outcome is given by $Y_i \coloneqq Y_{i}^1 D_i + Y_{i}^0 (1-D_i)$. The average treatment effect (ATE) is then defined as $\theta \coloneqq \bbE[ Y_{i}^1 - Y_{i}^0]$. 

Let $X_i$ be a set of covariates such that the unconfoundedness assumption holds true:
\begin{align*}
	(Y_{i}^0, Y_{i}^1 ) \perp D_i \mid X_i.
\end{align*}
Therefore, conditional on covariates, potential outcomes are independent of the treatment assignment. Hence, we have $z_i = (Y_i, D_i, X_i^\intercal)^\intercal$.

Since $D_i$ is binary, one readily gets that $D_i^2 = D_i$, which is equivalent to $D_i(1-D_i) = 0$. This further implies that $Y_i D_i \equiv Y_{i}^1 D_i$ and $Y_i (1-D_i) \equiv Y_{i}^0 (1 - D_i)$. Let $\bm{e}(X_i) \coloneqq \bbE[D_i \,\big|\, X_i]$. One can deduce that
\begin{align*}
	& \bbE\Big[ \frac{Y_{i} D_i }{\bm{e}(X_i)}  - \frac{Y_{i} (1-D_i)}{1-\bm{e}(X_i)} \, \Big] = \bbE\Big[ \bbE\Big( \frac{Y_{i} D_i }{\bm{e}(X_i)}  - \frac{Y_{i} (1-D_i)}{1-\bm{e}(X_i)} \, \cond X_i \Big) \Big] \\
	=\,& \bbE\Big[ \frac{ \bbE( Y_{i} D_i \cond X_i) }{\bm{e}(X_i)} - \frac{ \bbE( Y_{i} (1-D_i) \cond X_i) }{1-\bm{e}(X_i)} \Big] = \bbE\Big[ \frac{ \bbE( Y_{i}^1 D_i \cond X_i) }{\bm{e}(X_i)} - \frac{ \bbE( Y_{i}^0 (1-D_i) \cond X_i) }{1-\bm{e}(X_i)} \Big] \\
	=\,& \bbE\Big[ \frac{ \bbE( Y_{i}^1 \cond X_i) \, \bbE( D_i \cond X_i ) }{\bm{e}(X_i)} - \frac{ \bbE( Y_{i}^0 \cond X_i) \, \bbE( D_i \cond X_i )}{1-\bm{e}(X_i)} \Big] = \bbE[ \bbE( Y_i^1 \cond X_i ) - \bbE( Y_i^0 \cond X_i ) ] \\
	=\,& \bbE[ Y_i^1 - Y_i^0 ] = \theta.
\end{align*}


For notation simplicity, let $Y_1 \coloneqq Y D$, $Y_0 \coloneqq Y(1-D)$, $D_1 \coloneqq D$, and $D_0 \coloneqq 1- D$. Define the following kernel-based estimators:
\begin{gather*}
	\whatgamma_n^{D_1}(X_i) = \frac{1}{n-1} \sum_{j\neq i} \calK_h(X_i, X_j) \, D_j \quad \whatgamma_n^{D_0}(X_i) = \frac{1}{n-1} \sum_{j\neq i} \calK_h(X_i, X_j) (1-D_j), \\
	\whatgamma_n^{1}(X_i) = \frac{1}{n-1} \sum_{j\neq i} \calK_h^1(X_i, X_j).
\end{gather*}

If one chooses the Nadaraya-Watson (NW) method, then we have 
\begin{align*}
	\calK_h(s, t) = \calK_h^1(s, t) = \frac{1}{h^d} K\big( \frac{s-t}{h} \big).
\end{align*}
Accordingly, one can estimate the ATE by 
\begin{align}
\begin{gathered} \label{eq:ATE-NW}
	\whattheta_n^{\,\texttt{NW}} \coloneqq \frac{1}{n} \sum_{i=1}^n \Big( \frac{\whatgamma_n^{1}(X_i)}{\whatgamma_n^{D_1}(X_i)} \, Y_i D_i - \frac{\whatgamma_n^{1}(X_i)}{\whatgamma_n^{D_0}(X_i)} \, Y_i (1-D_i) \Big).
\end{gathered}
\end{align}
This is the estimator considered by \cite{Hirano&Imbens&Ridder:2003}.

In this case, $\gamma=(\gamma^{1}, \gamma^{D_1}, \gamma^{D_0})$ and the $g$ function writes as
\begin{align*}
	g(z_i, \theta, \gamma) = \frac{\gamma^{1}(X_i)}{\gamma^{D_1}(X_i)} \, Y_i D_i - \frac{\gamma^{1}(X_i)}{\gamma^{D_0}(X_i)} \, Y_i (1-D_i) - \theta.
\end{align*}
We will use $\gamma_0^1$ to denote the density function of $X$, $\gamma_0^{D_1}(X_i) = \bm{e}(X_i) \gamma^1(X_i)$ and $\gamma_0^{D_0}(X_i) = [1-\bm{e}(X_i)] \gamma^1(X_i)$. Let $\gamma_0 = (\gamma_0^1, \gamma_0^{D_1}, \gamma_0^{D_0})^\intercal$.

Following the definition in the main-text, we readily get the following expressions for the two biases:
\begin{align*}
	\Banb = \frac{1}{n} & \sum_{i=1}^n \Big\{  \Big( \frac{1}{\gamma_0^{D_1}(X_i)} \big[ \bargamma_n^{1}(X_i) - \gamma_0^{1}(X_i) \big] - \frac{\gamma_0^{1}(X_i)}{[\gamma_0^{D_1}(X_i)]^2} \big[ \bargamma_n^{D_1}(X_i) - \gamma_0^{D_1}(X_i) \big] \Big) Y_{1i} \\
		& - \Big( \frac{1}{\gamma_0^{D_0}(X_i)} \big[ \bargamma_n^{1}(X_i) - \gamma_0^{1}(X_i) \big] - \frac{\gamma_0^{1}(X_i)}{[\gamma_0^{D_0}(X_i)]^2} \big[ \bargamma_n^{D_0}(X_i) - \gamma_0^{D_0}(X_i) \big] \Big) Y_{0i} \Big\}
\end{align*}
and
\begin{align*}
	\Bnl = \frac{1}{n} \sum_{i=1}^n \Big\{ & - \frac{1}{[\gamma_0^{D_1}(X_i)]^2} [\whatgamma_n^{1}(X_i) - \bargamma_n^{1}(X_i)] [\whatgamma_n^{D_1}(X_i) - \bargamma_n^{D_1}(X_i)] Y_i D_i \\
		& + \frac{\gamma_0^{1}(X_i)}{[\gamma_0^{D_1}(X_i)]^3} [\whatgamma_n^{D_1}(X_i) - \bargamma_n^{D_1}(X_i)]^2 Y_i D_i \\
		& + \frac{1}{[\gamma_0^{D_0}(X_i)]^2} [\whatgamma_n^{1}(X_i) - \bargamma_n^{1}(X_i)] [\whatgamma_n^{D_0}(X_i) - \bargamma_n^{D_0}(X_i)] Y_i (1-D_i) \\
		& - \frac{\gamma_0^{1}(X_i)}{[\gamma_0^{D_0}(X_i)]^3} [\whatgamma_n^{D_0}(X_i) - \bargamma_n^{D_0}(X_i)]^2 Y_i (1-D_i) \Big\} \\
	= \frac{1}{n} \sum_{i=1}^n \Big\{ & - \frac{1}{[\gamma_0^{D_1}(X_i)]^2} \Big( [\whatgamma_n^{1}(X_i) - \bargamma_n^{1}(X_i)] - \frac{\gamma_0^{1}(X_i)}{\gamma_0^{D_1}(X_i)} [\whatgamma_n^{D_1}(X_i) - \bargamma_n^{D_1}(X_i)] \Big) \\
		& \quad \times [\whatgamma_n^{D_1}(X_i) - \bargamma_n^{D_1}(X_i)] \times Y_i D_i\\
	& + \frac{1}{[\gamma_0^{D_0}(X_i)]^2} \Big( [\whatgamma_n^{1}(X_i) - \bargamma_n^{1}(X_i)] - \frac{\gamma_0^{1}(X_i)}{\gamma_0^{D_0}(X_i)} [\whatgamma_n^{D_0}(X_i) - \bargamma_n^{D_0}(X_i)] \Big) \\
		& \quad \times [\whatgamma_n^{D_0}(X_i) - \bargamma_n^{D_0}(X_i)] \times Y_i (1-D_i) \Big\}.
\end{align*}

Define
\begin{gather*}
	\psi^1(z_i, z_j) = K_h(X_i-X_j), \quad \psi^{D_1}(z_i, z_j) = K_h(X_i-X_j) D_j, \\
	\psi^{D_0}(z_i, z_j) = K_h(X_i-X_j) (1-D_j) \quad \psi = (\psi^1, \psi^{D_1}, \psi^{D_0})^\intercal, \\
	g_{\gamma\gamma}^{\prime\prime}(z_i) = \left( \begin{matrix}
			0 &  -\frac{Y_{1i}}{[\gamma^{D_1}(X_i)]^2} & -\frac{Y_{0i}}{[\gamma^{D_0}(X_i)]^2} \\
			-\frac{Y_{1i}}{[\gamma^{D_1}(X_i)]^2} & 2\frac{\gamma^1(X_i)Y_{1i}}{[\gamma^{D_1}(X_i)]^3} & 0 \\
			-\frac{Y_{0i}}{[\gamma^{D_0}(X_i)]^2} & 0 & 2\frac{\gamma^1(X_i)Y_{0i}}{[\gamma^{D_0}(X_i)]^3} 
		\end{matrix} \right).
\end{gather*}
Let $\phi(z_i, z_j) = \psi(z_i, z_j) - \bbE[ \psi(z_i, z_j) | z_i ]$ (note that $\bbE[ \psi(z_i, z_j) | z_i ] = \bbE[ \psi(z_i, z_j) | X_i ]$). Then we have
\begin{align*}
	g_{\gamma\gamma}^{\prime\prime}(z_i, \theta_0, \gamma_0, \phi(z_i, z_j), \phi(z_i, z_j)) = \phi(z_i, z_j)^\intercal g_{\gamma\gamma}^{\prime\prime}(z_i) \phi(z_i, z_j) = \phi(z_i, z_j)^{\otimes 2} \, \vect(g_{\gamma\gamma}^{\prime\prime}(z_i) ),
\end{align*}
where $\phi(z_i, z_j)^{\otimes 2} = \phi(z_i, z_j) \otimes \phi(z_i, z_j)$.

Since $D$ can only be 0 or 1, we can get 
\begin{gather*}
	[ \psi^{D_1}(z_i, z_j) ]^2 + [ \psi^{D_0}(z_i, z_j) ]^2 \equiv [ \psi^{1}(z_i, z_j) ]^2, \quad \psi^{D_1}(z_i, z_j) \psi^{D_0}(z_i, z_j) = 0, \\
	0\leq \psi^{D_1}(z_i, z_j) \psi^{1}(z_i, z_j) = [\psi^{D_1}(z_i, z_j)]^2  \leq [\psi^{1}(z_i, z_j)]^2, \\
	0\leq \psi^{D_0}(z_i, z_j) \psi^{1}(z_i, z_j) = [\psi^{D_0}(z_i, z_j)]^2 \leq [\psi^{1}(z_i, z_j)]^2.
\end{gather*}
Hence, we have $c [ \psi^{1}(z_i, z_j) ]^2 \leq \| \psi(z_i, z_j) \|^2 \leq C [ \psi^{1}(z_i, z_j) ]^2$ for some constant $c$ and $C$. It is easy to derive that
\begin{gather*}
	\bbE \big( [ \psi^1(z_i, z_j) ]^2 | z_i \big) = \bbE[ K_h(X_i-X_j)^2 | X_i ] = \frac{1}{h^d} \int_{\bbR} K(u)^2 \gamma_0^1(X_i - hu) du = \OP\Big( \frac{1}{h^d} \Big), \\
	\bbE\big( \bbE[\psi^1(z_i, z_j)^2 | z_i]^2 \big) = O\Big( \frac{1}{h^{2d}} \Big) \quad \bbE\big( \bbE[\psi^1(z_i, z_j)^2 | z_j]^2 \big) = O\Big( \frac{1}{h^{2d}} \Big), \\
	\bbE\big( \| \bbE[\psi(z_i, z_j) \otimes \psi(z_i,z_l) | z_j, z_l ] \|^2 \big) \leq C \, \bbE\big( \| \bbE[\psi^1(z_i, z_j) \psi^1(z_i,z_l) | z_j, z_l ] \|^2 \big) = O\Big( \frac{1}{h^d} \Big).
\end{gather*}

On the other hand, for example, we have
\begin{align*}
	\| \bbE[ \psi(z_i, z_j) | z_i ] \|^2  = \OP(1).
\end{align*}
Since $h\rightarrow 0$, this term is much smaller than the above ones in the limit.

\begin{lemma}
(i) Suppose the following two terms
\begin{align*}
	h^{2d} \, \bbE\Big\{ \Big[ \frac{Y_{1i}^2}{[\gamma_0^{D_1}(X_i)]^4} \Big( \frac{1-\bm{e}(X_i)}{\bm{e}(X_i)} \Big)^2 +\frac{Y_{0i}^2}{[\gamma_0^{D_0}(X_i)]^4} \Big( \frac{\bm{e}(X_i)}{1-\bm{e}(X_i)} \Big)^2 \Big] \, \bbE[K_h(X_i-X_j)^2 | X_i ]^2 \Big\}, \\
	h^{2d} \, \bbE\Big\{ \bbE\Big[ \Big( \frac{|Y_{1i} D_{1j}|}{[\gamma_0^{D_1}(X_i)]^2} \Big( \frac{1-\bm{e}(X_i)}{\bm{e}(X_i)} \Big) +\frac{|Y_{0i} D_{0j}|}{[\gamma_0^{D_0}(X_i)]^2} \Big( \frac{\bm{e}(X_i)}{1-\bm{e}(X_i)} \Big) \Big) \, K_h(X_i-X_j)^2  \, \big| \, X_j \Big]^2 \Big\},
\end{align*}
are of order $o(n^2)$ and the following one
\begin{align*}
	h^d \,\bbE\Big\{ & \bbE\Big[ \Big( \frac{Y_{1i}}{[\gamma_0^{D_1}(X_i)]^2}  \big(  D_{1j} D_{1l} - \frac{1}{2\bm{e}(X_i)} [ D_{1j} + D_{1l} ] \big) + \frac{Y_{0i}}{[\gamma_0^{D_0}(X_i)]^2} \big(  D_{0j} D_{0j}  \\
					& \quad - \frac{1}{2[1-\bm{e}(X_i)]} [ D_{0j} + D_{0l} ] \big) \Big) K_h(X_i - X_j) K_h(X_i - X_l) \, \Big| \, z_j, z_l \Big]^2 \Big\}
\end{align*}
is of order $o(n^2)$. If we choose $h^d \, \propto \, n^{2r-1}$ with $r>0$,  then we have
\begin{align*}
	\calB^{\texttt{NL}} = \bbE[ \Bnl ] = O(n^{2r}) \quad \text{and} \quad \| \Bnl - \calB^{\texttt{NL}} \| = \oP(n^{-1/2}).
\end{align*}

(ii) Define
\begin{align*}
	\xi(1)_i &= \frac{Y_{1i}}{\gamma_0^{D_1}(X_i)} \Big( \big[ \bargamma_n^{1}(X_i) - \gamma_0^{1}(X_i) \big] - \frac{1}{\bm{e}(X_i)} \big[ \bargamma_n^{D_1}(X_i) - \gamma_0^{D_1}(X_i) \big] \Big), \\
	\xi(0)_i &=  \frac{Y_{0i}}{\gamma_0^{D_0}(X_i)} \Big( \big[ \bargamma_n^{1}(X_i) - \gamma_0^{1}(X_i) \big] - \frac{1}{1-\bm{e}(X_i)} \big[ \bargamma_n^{D_0}(X_i) - \gamma_0^{D_0}(X_i) \big] \Big).
\end{align*}
If $\bbE[\xi(1)_i + \xi(0)_i] = O(h^m)$ and $\Var[ \xi(1)_i + \xi(0)_i ] = o(1)$, then
\begin{align*}
	\calB^{\texttt{ANB}} = \bbE[ \xi(1)_i + \xi(0)_i ] = O(h^s) \quad \text{and} \quad  \| \Banb - \calB^{\texttt{ANB}} \| = \oP(n^{-1/2}),
\end{align*}
where $s=m(2r-1)/d$.
\end{lemma}

In Assumption 2, \cite{Hirano&Imbens&Ridder:2003} assume the support of $X$ is compact and the density of $X$ is bounded and bounded away from 0. In Assumption 4, the authors assume that the selection probability $e(x)$ satisfies $0<\underline{p} \leq e(x) \leq \bar{p} <1$. Under these conditions, the assumptions of the above lemma are all satisfied.

\begin{proof}

(i) First, we can derive that
\begin{align*}
	& \big\| \bbE[ g_{\gamma\gamma}^{\prime\prime}(z_i, \theta_0, \gamma_0, \phi(z_i, z_j), \phi(z_i, z_j))  ] \big\| \\
	=\,& \big\| \bbE\big( g_{\gamma\gamma}^{\prime\prime}(z_i, \theta_0, \gamma_0, \phi(z_i, z_j), \phi(z_i, z_j)) | z_i ] \big) \big\|
	= \big\| \bbE\big( \bbE[ \phi(z_i, z_j)^{\otimes 2} | z_i ] \, \vect(g_{\gamma\gamma}^{\prime\prime}(z_i) ) \big) \big\| \\
	=\,& 2 \, \Big\| \bbE\Big\{ \frac{-Y_{1i}}{[\gamma_0^{D_1}(X_i)]^2} \Big( \bbE[\phi^{D_1}(z_i, z_j)^2 | z_i ] - \frac{\gamma_0^{1}(X_i)}{\gamma_0^{D_1}(X_i)} \bbE[\phi^{D_1}(z_i, z_j) \phi^1(z_i, z_j) | z_i ]  \Big) \\
	& \quad + \frac{Y_{0i}}{[\gamma_0^{D_0}(X_i)]^2} \Big( \bbE[\phi^{D_0}(z_i, z_j)^2 | z_i ] - \frac{\gamma_0^{1}(X_i)}{\gamma_0^{D_0}(X_i)} \bbE[\phi^{D_0}(z_i, z_j) \phi^1(z_i, z_j) | z_i ]  \Big) \Big\} \Big\| \\
	=\,& 2 \, \Big\| \bbE\Big\{ \frac{-Y_{1i}}{[\gamma_0^{D_1}(X_i)]^2} \Big( \bbE[\phi^{D_1}(z_i, z_j)^2 | z_i ] - \frac{1}{\bm{e}(X_i)} \bbE[\phi^{D_1}(z_i, z_j) \phi^1(z_i, z_j) | z_i ]  \Big) \\
		& \quad + \frac{Y_{0i}}{[\gamma_0^{D_0}(X_i)]^2} \Big( \bbE[\phi^{D_0}(z_i, z_j)^2 | z_i ] - \frac{1}{1-\bm{e}(X_i)} \bbE[\phi^{D_0}(z_i, z_j) \phi^1(z_i, z_j) | z_i ]  \Big) \Big\} \Big\| \\
	\leq\,& C \, \bbE\Big\{ \frac{|Y_{1i}|}{[\gamma_0^{D_1}(X_i)]^2} \Big( \frac{1}{\bm{e}(X_i)} -1\Big) \bbE[\psi^{D_1}(z_i, z_j)^2 | z_i ] \\
			& \quad + \frac{|Y_{0i}|}{[\gamma_0^{D_0}(X_i)]^2} \Big( \frac{1}{1-\bm{e}(X_i)} - 1 \Big) \bbE[\psi^{D_0}(z_i, z_j)^2 | z_i ] \Big\} \\
	\leq\,& C \, \bbE\Big\{ \Big[ \frac{|Y_{1i}|}{[\gamma_0^{D_1}(X_i)]^2} \frac{1-\bm{e}(X_i)}{\bm{e}(X_i)} + \frac{|Y_{0i}|}{[\gamma_0^{D_0}(X_i)]^2} \frac{\bm{e}(X_i)}{1-\bm{e}(X_i)} \Big] \, \bbE\big( [ \psi^1(X_i, X_j) ]^2 | X_i \big)  \Big\}
\end{align*}
Therefore, if 
\begin{align*}
	& \bbE\Big\{ \Big[ \frac{|Y_{1i}|}{[\gamma_0^{D_1}(X_i)]^2} \frac{1-\bm{e}(X_i)}{\bm{e}(X_i)} + \frac{|Y_{0i}|}{[\gamma_0^{D_0}(X_i)]^2} \frac{\bm{e}(X_i)}{1-\bm{e}(X_i)} \Big] \, \bbE\big( [ \psi^1(X_i, X_j) ]^2 | X_i \big)  \Big\} \\
	\leq\,& \bbE\big\{ \bbE\big( [ \psi^1(X_i, X_j) ]^2 | X_i \big) \big\} = \bbE\big( [ \psi^1(X_i, X_j) ]^2 \big),
\end{align*}
then we will have
\begin{align*}
	\calB^{\texttt{NL}} = \frac{1}{2(n-1)} \bbE[ g_{\gamma\gamma}^{\prime\prime}(z_i, \theta_0, \gamma_0, \hatgamma_n - \bargamma_n, \hatgamma_n - \bargamma_n)  ] = O(n^{2r}).
\end{align*}

Next, we can show that
\begin{align*}
	& \Var\big( \bbE[g_{\gamma\gamma}^{\prime\prime}(z_i, \theta_0, \gamma_0, \phi(z_i, z_j), \phi(z_i, z_j)) \,|\, z_i ] \big) \\
	\leq\,& \bbE\big( \bbE[g_{\gamma\gamma}^{\prime\prime}(z_i, \theta_0, \gamma_0, \phi(z_i, z_j), \phi(z_i, z_j)) \,|\, z_i ]^2 \big) \\
	\leq\,& C \, \bbE\big( \bbE[g_{\gamma\gamma}^{\prime\prime}(z_i, \theta_0, \gamma_0, \psi(z_i, z_j), \psi(z_i, z_j)) \,|\, z_i ]^2 \big) \\
	\leq\,& C \, \bbE\Big\{ \Big[ \frac{Y_{1i}^2}{[\gamma_0^{D_1}(X_i)]^4} \Big( \frac{1-\bm{e}(X_i)}{\bm{e}(X_i)} \Big)^2 +\frac{Y_{0i}^2}{[\gamma_0^{D_0}(X_i)]^4} \Big( \frac{\bm{e}(X_i)}{1-\bm{e}(X_i)} \Big)^2 \Big] \, \bbE[K_h(X_i-X_j)^2 | X_i ]^2 \Big\}
\end{align*}
and
\begin{align*}
	& \Var\big( \bbE[g_{\gamma\gamma}^{\prime\prime}(z_i, \theta_0, \gamma_0, \phi(z_i, z_j), \phi(z_i, z_j)) \,|\, z_j ] \big) \\
	\leq\,& \bbE\big( \bbE[g_{\gamma\gamma}^{\prime\prime}(z_i, \theta_0, \gamma_0, \phi(z_i, z_j), \phi(z_i, z_j)) \,|\, z_j ]^2 \big) \\
	\leq\,& C \, \bbE\big( \bbE[g_{\gamma\gamma}^{\prime\prime}(z_i, \theta_0, \gamma_0, \psi(z_i, z_j), \psi(z_i, z_j)) \,|\, z_j ]^2 \big) \\
	\leq\,& C \, \bbE\Big\{ \bbE\Big[ \Big( \frac{|Y_{1i} D_{1j}|}{[\gamma_0^{D_1}(X_i)]^2} \Big( \frac{1-\bm{e}(X_i)}{\bm{e}(X_i)} \Big) +\frac{|Y_{0i} D_{0j}|}{[\gamma_0^{D_0}(X_i)]^2} \Big( \frac{\bm{e}(X_i)}{1-\bm{e}(X_i)} \Big) \Big) \, K_h(X_i-X_j)^2  \, \big| \, X_j \Big]^2 \Big\}.
\end{align*}
Hence, if the above two right-hand-side terms are of order $o(n^2)$,

Finally, note that
\begin{align*}
	& \Var\big( \bbE[g_{\gamma\gamma}^{\prime\prime}(z_i, \theta_0, \gamma_0, \phi(z_i, z_j), \phi(z_i, z_l)) \,|\, z_j, z_l ] \big) \\
	\leq\,& \bbE\big( \bbE[g_{\gamma\gamma}^{\prime\prime}(z_i, \theta_0, \gamma_0, \phi(z_i, z_j), \phi(z_i, z_l)) \,|\, z_j, z_l ]^2 \big) \\
	\leq\,& C \, \bbE\big( \bbE[g_{\gamma\gamma}^{\prime\prime}(z_i, \theta_0, \gamma_0, \psi(z_i, z_j), \psi(z_i, z_l)) \,|\, z_j, z_l ]^2 \big) \\
	\leq\,& C \, \bbE\Big\{ \bbE\Big[ \frac{-Y_{1i}}{[\gamma_0^{D_1}(X_i)]^2} \Big( \psi^{D_1}(z_i, z_j) \psi^{D_1}(z_i, z_l) - \frac{1}{2\bm{e}(X_i)} [ \psi^{D_1}(z_i, z_j)  \psi^1(z_i, z_l) \\
		&\quad\quad + \psi^{D_1}(z_i, z_l)  \psi^1(z_i, z_j) ]  \Big) \\
			& \quad + \frac{Y_{0i}}{[\gamma_0^{D_0}(X_i)]^2} \Big( \psi^{D_0}(z_i, z_j) \psi^{D_0}(z_i, z_l) - \frac{1}{2[1-\bm{e}(X_i)]} [ \psi^{D_0}(z_i, z_j)  \psi^1(z_i, z_l) \\
		&\quad\quad + \psi^{D_0}(z_i, z_l)  \psi^1(z_i, z_j) ]  \Big) \, \Big| \, z_j, z_l  \Big]^2 \Big\} \\
	= \,& C \, \bbE\Big\{ \bbE\Big[ \Big( \frac{Y_{1i}}{[\gamma_0^{D_1}(X_i)]^2}  \big(  D_{1j} D_{1l} - \frac{1}{2\bm{e}(X_i)} [ D_{1j} + D_{1l} ] \big) + \frac{Y_{0i}}{[\gamma_0^{D_0}(X_i)]^2} \big(  D_{0j} D_{0j}  \\
				& \quad - \frac{1}{2[1-\bm{e}(X_i)]} [ D_{0j} + D_{0l} ] \big) \Big) K_h(X_i - X_j) K_h(X_i - X_l) \, \Big| \, z_j, z_l \Big]^2 \Big\}.
\end{align*}
Then the conclusions readily follows the proof of Lemma \ref{lem:BO}.

(ii) Since $\{\xi(1)_i + \xi(0)_i\}_{i=1}^n$ is an \text{i.i.d.} sequence, the proof is essentially the same as illustrated in the main-text. 

\end{proof}

\end{appendices}

\end{document}